\documentclass[12pt,notitlepage]{amsart}
\usepackage{latexsym,amsfonts,amssymb,amsmath,amsthm}
\usepackage{color}
\usepackage[inner=1.0in,outer=1.0in,bottom=1.0in, top=1.0in]{geometry}
\usepackage{tikz}

\usepackage{hyperref}

\newcommand{\M}{\ensuremath{\mathrm{M}}}
\newcommand{\Sp}{\ensuremath{\mathrm{Sp}}}
\newcommand{\SL}{\ensuremath{\mathrm{SL}}}
\newcommand{\GL}{\ensuremath{\mathrm{GL}}}
\newcommand{\mz}{\ensuremath{\mathbb Z}}
\newcommand{\mr}{\ensuremath{\mathbb R}}
\newcommand{\mh}{\ensuremath{\mathbb H}}
\newcommand{\mq}{\ensuremath{\mathbb Q}}
\newcommand{\mc}{\ensuremath{\mathbb C}}

\newcommand{\shortmod}{\ensuremath{\negthickspace \negthickspace \negthickspace \pmod}}

\newcommand{\intR}{\int_{-\infty}^{\infty}}

\newcommand{\sumstar}{\sideset{}{^*}\sum}

\newcommand{\sumdag}{\sideset{}{^\dagger}\sum}

\theoremstyle{plain}		
	\newtheorem{mytheo}{Theorem}[section]
	
	\newtheorem{myprop}[mytheo]{Proposition}
	\newtheorem{mycoro}[mytheo]{Corollary}
     \newtheorem{mylemma}[mytheo]{Lemma}

\theoremstyle{remark}

\numberwithin{equation}{section}
\begin{document}
\author{Wing Hong Leung}
 \email{joseph.leung@rutgers.edu}
 \author{Matthew P. Young} 
  \email{mpy4@rutgers.edu}
 \address{Department of Mathematics \\
 	  Rutgers University \\
 	  Piscataway \\
 	  NJ 08854 \\
 		U.S.A.}		
 \thanks{This material is based upon work supported by the National Science Foundation under agreement No. DMS-2302210 (M.Y.).  Any opinions, findings and conclusions or recommendations expressed in this material are those of the authors and do not necessarily reflect the views of the National Science Foundation.}

 \begin{abstract}
 We provide a power-saving bound for certain smoothed shifted convolution sums for Fourier coefficients of Siegel cusp forms.  
 This result is the first nontrivial estimate for a shifted convolution sum with two cusp forms on a group of higher rank than $\GL_2$.
 Our approach is based on a novel automorphic reinterpretation of the delta method of Duke, Friedlander, and Iwaniec.  The method reduces the problem to the estimation of Fourier coefficients of Siegel Poincare series, which is ultimately based on the Weil bound.
\end{abstract}
%\today
\title[The shifted convolution problem for Siegel modular forms]{The shifted convolution problem for Fourier coefficients of Siegel modular forms of degree $2$}
\maketitle

\section{Introduction}
\subsection{Motivation}
The shifted convolution problem for a sequence $a_n$ is to prove cancellation (or develop a main term with a good error term) in sums of the form $\sum_{n \leq X} a_n a_{n+h}$, where $h \neq 0$.  Such sums arise in several contexts.  For instance, a $t$-aspect second moment of an $L$-function directly leads to such sums.  In this case, the shift $h$ should be small compared to $X$.  In a different, but closely related, direction, the second moment of Dirichlet twists of prime conductor $p$ of an $L$-function gives rise to such sums where $h \equiv 0 \pmod{p}$.  More generally, the van der Corput method reduces cancellation in the linear sum $\sum_n a_n$ to that in the correlated sums $\sum_n a_n a_{n+h}$ (see \cite[Lemma 8.17]{IwaniecKowalski}, for instance, for such a result).

One of the most intensively-studied shifted convolution problems is that generated by the ordinary divisor function, $\sum_{n \leq X} d(n) d(n+h)$.  The progress in understanding this sum is responsible for the best error term in the fourth moment of the Riemann zeta function.  See \cite{MotohashiBinary} for a comprehensive discussion of this problem.
A cusp form analog of this problem is to prove cancellation in $\sum_{n \leq X} \lambda(n) \lambda(n+h)$, where $\lambda(n)$ are the Fourier coefficients of a Hecke cusp form on $\GL_2$.  
There are various methods of solving this well-studied problem;
for instance,
see \cite{BlomerHarcos} for some advanced techniques utilizing the spectral theory of automorphic forms.  We also refer to \cite{Michel} for more discussions on the connections between the subconvexity problem for $L$-functions in families and the shifted convolution problem.

There are some solutions to higher rank variants of the shifted convolution problem for two different cusp forms, one on $\GL_m$ and one on $\GL_2$, as in \cite{Pitt, Munshi}, for $m=3$.
Nevertheless,
the shifted convolution problem has not been solved for two cusp forms on $\GL_m$ with $m \geq 3$, and this remains a major barrier for progress on moment problems for higher degree $L$-functions, including the sixth moment of the Riemann zeta function.
The main result in this paper (Theorem \ref{thm:mainthm} below) gives the first (unconditional) higher rank shifted convolution sum bound for two holomorphic cusp forms on $\Sp_4(\mz)$.  In Section \ref{section:statement} we give a precise statement of our result, and in Section \ref{section:broadoverview}, we give an overview of the ideas used in the proof.

\subsection{Statement of result}
\label{section:statement}
Let $F$ be a holomorphic Siegel cusp form of weight $k$ on $\Gamma = \mathrm{Sp}_4(\mz)$, and 
let $\mh_2 = \{Z=X+iY \in \mathrm{M}_2(\mc) : Z^t =  Z,  Y > 0 \}$.
The group $\mathrm{Sp}_4(\mr)$ acts on $\mh_2$ as follows.  Let $\gamma = (\begin{smallmatrix} A & B \\ C & D \end{smallmatrix}) \in \mathrm{Sp}_4(\mr)$, and for $Z \in \mh_2$ let $\gamma Z = \gamma(Z) = (AZ+B)(CZ+D)^{-1}$.  The automorphy condition of $F$ is that $F(\gamma Z) = \det(CZ+D)^k F(Z)$ for all $\gamma \in \Gamma = \mathrm{Sp}_4(\mz)$.
The Fourier expansion for $F$ takes the form
\begin{equation}
\label{eq:SiegelFourierExpansion}
F(Z) = \sum_{M \in \Lambda^{+}} a_F(M) e(\mathrm{Tr}(MZ)),
\end{equation}
where
$\Lambda = \{ (\begin{smallmatrix} a & b/2 \\ b/2 & c \end{smallmatrix}) : a,b,c \in \mz \}$
and where $\Lambda^{+} = \{ M \in \Lambda: M > 0 \}$.
Let $\widetilde{a_F}(M) = \frac{a_F(M)}{\det(M)^{\frac{k}{2} - \frac{3}{4}}}$, which are well-normalized Fourier coefficients.
Indeed, if $\det(T)$ is a fundamental discriminant, then a conjecture of Resnikoff and Salda\~{n}a \cite{Resnikoff1974} predicts that $|\widetilde{a_F}(T)| \ll |\det T|^{\varepsilon}$.   Although this conjecture is wide open, and apparently goes beyond the generalized Lindel\"{o}f Hypothesis, it is known in an average sense for which 
see Section \ref{section:RankinSelbergBound} below.  

To state our main theorem, 
we need to define a certain parameterized family of Poincare series. 
Let $\varphi: \mr \times \mr_{>0} \times \mr_{>0} \rightarrow \mc$ be smooth of compact support.
For each real $N \geq 1$, 
let $\phi = \phi_N$ be the smooth function on $\mathrm{GL}_2^{+}(\mr)/\mathrm{SO}_2(\mr)$ of the form
\begin{equation}
\label{eq:phiNdef}
    \phi\Big( \begin{pmatrix} 1 & u \\ & 1 \end{pmatrix} \begin{pmatrix} \sqrt{r_1} & \\ & \sqrt{r_2} \end{pmatrix} k \Big)
    = \varphi(u, N r_1, N r_2),
\end{equation}
where $u \in \mr$, $r_1, r_2 > 0$, and $k \in \mathrm{SO}_2(\mr)$.  Let $\iota: \mathrm{GL}_2(\mr)^{+}/\mathrm{SO}_2(\mr) \xrightarrow{\sim} \mathrm{M}_2(\mr)^{\mathrm{sym}, >0}$ be defined by $\iota(g) = g g^t$, which is a smooth bijection.
Let $Q \in \Lambda$, and define the Poincare series
\begin{equation}
\label{eq:PoincareSeriesDefinition}
P_{Q}(Z, \phi) =
\sum_{\gamma \in \Gamma_{\infty} \backslash \Gamma} 
e(\mathrm{Tr}(Q \mathrm{Re}(\gamma Z))) \phi(\iota^{-1} \mathrm{Im}(\gamma Z)),
\end{equation}
where $\Gamma_{\infty} = \{ (\begin{smallmatrix} I & X \\ & I \end{smallmatrix}) \in \Gamma \}$.  
The Petersson inner product on $S_k(\mathrm{Sp}_4(\mz))$ (the space of Siegel cusp forms of weight $k$) is given by
\begin{equation}
\label{eq:PeterssonInnerProductDefinitionSiegel}
\langle F, G \rangle 
= \int_{\Gamma \backslash \mh_2} F(Z) \overline{G}(Z) (\det Y)^k d\mu,
\qquad
d\mu = (\det Y)^{-3} dX dY.
\end{equation}
Here $dX=dx_1dx_2dx_3$ for $X=(\begin{smallmatrix}
    x_1 & x_2\\ x_2 & x_3
\end{smallmatrix})$ is the standard Lebesgue measure, and similarly for $dY$. By a standard unfolding calculation, we have
\begin{equation*}
    \langle F P_Q(\cdot, \phi), F \rangle = \sum_{\substack{M_1, M_2 \in \Lambda^{+} \\ M_1 + Q = M_2}}
    \widetilde{a_F}(M_1) \overline{\widetilde{a_F}(M_2)} W_{N}(M_1, M_2, \varphi),
\end{equation*}
where $W_{N}$ is  the Laplace transform of $\phi(Y)\det(Y)^{k-3}$ on $\mathrm{M}_2(\mr)^{\mathrm{sym}, >0}$ (e.g. see \cite[Sect. 7.2.3]{muirhead1982aspects} for the definition of the Laplace transform on this space). Some analysis (see Proposition \ref{prop:PoincareSeriesIntegralEstimates} below) shows that for fixed $Q$, $W_{N}$ is approximately supported on (symmetric) matrices $M_1$ and $M_2$ with entries of size $\ll N$. Hence, the ``trivial" bound on this inner product is $|\langle F P_Q(\cdot, \phi), F \rangle| \ll_{Q,F, \varepsilon} N^{3+\varepsilon}$.  The main theorem of this article is the following.
\begin{mytheo}
\label{thm:mainthm}
Let notation be as above.  Suppose $Q \neq 0$.  Then as $N \rightarrow \infty$,
\begin{equation}
    |\langle F P_Q(\cdot, \phi), F \rangle| \ll_{Q,F, \varphi, \varepsilon} N^{5/2+\varepsilon}.
\end{equation}   
\end{mytheo}
In particular, Theorem \ref{thm:mainthm} gives a power-saving bound for the shifted convolution sum whose test function arises from this Poincare series.

J\"{a}\"{a}saari, Lester, and Saha have obtained a QUE result for holomorphic Siegel cusp forms on $\Sp_4(\mz)$ in the large weight limit, conditional on GRH.  Their result relies on bounds for shifted sums, but they apply the triangle inequality and give up cancellation in the shifted convolution problem, so their method is very different from ours.  Another point of contrast is that they consider the difficult case where the underlying cusp form $F$ varies, whereas $F$ in Theorem \ref{thm:mainthm} is fixed.

\subsection{Broad overview of the proof}
\label{section:broadoverview}
The first key step in the proof is a novel reinterpretation of the DFI delta method in terms of automorphic forms.  To elaborate on this, we  recall the construction of \cite{DFI}.  Their idea is based on the simple observation that every nonzero integer has its divisors appearing in pairs, while the integer $0$ is divisible by every integer.  Let $\omega: \mr \rightarrow \mc$ be a smooth, even, function with compact support, satisfying $\omega(0) = 0$ and $\sum_{n \in \mz} \omega(n) = 1$.  Then
\begin{equation}
\label{eq:deltamethodDFI}
\delta(n=0) = 2 \sum_{\substack{1 \leq d \mid n }} (\omega(d) - \omega(n/d)).
\end{equation}
On the other hand, one can detect the divisibility $d | n$ using additive characters, giving
\begin{equation*}
    \delta(n=0) = 2\sum_{d \geq 1} \frac{1}{d} \sum_{h \shortmod{d}} e_d(hn) (\omega(d) - \omega(n/d)).
\end{equation*}
Letting $k = (h,d)$ and changing variables $d = ck$, we obtain
\begin{equation}
\label{eq:deltamethodDFIadditivechars}
    \delta(n=0) = \sum_{c \geq 1} \frac{S(n,0;c)}{c} \sum_{k \geq 1} \frac{2}{k} (\omega(ck) - \omega(\frac{n}{ck})).
\end{equation}

One might recognize that the previous process has some similarities to the calculation of the Fourier expansion of Eisenstein series.  Indeed, the non-holomorphic Eisenstein series $E(z,s)$ has Fourier coefficients that are a priori given as sums of Ramanujan sums, but which can be calculated in ``closed form" as a divisor sum.  This is similar to the process of derivation of the delta symbol, but in reverse order.  

A precise way to link these two observations is to use a construction of Nelson \cite[Thm.\ 5.6]{Nelson} which expresses the constant function $1$ as an incomplete Eisenstein series.  
It is easy to describe the construction.  For $g:(0,\infty) \rightarrow \mc$ smooth with rapid decay at $0$, bounded at $\infty$, define the incomplete Eisenstein series
 $E(z, g) = \sum_{\gamma \in \Gamma_{\infty} \backslash \Gamma} g(\mathrm{Im}(\gamma z))$.
By Mellin inversion, 
\begin{equation*}
E(z, g) = \frac{1}{2 \pi i} \int_{(2)} \widetilde{g}(-s) E(z, s) ds.
\end{equation*}
Nelson chooses $\widetilde{g}(-s) = (H(s) - H(1-s)) \zeta^*(2s)$ where $H$ is a nice function satisfying $H(1) \neq H(0)$.  By shifting contours to the line $1/2$, we see that $E(z,g) = \zeta^*(2) (H(1) - H(0)) \frac{3}{\pi}$, since the integral along the line $1/2$ vanishes by the functional equation of the Eisenstein series, and by the fact that $H(s) - H(1-s)$ is odd under $s \rightarrow 1-s$.  Scaling $H$ appropriately gives that $E(z,g) = 1$.  This choice of $g$ is smooth on the positive reals, and is defined by
\begin{equation*}
    g(y) = \frac{1}{2 \pi i} \int_{(2)} (H(s) - H(1-s)) \zeta^*(2s) y^s ds.
\end{equation*}
Shifting contours either right or left shows $g(y) \ll y^A$ for $y \rightarrow 0^{+}$ and $g(y) = 1 + O(y^{-A})$ for $y \rightarrow \infty$, so this is indeed an allowable test function.

The fact that $E(z,g)$ is identically $1$ means that its $n$th Fourier coefficient is $\delta(n=0)$.  We can alternatively write the Fourier coefficients of $\zeta^*(2s) E(z,s)$ in terms of divisor functions, and derive a formula that closely resembles \eqref{eq:deltamethodDFI}.  That is, we have
\begin{equation}
\label{eq:deltaViaNelson}
\delta(n=0) = \frac{1}{2 \pi i} \int_{(2)} (H(s) - H(1-s)) (\text{$n$th Fourier coefficient of $\zeta^*(2s) E(z,s)$}) ds.
\end{equation}

Theorem \ref{thm:Einc=1} below gives a natural generalization of Nelson's construction to Eisenstein series on $\mathrm{Sp}_4$.  The proof is more difficult because the $\mathrm{Sp}_4$ Eisenstein series now depends on two complex variables, and has twelve polar lines (see Theorem \ref{thm:EisensteinSeriesMinimalProperties} for more details).

One of the earliest applications of the DFI delta method is a solution to the shifted convolution problem on $\GL_2$, which we describe now.  Let $f$ be a holomorphic cusp form of weight $k$ on $\SL_2(\mz)$, which we write as
\begin{equation*}
f(z) = \sum_{n \geq 1} \lambda_f(n) n^{\frac{k-1}{2}} e(nz).
\end{equation*}
Then \cite{DFI} showed for $q \geq 1$ that
\begin{equation}
\label{eq:DFIshiftedsumBound}
    \sum_{n} w(n/N) \lambda_f(n) \lambda_f(n+q) \ll_{q,f,w,\varepsilon} N^{3/4+\varepsilon}.
\end{equation}
Their method rewrites this shifted convolution sum as $\sum_{m,n} w(m/N) \lambda_f(m) \lambda_f(n) \delta(n=m+q)$, applies \eqref{eq:deltamethodDFIadditivechars}, and then uses the Voronoi formula for both the $m$- and $n$-sums.  The Ramanujan sum transforms into a Kloosterman sum, and the Weil bound completes the proof.

Now we reinterpret this proof using the Eisenstein series.  
We give a very brief sketch here, with more details in Section \ref{section:GL2model}.
We begin by expressing the shifted convolution sum as an inner product with Poincare series, namely $\langle f P_q, f \rangle$, where $P_q$ is a certain incomplete Poincare series.  This becomes a shifted sum after unfolding the integral with Poincare series.  To mimic the DFI delta method, we insert $1 = E(z, g)$ into the integral instead, and unfold with this incomplete Eisenstein series.  This leads to an alternative arrangement involving the integral of $|f|^2 P_q$ over the strip $x \pmod{1}$, $y \in \mr_{>0}$.  We apply the Fourier expansions of $f$ and $P_q$ and perform the $x$-integral.  The Fourier expansion of $P_q$ involves a sum of Kloosterman sums, and the Weil bound leads to a result equivalent to \eqref{eq:DFIshiftedsumBound}.
By the way, the Voronoi formula is not needed in this approach-- the analogous step is instead directly using the automorphy of $f$ which is crucially used in the unfolding process. 
We also mention that Holowinsky \cite{holowinsky2010sieving}
 implicitly used an approximate identity $E(z, g) \approx 1$ towards the QUE problem for $\GL_2$, which gives an alternative unfolding of the integral $\langle f P_q, f \rangle$ (see \cite[Prop 3.2]{holowinsky2010sieving} for this result).

Our strategy for Theorem \ref{thm:mainthm} is based on the above process.  First, we write the shifted convolution sum as an inner product with Poincare series (see Proposition \ref{prop:unfoldingBasicPoincare} below).  Inserting artificially $1$ as an incomplete Eisenstein series (now on $\mathrm{Sp}_4$) and unfolding with it, we obtain a similar integral over a $6$-dimensional strip.  Applying Fourier expansions, we can perform an integral over the torus (now $3$-dimensional), which leads to Proposition \ref{prop:ShiftedSumUnfoldedViaEisenstein} below.
Finally, we need to estimate the Fourier coefficients of the Poincare series.  The Fourier coefficients were essentially calculated by Kitaoka \cite{Kitaoka}, though he worked in the more restricted setting of holomorphic Siegel modular forms.  Nevertheless, by following his arguments, we obtain preliminary formulas for the Fourier coefficients in Propositions \ref{prop:rank2SumOfKloostermanSums}, \ref{prop.Rank1Poincare}, and \ref{prop.Rank0Poincare}.
However, what is needed here is to understand their behavior as a function of $N$.  This latter property has been glossed over so far in the $\GL_2$ toy model of the problem, but is brought to the forefront in the more detailed description in Section \ref{section:GL2model}.  
Bounding the Fourier coefficients requires both analytic and arithmetic inputs.  
We need to estimate certain integrals, and moreover, to use these estimates to constrain the moduli $C \in \mathrm{M}_2(\mz)$ in the sums of symplectic Kloosterman sums arising in Kitaoka's formulas.
A key arithmetic input is Kitaoka's bound on the symplectic Kloosterman sums (see Lemma \ref{lemma:KloostermanSumKitaokasBound} for the statement of his bound).  In the $\GL_2$ situation, one would encounter sums of the form $\sum_{c \leq X} |S(m,n;c)| \ll X^{\varepsilon} \sum_{c \leq X} c^{1/2} (m,n,c)^{1/2}$, using Weil's bound.  The factor $(m,n,c)^{1/2}$ here is innocuous, since one has standard bounds such as $\sum_{c \leq X} (c,q) \ll (Xq)^{\varepsilon}$, which follows from simply counting nonzero integers $m \leq X$ with $m \equiv 0 \pmod{d}$.  Kitaoka's bound also exhibits an extra gcd factor, though it is unfortunately much more difficult to control in our setting here.
The first problem with using Kitaoka's bound is that the gcd is with an integer (denoted $t$) which is the lower-right entry of a matrix $V^t T V$, where $V$ depends on $C$ in a somewhat indirect way.  In Lemmas \ref{lemma:UandVclassification} and \ref{lemma:CUVpq}, culminating with Corollary \ref{coro:tformula}, we work out a practically useful formula for this integer $t$. 
Aided with this formula for $t$, we then turn to estimating the sum of symplectic Kloosterman sums over the moduli $C$.  This turns into a fairly involved, but self-contained, counting problem, which is treated in Section \ref{section:countingproblems}.
Finally, we estimate the Fourier coefficients of the Poincare series with Theorem \ref{thm.Rank2Bound}, \ref{thm.Rank1Bound}, and Proposition \ref{prop.Rank0Poincare}.

\subsection{\texorpdfstring{$\GL_2$}{GL2} model of the proof}
\label{section:GL2model}
In this section, we give a more in-depth sketch of the $\GL_2$ version of our method, which may be of independent interest.  More importantly, this serves as a useful guide and model for the rest of the paper, since many of the aspects of the proof of Theorem \ref{thm:mainthm} are direct analogs of those sketched out here.

Let $\psi:(0,\infty) \rightarrow \mr$ be a smooth function with rapid decay at $0$, bounded at $\infty$, and for $q \in \mz$ define the Poincare series
\begin{equation*}
P_q(z, \psi) = \sum_{\gamma \in \Gamma_{\infty} \backslash \Gamma} e(q \mathrm{Re}(\gamma z)) \psi(\mathrm{Im}(\gamma z)),
\end{equation*}
which is an automorphic function on $\mathrm{SL}_2(\mz) \backslash \mh$.
For $f$ a cusp form on $\mathrm{SL}_2(\mz)$ of weight $k$, 
consider $I = \langle f P_q, f \rangle$. By unfolding, 
\begin{equation*}
I 
= \int_0^{\infty} \int_0^{1} y^k |f(z)|^2 e(qx) \psi(y) \frac{dx dy}{y^2}.
\end{equation*}
Applying the Fourier expansion of $f$, we get
\begin{align}
\label{eq:IformulaGL2model}
I = \sum_{m+q=n} \lambda_f(m) \overline{\lambda_f}(n) m^{\frac{k-1}{2}} n^{\frac{k-1}{2}}
\int_0^{\infty}  y^{k-1} \exp(-2 \pi (m+n) y) \psi(y) \frac{dy}{y}.
%\\
%= 
%\sum_{m-n=h} \lambda_f(m) \overline{\lambda_f}(n) \frac{m^{\frac{k-1}{2}} n^{\frac{k-1}{2}}}{(m+n)^{k-1}}
%\int_0^{\infty}  y^{k-1} \exp(-2 \pi  y) \psi(\frac{y}{m+n}) \frac{dy}{y}.
\end{align}
For $N \geq 1$ we choose $\psi$ to have support on $[1/N, 2/N]$, so that $I$ captures a shifted convolution sum with $m+n \lesssim N$.  Note the trivial bound on $I$ is $I \ll N^{1+ \varepsilon}$.

Next we approach $I$ using Nelson's construction, $1 = E(z,g)$ as in Section \ref{section:broadoverview}.  Unfolding with this Eisenstein series, we have
\begin{equation*}
I = \langle f E(\cdot, g) P_q, f \rangle 
= \int_0^{\infty} \int_0^{1} |f(z)|^2 P_q(z, \psi) g(y) y^k \frac{dx dy}{y^2}.
\end{equation*}
Next we apply the Fourier expansions of $f$ and $P_q$ and perform the $x$-integral.
Suppose $P_q$ has the Fourier expansion
\begin{equation*}
P_q(z, \psi) = \sum_{r \in \mz} a_{q}(r, y, \psi) e(rx).
\end{equation*}
Then
\begin{equation}
\label{eq:GL2sketchI}
I = 
\sum_{m+r=n} \lambda_f(m) m^{\frac{k-1}{2}} \overline{\lambda_f}(n) n^{\frac{k-1}{2}} 
\int_0^{\infty} g(y) y^{k-2} a_q(r,y,\psi) 
\exp(-2\pi(m+n) y)
dy.
\end{equation}
The decay of $g(y)$ near $y=0$ means the integral above is concentrated on $y \gg 1$.  On the other hand, the exponential decay effectively means that $y \ll 1$ as well (and $m$, $n$, $r \ll 1$ too).
The problem of estimating $I$ then boils down to estimating $a_q(r,y, \psi)$ with $y \asymp 1$, $r \ll 1$, and $\psi(v)$ supported on $v \asymp N^{-1}$.  

We now briefly compute $a_q$, which is a standard computation, but will nevertheless be useful as a reference and a guide for the Siegel case.  For more details, see 
\cite[Section 3.4]{IwaniecSpectral}.
We have
\begin{align*}
a_q(r, y, \psi) &= \int_0^{1} e(-rx) \sum_{\gamma \in \Gamma_{\infty} \backslash \Gamma} e(q \mathrm{Re}(\gamma z)) \psi(\mathrm{Im}(\gamma z)) dx
\\
&= 
 \sum_{(c,d) = 1}  
\int_0^{1} e(-rx) 
 e(q \mathrm{Re}(\frac{az+b}{cz+d})) \psi(\mathrm{Im}(\frac{az+b}{cz+d})) dx
=: \Sigma_{c=0} + \Sigma_{c \geq 1}.
\end{align*}
Note
$\Sigma_{c=0} = 
\int_0^{1} e(-rx) 
 e(q x) \psi(y) dx = \psi(y) \delta_{q=r}$.
The contribution of this term to \eqref{eq:GL2sketchI} is $O(N^{-100})$ by the rapid decay of $\psi(y)$ at $y=0$.

For $\Sigma_{c \geq 1}$, by a minor modification of \cite[(3.17)]{IwaniecSpectral}, we have
\begin{align*}
\Sigma_{c \geq 1}
% &=
% \sum_{c \geq 1} \sum_{(d,c) = 1} e_c(rd)
% \int_{d/c}^{1 + \frac{d}{c}} e(-rx)) 
%  e(q \mathrm{Re}(\frac{az+b - \frac{ad}{c}}{cz})) \psi(\mathrm{Im}(\frac{az+b - \frac{ad}{c}}{cz})) dx
%  \\
%  &=
%  \sum_{c \geq 1} 
% \sumstar_{d \shortmod{c}} e_c(rd) 
% \intR e(-rx)
%  e(q \mathrm{Re}(\frac{az-\frac{1}{c}}{cz})) \psi(\mathrm{Im}(\frac{az-\frac{1}{c}}{cz})) dx
%  \\
%  &= 
%   \sum_{c \geq 1} 
% \sumstar_{d \shortmod{c}} e_c(rd + q a) 
% \intR e(-rx)
%  e(q \mathrm{Re}(\frac{-1}{c^2z})) \psi(\mathrm{Im}(\frac{-1}{c^2z})) dx
%  \\
%  &= 
%   \sum_{c \geq 1} S(r,q;c)
% \intR e(-rx)
%  e(q \mathrm{Re}(\frac{-1}{c^2 z})) \psi(\mathrm{Im}(\frac{-1}{c^2 z})) dx
%  \\
%  &= \sum_{c \ge 1} S(r,q;c)
%  \intR e(-rx - \frac{qx}{c^2(x^2+y^2)}) 
%  \psi(\frac{y}{c^2 (x^2 + y^2)}) dx
%  \\
 &=  y \sum_{c \ge 1} S(r,q;c)
 \intR e(-rxy - \frac{qx}{c^2 y(x^2+1)}) 
 \psi(\frac{1}{c^2 y (x^2 + 1)}) dx.
\end{align*}
The support of $\psi$ means that the $x$-integral may be restricted to
$c^2 y (1+x^2) \asymp N$.  Moreover, since $y \asymp 1$ by assumption, this means $c^2 (1+x^2) \asymp N$.  Thus the $c$-sum may be truncated at $c \ll N^{1/2}$, and then the $x$-integral is localized on $x \asymp N^{1/2}/c$.  Hence by the triangle inequality and the Weil bound we obtain
\begin{equation*}
|a_q(r, y, \psi)| \ll \sum_{c \ll \sqrt{N}} |S(r,q;c)| \frac{N^{1/2}}{c} \ll N^{3/4+\varepsilon}.
\end{equation*}
This implies $I \ll N^{3/4+\varepsilon}$, which is consistent with \eqref{eq:DFIshiftedsumBound}.  Indeed, we would argue that this proof is essentially the same as that of \cite{DFI}, though with a rather different cosmetic appearance.

\subsection{Discussions on test functions}
The authors did not thoroughly investigate the image of the map $\phi_N \mapsto W_{N}$, which is the Laplace transform of $\phi_N(Y)\det(Y)^{k-3}$ in $\mathrm{M}_2(\mr)^{\mathrm{sym}, >0}$ (see Proposition \ref{prop:unfoldingBasicPoincare} for an explicit expression for the weight function, and \cite[Sect. 7.2.3]{muirhead1982aspects} for the Laplace transform on $\mathrm{M}_2(\mr)^{\mathrm{sym}, >0}$).  A simple observation is that the integral transform on the RHS of \eqref{eq:PoincareSeriesIntegralTransform} extends to a holomorphic function of $M$, which places an obvious restriction on the class of test functions arising in the image of this map.  For comparison, in the $\GL_2$ setting, the resulting integral which appears in \eqref{eq:IformulaGL2model} is the Laplace transform of a smooth function with rapid decay at $0$, which implies that the resulting integral transform extends to a holomorphic function of $m+n$ with $\mathrm{Re}(m+n) > 0$.

Based on the analogy with the $\GL_2$ problem, the success of Theorem \ref{thm:mainthm} suggests that one could study the shifted convolution problem for more general test functions using a symplectic delta method and a similar sequence of steps as in \cite{DFI} (delta method, Voronoi summation, and the Weil bound).
In order for this method to succeed, some groundwork needs to be developed, especially the derivation of a symplectic delta method.  We expect that such a delta method could be reverse-engineered from Theorem \ref{thm:Einc=1}, 
similarly to \eqref{eq:deltaViaNelson},
but leave this problem for a future occasion.

\subsection{Acknowledgments}
The authors thank Stephen D. Miller for helpful discussions on Eisenstein series.

\section{Setup}
The authors found \cite{JLSfundamental, JLSQUE, Hangman} useful sources for much of this material, and refer to them for more comprehensive discussions.  Many of their notational choices have been adopted in this work.
\subsection{Notation with groups}
Let $\M_n(R)$ denote $n \times n$ matrices with elements in a ring $R$, and similarly let
$\M_n^{\text{sym}}(R)$ denote the symmetric matrices in $\M_n(R)$.
Let $G= \mathrm{Sp_4}(\mr)$ and $\Gamma = \mathrm{Sp}_4(\mz)$.
For $X \in \M_2^{\text{sym}}(\mr)$, let $n(X) = (\begin{smallmatrix} I & X \\ & I \end{smallmatrix}) \in G$.
For $A \in \GL_2(\mr)$, let $m(A) = (\begin{smallmatrix} A \\ & A^{-t} \end{smallmatrix})$.
The Iwasawa decomposition implies that $G = N(\mr) M(\mr)^{+} K_{\infty}$, where $N(\mr) = \{ n(X): X \in \M_2^{\text{sym}}(\mr) \}$, $M(\mr)^{+} = \{ m(A) : A \in \mathrm{GL}_2^{+}(\mr) \}$, and $K_{\infty} = \{ (\begin{smallmatrix} A & B \\ - B & A\end{smallmatrix}) \in G \}$.

It is useful to pass back and forth between $\mh_2$ and $G$.  
Recall the map $\iota: \mathrm{GL}_2^{+}(\mr)/\mathrm{SO}_2(\mr) \xrightarrow{\sim} \mathrm{M}_2(\mr)^{\mathrm{sym}, >0}$, $\iota(g) = g g^t$.
We have the identification $G/K_{\infty} \leftrightarrow \mh_2$, as follows.  For $g = (\begin{smallmatrix} A & B \\ C & D \end{smallmatrix}) \in G$, we have $g(iI) = (Ai+B)(Ci+D)^{-1} \in \mh_2$.  The stabilizer of $iI$ in $G$ is easily seen to be $K_{\infty}$.  Moreover, $G$ acts transitively on $\mh_2$, since if $Z=X+iY \in \mh_2$, then we can write $Y = \iota(R) = R  R^t$ with $R \in \mathrm{GL}_2^{+}(\mr)/\mathrm{SO}_2(\mr)$, and then $n(X) m(R)(i) = X + iY$.  Hence we have the identification
\begin{equation}
\label{eq:SigelidentificationUpperHalfPlane}
    Z = X + iY \in \mh_2 \longleftrightarrow \begin{pmatrix} 1 & X \\ & 1 \end{pmatrix} \begin{pmatrix} R & \\ & R^{-t} \end{pmatrix} K_{\infty} \in G/K_{\infty}, \qquad Y = R R^t.
\end{equation}
We will frequently use the following coordinates on $\mh_2$.  
For $X+iY \in \mathbb{H}_2$, 
write $Y = (\begin{smallmatrix} y_1 & y_2 \\ y_2 & y_3 \end{smallmatrix}) = \iota(R)$ with $R \in \mathrm{GL}_2^{+}(\mr)$ expressed in Iwasawa form as $R  = (\begin{smallmatrix} 1 & u \\ & 1 \end{smallmatrix}) (\begin{smallmatrix} \sqrt{r_1} & \\ & \sqrt{r_2} \end{smallmatrix}) k$ with $k \in \mathrm{SO}_2(\mr)$.
Explicitly performing the matrix multiplication gives
\begin{equation}
\label{eq:YvsAcoordinates}
y_1 = r_1 + u^2 r_2, \qquad y_2 = u r_2, \qquad y_3 = r_2, \qquad \det(Y) = r_1 r_2.  
\end{equation}
Moreover, $\frac{dY}{\det(Y)^3} = \frac{du dr_1 dr_2}{r_1^3 r_2^2}$ (cf. \eqref{eq:PeterssonInnerProductDefinitionSiegel}).

%\blue{Note to self: The condition that this matrix is in $K_{\infty}$ is that $A A^t + B B^t = I$ and $A B^t = B A^t$.  This shows that $m((\begin{smallmatrix} -1 & \\ & 1 \end{smallmatrix}))$ of determinant $-1$ is in $K_{\infty}$.  This explains why $M(\mr)^{+}$ appears here.}

Following the presentation in \cite{Hangman},
define the minimal parabolic subgroup $P_0$ and the Siegel parabolic subgroup $P_{\alpha}$ by 
\begin{equation*}
P_0 = \left\{ \begin{pmatrix} * & * & * & * \\ & * &* & * \\ & & * & \\ & & * & * \end{pmatrix} \right\} \cap G,
\qquad 
P_\alpha = \left\{ \begin{pmatrix} * & * & * & * \\ * & * &* & * \\ & & * & * \\ & & * & * \end{pmatrix} \right\} \cap G.
\end{equation*}
We have the decompositions $P_{\alpha}  = N_{\alpha} M_{\alpha}$ and $P_0 = N_0 M_0$ where $N_{\alpha} = N(\mr)$, $M_{\alpha} = M(\mr)$, and
\begin{equation*}
    N_0 = \left\{ \begin{pmatrix} 1 & n_1 & n_2 & n_3 \\ & 1 & n_4 & n_5 \\ & & 1 & \\ & & -n_1 & 1 \end{pmatrix} \in G,
    \qquad n_3 = n_1 n_5 + n_4
    \right\},
\end{equation*}
and $M_0 = \{ \mathrm{diag}(y_1, y_2, y_1^{-1}, y_2^{-1}) : y_1, y_2 > 0 \}$.  Define the map $m_{\alpha}:G/K_{\infty} \rightarrow M_{\alpha}^+/(K_{\infty} \cap M_{\alpha}^+)$ as the projection map with respect to the Iwasawa decomposition $G = N_{\alpha} M_{\alpha}^+ K_{\infty}$ (note that $K_{\infty} \cap M_{\alpha}^+ = \{ (\begin{smallmatrix} k & \\ & k^{-t} \end{smallmatrix}) : k \in \mathrm{SO}_2(\mr) \} = m(\mathrm{SO}_2(\mr)))$. 
It is also convenient to define $\mu_{\alpha}: M_{\alpha}^{+}/(K_{\infty} \cap M_{\alpha}^{+}) \rightarrow \GL_2^{+}(\mr)/\mathrm{SO}_2(\mr)$, by $\mu_{\alpha}(\begin{smallmatrix} A & \\ & A^{-t} \end{smallmatrix}) = A \cdot \mathrm{SO}_2(\mr)$.  Finally, we let
$\overline{m}_{\alpha} : G/K_{\infty} \rightarrow \GL_2^{+}(\mr)/\mathrm{SO}_2(\mr)$, by $\overline{m}_{\alpha} = \mu_{\alpha} \circ m_{\alpha}$.

In particular, for $Z \in \mh_2$, corresponding to $g K_{\infty}$ via \eqref{eq:SigelidentificationUpperHalfPlane} (i.e., $g(i) = Z$), we have
\begin{equation}
\label{eq:iotainverseVSmalpha}
    \iota^{-1}(\mathrm{Im}(Z)) = \overline{m}_{\alpha}(g).
\end{equation}

For $X \in \mathrm{M}_{2}(\mr)$, we typically use the sup norm $\| X \| = \| X \|_{\infty}$, though occasionally the operator norm $\|X \|_{\mathrm{op}}$ is useful.

\subsection{Groups, cosets, etc.}

The following simple lemma is used a number of times.
\begin{mylemma}
\label{lemma:malphaForGL2vsSp4commutingFormula}
Let $p = n(X) m(A) \in P_{\alpha}$ and $q \in G$.  Then
$\overline{m}_{\alpha}(p q) = A \overline{m}_{\alpha}(q)$.  In particular,
if $\gamma \in \mathrm{GL}_2(\mr)$ and $g \in G$, then
\begin{equation}
\label{eq:malphaForGL2vsSp4commutingFormula}
\gamma \cdot \overline{m}_{\alpha}(g) = \overline{m}_{\alpha}(m(\gamma) g).
\end{equation}
\end{mylemma}
\begin{proof}
Write $q$ in Iwasawa form, say
%$p = (\begin{smallmatrix} 1 & X \\  & 1 \end{smallmatrix}) %(\begin{smallmatrix} A & \\ & A^{-t} \end{smallmatrix})$ and
$q = (\begin{smallmatrix} 1 & Y \\  & 1 \end{smallmatrix}) (\begin{smallmatrix} B & \\ & B^{-t} \end{smallmatrix}) k$.
Then
\begin{equation*}
p q = (\begin{smallmatrix} 1 & X \\  & 1 \end{smallmatrix}) (\begin{smallmatrix} A & \\ & A^{-t} \end{smallmatrix})
(\begin{smallmatrix} 1 & Y \\  & 1 \end{smallmatrix}) (\begin{smallmatrix} B & \\ & B^{-t} \end{smallmatrix}) k
= (\begin{smallmatrix} 1 & X + AYA^t \\  & 1 \end{smallmatrix}) (\begin{smallmatrix} AB & \\ & (AB)^{-t} \end{smallmatrix}) k.
\end{equation*}
Then
 $\overline{m}_{\alpha}(pq) = AB \cdot \mathrm{SO}_2(\mr) = A \overline{m}_{\alpha}(q)$.  Then \eqref{eq:malphaForGL2vsSp4commutingFormula} follows from $p = m(\gamma)$.
\end{proof}

% The following is easily verified.
% \begin{mylemma}[Explicit computation of Iwasawa form]
% \label{lemma:IwasawaFormExplicit}
% Suppose $g = (\begin{smallmatrix} a & b \\ c & d \end{smallmatrix}) \in \mathrm{GL}_2(\mr)$.  
% Let $h = \sqrt{c^2 + d^2}$,  $D = \det(g)$, and $\delta = D/|D|$.
% Then
% \begin{equation*}
% g = \begin{pmatrix} 1 & \frac{ac + bd}{h^2} \\ 0 & 1 \end{pmatrix}
% \begin{pmatrix} \frac{\delta D}{h} & 0 \\ 0 & h \end{pmatrix}
% \underbrace{\begin{pmatrix} \delta d/h & -\delta c/h \\ c/h & d/h \end{pmatrix}}_{\in O_2(\mr)}.
% \end{equation*}
% \end{mylemma}
%Note that $\frac{ac+bd}{h^2}$ is invariant under the center, while $D/h$ and $h$ scale linearly with the center.
  %\begin{proof}
 %Direct calculation.
% First note that
% \begin{equation}
% \begin{pmatrix} a & b \\ c & d \end{pmatrix}
% = 
% \begin{pmatrix} \frac{ad-bc}{h} & \frac{ac+bd}{h} \\ 0 & h \end{pmatrix}
% \begin{pmatrix} d/h & -c/h \\ c/h & d/h \end{pmatrix},
% \end{equation}
% which is a $BK$ factorization.  (I computed this by multiplying $g$ on the right by a rotation to kill the lower-left entry). Then it's easy to factor in $B$.
%\end{proof}

\subsection{Poincare series}
Recall the definition \eqref{eq:PoincareSeriesDefinition}.
These types of Poincare series facilitate the unfolding process with respect to the Fourier expansion \eqref{eq:SiegelFourierExpansion}.  Suppose that $f \in L^{\infty}(\Gamma \backslash \mh_2)$.  Then unfolding gives
\begin{equation}
\label{eq:PoincareSeriesUnfoldingWithYcoord}
\int_{\Gamma \backslash \mh_2} f(Z) P_Q(Z, \phi) d\mu 
= \int_{\Gamma_{\infty} \backslash \mh_2} f(X+iY) e(\mathrm{Tr}(Q X)) \phi(\iota^{-1}(Y)) d\mu.
\end{equation}
Using the coordinates \eqref{eq:YvsAcoordinates}, we obtain
\begin{equation}
\label{eq:PoincareSeriesUnfolding}
\int_{\Gamma \backslash \mh_2} f(Z) P_Q(Z, \phi) d\mu 
= \int_{\substack{x_1, x_2, x_3 \in \mz \backslash \mr \\ 
r_1, r_2 \in \mr_{>0}, u \in \mr}
}
f\big((\begin{smallmatrix} 1 & X \\ & 1 \end{smallmatrix})( \begin{smallmatrix} R & \\ & R^{-t} \end{smallmatrix}) \big)
 e(\mathrm{Tr}(Q X)) \phi(R) 
dX \frac{ du dr_1 dr_2}{r_1^3 r_2^2}
.
\end{equation}

Taking $Q=0$ in \eqref{eq:PoincareSeriesDefinition} gives rise to an incomplete Eisenstein series
\begin{equation}
\label{eq:EincompleteDef}
P_0(Z, \phi) = \sum_{\gamma \in \Gamma_{\infty} \backslash \Gamma} 
\phi(\iota^{-1} \mathrm{Im}(\gamma Z)).
\end{equation}
We wish to connect \eqref{eq:EincompleteDef} with the automorphic Eisenstein series appearing in \cite{Hangman}, which uses different notation.
First
note that $\Gamma_{\infty} = N_{\alpha} \cap \Gamma$.
Using the identification \eqref{eq:SigelidentificationUpperHalfPlane}, more specifically via \eqref{eq:iotainverseVSmalpha},
we have the desired alternative expression
  %Hence, with $Z = X+iY$, $Y=R R^t$ and $g = (\begin{smallmatrix} 1 & X \\ & 1 \end{smallmatrix}) (\begin{smallmatrix} R & \\ & R^{-t} \end{smallmatrix}) k$, we have
\begin{equation}
\label{eq:P0defVariant}
    P_0(Z, \phi) = \sum_{\gamma \in (N_{\alpha} \cap \Gamma) \backslash \Gamma} 
    \phi(\overline{m}_{\alpha}(\gamma g)),
    \qquad (\text{where } g(i) = Z).
\end{equation}

 %Therefore, a function $\phi: \mathrm{GL}_2^{+}(\mr)/\mathrm{SO}_2(\mr)$ of the form $\phi(R)$ can be thought of as a function on the $M_{\alpha}$ component of $G$ or $G/K = \mh_2$.  

%Then in the notation of \eqref{eq:EincompleteDef}, we have that $\phi(\iota^{-1}(\mathrm{Im}(Z))$ is the same as $\phi(m_{\alpha}(g))$ where $g$ corresponds to $Z$, meaning $g(i) = Z$.

\subsection{Kloosterman sums}
\label{section:KloostermanSumsSymplectic}
We define symplectic Kloosterman sums following Kitaoka.  Let $C \in \mathrm{M}_2(\mz)$ with $\det(C) \neq 0$.  For $Q, T \in \Lambda$, define
\begin{equation*}
K(Q,T;C) = \sum_{D} e(\mathrm{Tr}(A C^{-1} Q + C^{-1} D T)),
\end{equation*}
where $D$ runs over $\{D \in \mathrm{M}_2(\mz) \mod C \Lambda' : (\begin{smallmatrix} * & * \\ C & D \end{smallmatrix}) \in \Gamma\}$, where $\Lambda' = \mathrm{M}_2^{\mathrm{sym}}(\mz)$.  Moreover, $A$ is any matrix in $\mathrm{M}_2(\mz)$ such that $(\begin{smallmatrix} A & * \\ C & D \end{smallmatrix}) \in \Gamma$.  Kitaoka showed that this Kloosterman sum is well-defined.  
In addition, he showed the useful symmetry \cite[Prop. 1]{Kitaoka}
\begin{equation}
\label{eq:KloostermanSymmetry}
    K(Q, T;C) = K(T,Q;C^t).
\end{equation}

\begin{mylemma}[Kitaoka's bound]
\label{lemma:KloostermanSumKitaokasBound}
Let $C \in \mathrm{M}_2(\mz)$, $\det(C) \neq 0$, and $C = U^{-1} (\begin{smallmatrix} \alpha_1 & \\ & \alpha_2 \end{smallmatrix}) V^{-1}$ with $U \in \mathrm{GL}_2(\mz)$$, V \in \mathrm{SL}_2(\mz)$ and $\alpha_1 \mid \alpha_2$ with $\alpha_1, \alpha_2 > 0$.  Then
\begin{equation}
K(Q,T;C) \ll \alpha_1^2 \alpha_2^{1/2+\varepsilon} (\alpha_2, t)^{1/2},
\end{equation}
where $t$ is the lower-right entry of $V^t T V$.
\end{mylemma}
Remarks.  There have been additional improvements to Kitaoka's bound, as in \cite{Toth}, but Kitaoka's original result is sufficient for our purposes.  
%In addition, Kitaoka stated the result with $V \in \mathrm{GL}_2(\mz)$ but if $\det(V) = -1$ one can multiply through with $(\begin{smallmatrix} -1 & \\ & 1 \end{smallmatrix})$ to reduce to the case stated above.

We collect some more refined pieces of information to aid in applying Kitaoka's bound, the main issue being the parameter $t$ which is not yet given in an explicit form amenable to further estimates.  First, we note that $\alpha_1 = \gcd(c_1, c_2, c_3, c_4)$.  Say $\alpha_2 = \beta \alpha_1$. 
Then we have the simplified bound
\begin{equation}
\label{eq:KitaokaSimplified}
K(Q, T;C) \ll \alpha_1^{3+\varepsilon} \beta^{1/2+\varepsilon} (\beta, t)^{1/2}.
\end{equation}

We employ the standard notations 
\begin{equation*}
    \Gamma^0(q)=\{(\begin{smallmatrix}
        a& b \\ c & d
    \end{smallmatrix})\in \mathrm{SL}_2(\mathbb{Z}):q|b\} \quad \text{ and } \quad \Gamma_0(q)=\{(\begin{smallmatrix}
        a& b \\ c & d
    \end{smallmatrix})\in \mathrm{SL}_2(\mathbb{Z}):q|c\},
\end{equation*}
and recall the well-known relation $(\begin{smallmatrix} 1 & \\ & \beta \end{smallmatrix}) \Gamma^0(\beta) (\begin{smallmatrix} 1 & \\ & \beta \end{smallmatrix})^{-1} = \Gamma_0(\beta)$.

\begin{mylemma}
\label{lemma:UandVclassification}
Suppose $C = U_0^{-1} (\begin{smallmatrix} 1 & \\ & \beta \end{smallmatrix}) V_0^{-1}$ with $(U_0, V_0) \in \mathrm{GL}_2(\mz) \times \mathrm{SL}_2(\mz)$, and $\beta \geq 1$.  Then the set of all pairs $(U,V) \in \mathrm{GL}_2(\mz) \times \mathrm{SL}_2(\mz)$ such that $C = U^{-1} (\begin{smallmatrix} 1 & \\ & \beta \end{smallmatrix}) V^{-1}$ are those of the form $(\gamma' U_0, V_0 \gamma)$ where $\gamma \in \Gamma^{0}(\beta)$ and $\gamma' = (\begin{smallmatrix} 1 & \\ & \beta \end{smallmatrix}) \gamma^{-1} (\begin{smallmatrix} 1 & \\ & \beta \end{smallmatrix})^{-1}$.
\end{mylemma}
\begin{proof}
Suppose $U_0^{-1} (\begin{smallmatrix} 1 & \\ & \beta \end{smallmatrix}) V_0^{-1} = U^{-1} (\begin{smallmatrix} 1 & \\ & \beta \end{smallmatrix}) V^{-1}$ and $V = V_0 \gamma$ and $U = \gamma' U_0$, with $\gamma \in \mathrm{SL}_2(\mz)$ and $\gamma' \in \mathrm{GL}_2(\mz)$.  Then $\gamma' (\begin{smallmatrix} 1 & \\ & \beta \end{smallmatrix}) \gamma = (\begin{smallmatrix} 1 & \\ & \beta \end{smallmatrix})$, which implies $\gamma \in \Gamma^{0}(\beta)$, and then $\gamma' = (\begin{smallmatrix} 1 & \\ & \beta \end{smallmatrix}) \gamma^{-1} (\begin{smallmatrix} 1 & \\ & \beta \end{smallmatrix})^{-1}$, as claimed.
\end{proof}

\begin{mylemma}
\label{lemma:CosetRepresentativesGamma0beta}
A complete set of coset representatives for $\mathrm{SL}_2(\mz)/\Gamma^{0}(\beta)$ is
the set of matrices
of the form
\begin{equation*}
\gamma_{p,q} = \begin{pmatrix} * & p \\ * & q \end{pmatrix} 
\end{equation*}
where $q \mid \beta$ and $p \pmod{\beta/q}$ (with $\gcd(p,q) = 1$, and the left column may be chosen so $\det(\gamma_{p,q}) = 1$).
\end{mylemma}
Remark.  This is a small modification of \cite[Prop. 2.5]{IwaniecTopics}.
\begin{proof}
For $i=1,2$ let $\gamma_i = (\begin{smallmatrix} * & p_i \\ * & q_i \end{smallmatrix}) \in \mathrm{SL}_2(\mz)$ be of the stated form.  The condition that $\gamma_1 \Gamma^0(\beta) = \gamma_2 \Gamma^0(\beta)$ is $\gamma_2^{-1} \gamma_1 \in \Gamma^{0}(\beta)$.  The upper-right entry of $\gamma_2^{-1} \gamma_1$ is $p_1 q_2 - p_2 q_1$, which is $\equiv 0 \pmod{\beta}$ if and only if $q_1 = q_2$ and $p_1 \equiv p_2 \pmod{\beta/q_1}$.
Hence the $\gamma_{p,q}$'s represent distinct cosets.  
Next we show that these cosets cover $\SL_2(\mz)$.
For $g = (\begin{smallmatrix} a & b \\ c & d \end{smallmatrix}) \in \SL_2(\mz)$, let $q = \gcd(d, \beta)$.  Then $\gamma_{p,q}^{-1} g = (\begin{smallmatrix} * & bq-pd \\ * & * \end{smallmatrix})$.  The condition that $\gamma_{p,q}^{-1} g \in \Gamma^{0}(\beta)$ amounts to $bq \equiv pd \pmod{\beta}$, equivalently, $p \equiv b (\frac{d}{q})^{-1} \pmod{\beta/q}$.  
Finally, we claim that we may choose a representative of this $p$ modulo $\beta/q$ with $(p,q) = 1$.  This is an exercise in elementary number theory, which we omit. 
% which we abstract out as follows. Suppose $m \in \mz$ and $\ell, r \geq 1$ are given, with $(m, \ell) = 1$.  When does there exist $x \in \mz$ such that $x \equiv m \pmod{\ell}$ with $(x, q) = 1$?
% The answer is, always.  Say $\ell = \ell_1 \ell_2$ where $\ell_1 \mid q^{\infty}$ and $(\ell_2, m) = 1$.  By the Chinese remainder theorem, we seek a solution to the system $x \equiv m \pmod{\ell_1}$, $x \equiv m \pmod{\ell_2}$, and $(x, q) = 1$.
% We can then write $q = q_1 q_2$ where $q_1 \mid \ell_1^{\infty}$ and $(q_2, \ell_1) = 1$.  The condition $(x, q_2) = 1$ is free from $x \equiv m \pmod{\ell_1}$.  Then we have two separate conditions that $x \equiv m \pmod{\ell_2}$ and $(x, q_2) = 1$.  
\end{proof}

Combining Lemmas \ref{lemma:UandVclassification} and \ref{lemma:CosetRepresentativesGamma0beta}, for each $C \in \mathrm{M}_2(\mz)$ with $\gcd(c_1, c_2, c_3, c_4) = 1$ and $\det(C) = \beta > 0$, 
then $C = U^{-1} (\begin{smallmatrix} 1 & \\ & \beta \end{smallmatrix}) V^{-1}$ for some
$V$ of the form $\gamma_{p,q}$.  %One small loose end here is that $V \in \mathrm{GL}_2(\mz)$, but not necessarily $\mathrm{SL}_2(\mz)$, but if $\det(V) = -1$ then we can multiply $V$ on the right by $(\begin{smallmatrix} -1 & \\ & 1 \end{smallmatrix})$, which commutes with $(\begin{smallmatrix} 1 & \\ & \beta \end{smallmatrix})$, and absorb it into $U$.  Thus we may assume that $V \in \mathrm{SL}_2(\mz)$.  
The next lemma characterizes when $C$ has its $V$ in the coset $\gamma_{p,q} \Gamma^0(\beta)$.
\begin{mylemma}
\label{lemma:CUVpq}
Suppose $C = U^{-1} (\begin{smallmatrix} 1 & \\ & \beta \end{smallmatrix}) V^{-1}$ with $U \in \mathrm{GL}_2(\mz)$, 
$V \in \mathrm{SL}_2(\mz)$, 
and $\beta \geq 1$.  The condition that $V \in \gamma_{p,q} \Gamma^{0}(\beta)$ is determined from $C$ as follows.  First, $q = \gcd(c_1, c_3)$.  Supposing $r, s \in \mz$ are such that $\frac{c_1}{q} r + \frac{c_3}{q} s = 1$, then $p$ is uniquely determined $\pmod{\beta/q}$ by
\begin{equation*}
-p \equiv c_2 r + c_4 s \pmod{\beta/q}.
\end{equation*}
\end{mylemma}
\begin{proof}
By Lemma \ref{lemma:UandVclassification},
the condition that $V \in \gamma_{p,q} \Gamma^{0}(\beta)$ is equivalent to 
the existence of $U \in \mathrm{GL}_2(\mz)$ such that
$UC = (\begin{smallmatrix} 1 & \\ & \beta \end{smallmatrix}) \gamma_{p,q}^{-1} = (\begin{smallmatrix} q & -p \\ \beta x & \beta y \end{smallmatrix})$, for some $x, y \in \mz$ with $qy+px = 1$.  
Observe that left multiplication by $U$ does not change the gcd of the elements in a single column.  Noting that $\gcd(q,\beta x) = q$ since $\gcd(x,q) = 1$, then
this implies that $\gcd(c_1, c_3) = q$.
%Similarly, $\gcd(c_2, c_4) = \gcd(p, \beta y) = \gcd(p, \beta) = \gcd(p, \beta/q)$.  
Next we compute
\begin{equation*}
\begin{pmatrix} 1 & \\ & \beta \end{pmatrix} \gamma_{p,q}^{-1} C^{-1} = \pm \begin{pmatrix} q & - p \\ \beta x & \beta y \end{pmatrix} \begin{pmatrix} c_4/\beta & -c_2/\beta \\ -c_3/\beta & c_1/\beta \end{pmatrix} 
= \pm
\begin{pmatrix} 
\frac{qc_4 + p c_3}{\beta} & - \frac{(qc_2 + p c_1)}{\beta} \\ c_4x -c_3 y & c_1 y - c_2 x
\end{pmatrix}.
\end{equation*}
For this matrix to be in $\mathrm{GL}_2(\mz)$, it is both necessary and sufficient that $qc_4 +p c_3 \equiv 0 \pmod{\beta}$
and $qc_2 + pc_1 \equiv 0 \pmod{\beta}$.  Given that $q$ divides $c_1$ and $c_3$, these two congruences in turn are equivalent to $p \frac{c_3}{q} + c_4 \equiv 0 \pmod{\beta/q}$ and $p \frac{c_1}{q} + c_2 \equiv 0 \pmod{\beta/q}$.
Hence $p \frac{c_3}{q} s \equiv - c_4 s \pmod{\beta/q}$ and $p\frac{c_1}{q} r \equiv - c_2 r \pmod{\beta/q}$.  Adding the two congruences gives $p \equiv - c_2 r - c_4 s \pmod{\beta/q}$.
\end{proof}
As a direct consequence of Lemma \ref{lemma:CUVpq}, we may deduce the following.
\begin{mycoro}
\label{coro:tformula}
Suppose $UC V = (\begin{smallmatrix} 1 & \\ & \beta \end{smallmatrix})$, with $\beta \geq 1$, $U \in \mathrm{GL}_2(\mz)$, and $V \in \gamma_{p,q} \Gamma^{0}(\beta)$.
Let $T= (\begin{smallmatrix} t_1 & t_2/2 \\ t_2/2 & t_3 \end{smallmatrix}) \in \Lambda$.
Let $t$ be the lower-right entry of $V^t T V$.  Then
$t \equiv t_1 p^2 + t_2 pq + t_3 q^2 \pmod{\beta}$.  In particular,
\begin{equation*}
t \equiv t_1 (c_2 r + c_4 s)^2 - t_2 q (c_2 r + c_4 s) + t_3 q^2 \pmod{\beta},
\end{equation*}
where $q = \gcd(c_1, c_3)$ and $r, s \in \mz$ are such that $c_1 r + c_3 s = q$.
\end{mycoro}

% \begin{mylemma}\label{lem.quadraticformDivisor0}
%     Let $T= (\begin{smallmatrix} t_1 & t_2/2 \\ t_2/2 & t_3 \end{smallmatrix}) \in \Lambda$, with $T \neq 0$. Write $T(x,y)=(x\, y) T (x\, y)^t= t_1x + t_2xy + t_3y$. Then for $X \geq 1$, we have
%     \begin{align}
%         R_{T,X}(0):=|\{(x,y)\in \mathbb{Z}^2:T(x,y)=0, |x| + |y| \leq X\}|\ll X\delta(-4\det(T)=\square).
%     \end{align}
% \end{mylemma}
% \begin{proof}
%     We first treat the case $t_1=t_3=0$. In such a case, we have $t_2\neq0$, $-4\det(T)=t_2^2$, and $T(x,y)=t_2xy$. Then either $x=0$ or $y=0$, and we have $R_{T,X}(0)\ll X$ in this case.

% Now we are left with the case where either $t_1\neq0$ or $t_3\neq0$. Since our treatment will be symmetric with respect to $t_1$ and $t_3$, we may assume that $t_1\neq 0$. Note that \begin{align}
%     4t_1T(x,y)=(2t_1 x + t_2 y)^2 + \Delta y^2 \quad \text{ where } \quad \Delta = 4t_1 t_3 - t_2^2.
% \end{align}
% If $T(x,y)=0$, then $$2t_1 x + t_2 y=\pm \sqrt{-\Delta} y.$$
% As a result, $R_{T,X}^*(0)=0$ unless $\Delta=-a^2$ for some integer $a\geq0$. In such a case, we have \begin{align}
%     R_{T,X}(0)\ll X.
% \end{align}

% \end{proof}

\section{Counting problems}
\label{section:countingproblems}
This self-contained section considers some counting problems that turn out to arise when estimating the Fourier coefficients of Poincare series using Kitaoka's bound.

\begin{mylemma}[Quadratic forms are nearly one-to-one]
\label{lem.quadraticformDivisor}
    Let $T= (\begin{smallmatrix} t_1 & t_2/2 \\ t_2/2 & t_3 \end{smallmatrix}) \in \Lambda$, with $\det(T) \neq 0$. Write $T(x,y)=(x\, y) T (x\, y)^t= t_1x^2 + t_2xy + t_3y^2$. Then for $X \geq 1$ and $n \neq 0$, we have
    \begin{align*}
        R_{T,X}(n):=|\{(x,y)\in \mathbb{Z}^2:T(x,y)=n, |x| + |y| \leq X\}|\ll (|n| X \|T\| )^\varepsilon,
    \end{align*}
    with an implied constant depending on $\varepsilon > 0$ only.
\end{mylemma}
\begin{proof}
% Since $T$ is positive definite, the discriminant is equal to $t_2^2-4t_1t_3=-4\det(T)<0$. Define \begin{align}
%         R^*(n):=\{(x,y)\in \mathbb{Z}^2:(x,y)=1, T(x,y)=n\}.
%     \end{align}
% We claim that \begin{align}
%     |R^*(n)|\ll n^\varepsilon.
% \end{align}
% With this claim, we have \begin{align}
%         |\{(x,y)\in \mathbb{Z}^2:T(x,y)=n\}|\leq \sum_{a^2|n}|R^*(n/a^2)|\ll n^\varepsilon,
%     \end{align}
%     which concludes the proof.

% The claim follows from standard number theory results. For completeness, we follow \cite[Ch 11, Thm 4.1, 4.3]{Hua's intro to number theory} and illustrate the proof here.

% Let $(x,y)\in R^*$, then there exists $r_0,s_0$ such that $xs_0-yr_0=1$. Then all solutions $(r,s)$ to the equation $xs-yr=1$ are given by $r=r_0+hx$ and $s=s_0+hy$ for any integer $h$. Let \begin{align}
%     \ell=&\ (2t_1x+t_2y)r+(t_2x+2t_3y)s=(2t_1x+t_2y)r_0+(t_2x+2t_3y)s_0+2hT(x,y)\nonumber\\
%     =&\ (2t_1x+t_2y)r_0+(t_2x+2t_3y)s_0+2hn.
% \end{align}
% Choose $h$ so that $0\leq \ell < 2n$. On the other hand, we have \begin{align}
% l^2 & =[(2 t_1 x+t_2 y) r+(t_2 x+2 t_3 y) s]^2 \nonumber\\
% & =4\left(t_1 r^2+t_2 r s+t_3 s^2\right)T(x,y)+\left(t_2^2-4 t_1 t_3\right)(x s-y r)^2 \nonumber\\
% & \equiv -4\det(T)  \ \shortmod{4n} .
% \end{align}
% In other words, for each $(x,y)\in R^*$, we have associated an integer $0\leq \ell <2n$ such that $\ell^2\equiv -4\det(T)\mod{4n}$. This implies that there are at most $\ll 2^{\omega(n)}$ many distinct $\ell$.

We first treat the case $t_1=t_3=0$. Since $\det(T) \neq 0$, then $t_2\neq0$ and $T(x,y)=t_2xy$. 
The standard divisor function bound implies $R_{T,X}(n) \leq 2 d(n) \ll |n|^{\varepsilon}$.
% \begin{align}
%     R_{T,X}(n)=|\{(x,y)\in \mathbb{Z}^2:t_2xy=n, |x| + |y| \leq X\}|\ll |n|^\varepsilon.
% \end{align}

Now assume $t_1 \neq 0$ or $t_3 \neq 0$. Say by symmetry that $t_1\neq 0$. Note that \begin{align*}
    4t_1T(x,y)=(2t_1 x + t_2 y)^2 + \Delta y^2 \quad \text{ where } \quad \Delta = 4t_1 t_3 - t_2^2.
\end{align*}
Write $\Delta=\Delta_0^2\Delta'$ with $\Delta'$ square-free. If $\Delta'=-1$, then we have the factorization
\begin{align*}
    4t_1T(x,y)=(2t_1 x + t_2 y + \Delta_0 y)(2t_1 x + t_2 y - \Delta_0 y).
\end{align*}
By the standard divisor function bound, this implies
\begin{align*}
    R_{T,X}(n)\ll (|t_1|n)^\varepsilon.
\end{align*}

Now assume $\Delta'\neq-1$. Let $K=\mathbb{Q}(\sqrt{-\Delta'})$, with ring of integers $\mathcal{O}_K$ and unit group $U_K$.  We need to count the number of pairs $(x,y) \in \mz^2$ with $|x| + |y| \leq X$, and with
\begin{equation*}
    4t_1 n = 4 t_1 T(x,y) = (2 t_1 x + t_2 y + \sqrt{-\Delta}  y)(2t_1 x + t_2 - \sqrt{-\Delta}  y).
\end{equation*}
The number of integral ideals dividing $(4 t_1 n)$ in $\mathcal{O}_K$ is $O(|t_1 n|^{\varepsilon})$, uniformly in $\Delta$ (indeed, the number of integral ideals is at most $d(4 t_1 n)^2$).  %Using the factorization
%Since $[K:\mathbb{Q}]=2$, the number of distinct prime ideals dividing $(4t_1n)$ is bounded by $\leq 2^{\omega(4t_1n)}\ll (|t_1n|)^\varepsilon$, where $\omega(4t_1 n)$ is the number of distinct prime factors of $4t_1n$. Factorizing 
% \begin{align}
%     (2t_1 x + t_2 y)^2 + \Delta y^2= (2t_1 x + t_2 y + \sqrt{-\Delta} y)(2t_1 x + t_2 y - \sqrt{-\Delta} y),
% \end{align}
Hence the number of principal ideals of the form $(2 t_1 x + t_2 y + \sqrt{-\Delta}  y)$ dividing $(4t_1 n)$ is at most $O(|t_1 n|^{\varepsilon})$.

If $(2t_1 x_1 + t_2 y_1 + \sqrt{-\Delta} y_1) \mathcal{O}_K = (2t_1 x_2 + t_2 y_2 + \sqrt{-\Delta} y_2) \mathcal{O}_K$ then
\begin{equation}
\label{eq:vunitformula}
 v:=   
 \frac{2t_1 x_1 + t_2 y_1 + \sqrt{-\Delta} y_1}{2t_1 x_2 + t_2 y_2 + \sqrt{-\Delta} y_2} \ll \frac{ \|T\|^2 X^2}{4 |t_1 n|},
% \frac{2t_1 x_1 + t_2 y_1 + \sqrt{-\Delta} y_1}{2t_1 x_2 + t_2 y_2 + \sqrt{-\Delta} y_2} \frac{2t_1 x_1 + t_2 y_1 + \sqrt{-\Delta} y_1}{2t_1 x_2 + t_2 y_2 + \sqrt{-\Delta} y_2}
\end{equation}
where $v \in U_K$.
If $\Delta>0$, then $|U_K|\leq 6$ and this concludes the proof of the lemma for this case.
If $\Delta'<-1$, then $U_K$ is given by 
\begin{align*}
    U_K=\{\pm u^j\} \quad \text{ with }\quad u=\tfrac{a+b\sqrt{D}}{2},
\end{align*}
where $$D=\begin{cases}
    -\Delta' & \text{ if } \Delta'\equiv 3\mod 4\\ -4\Delta' & \text{ if } \Delta'\equiv 1, 2\mod 4,
\end{cases}$$ 
and $a, b>0$ is the smallest integral solution to $x^2-Dy^2=\pm 4$. Since $a,b\geq1$ and $D\geq 3$, we have $u \geq \frac{1 + \sqrt{3}}{2} > 1.366\dots$.
. Hence the number of possible $v \geq 1$ satisfying \eqref{eq:vunitformula} is $\ll (\|T \| X)^{\varepsilon}$, uniformly in $\Delta$.  By symmetry, we may assume $v \geq 1$, which completes the proof.
\end{proof}

\begin{mytheo}\label{thm.Bounding(t,C)}
Let $T= (\begin{smallmatrix} t_1 & t_2/2 \\ t_2/2 & t_3 \end{smallmatrix}) \in \Lambda$, with $T \neq 0$.  Let $t = t(C, T)$ be the lower-right entry of $V^t T V$, where $UCV = (\begin{smallmatrix} 1 & \\ & \beta \end{smallmatrix})$ with $U \in \mathrm{GL}_2(\mz)$, $V \in \mathrm{SL}_2(\mz)$, and $\beta \geq 1$.  Let $1 \leq W \ll X^2$.  
Then
\begin{equation}
\label{eq:Nbound}
\mathcal{N}(X, W, T) :=\sum_{\substack{(c_1, c_2, c_3, c_4) = 1 \\ c_1^2 + c_2^2 + c_3^2 + c_4^2 \ll X^2 \\ 0 \neq |\det(C)| \ll W }}
\gcd(t(C, T), \det(C))^{1/2} \ll
\|T\|^{\varepsilon}
W X^{2+\varepsilon},
\end{equation}
where the implied constant depends on $\varepsilon > 0$ only.
\end{mytheo}
Remark.  The quantity $t=t(C,T)$ technically depends on $V$, not just $C$, but $\gcd(t, \det(C))$ does not depend on $V$, so $\mathcal{N}$ is well-defined.
\begin{proof}
Write $\mathcal{N} = \mathcal{N}_0 + \mathcal{N}'$ where $\mathcal{N}_0$ corresponds to the sub-sum with some $c_i =0$, and where $\mathcal{N}'$ corresponds to the sub-sum with all $c_i \neq 0$.
We first estimate $\mathcal{N}_0$.  Say by symmetry $c_1 = 0$.  Then $|\det(C)| = |c_2 c_3|$, $(\gcd(t, \det(C))^{1/2} \leq |\det(C)|^{1/2} \ll W^{1/2} \ll X$ and the contribution to $\mathcal{N}_0$ from such $C$'s is
\begin{equation*}
    \ll
    W^{1/2} \sum_{\substack{1 \leq |c_2 c_3| \ll W \\ c_4 \ll X}} 1 \ll W^{3/2} X \log X \ll W X^{2+\varepsilon}.
\end{equation*}

Motivated by later analysis, we subdivide further $\mathcal{N}' = \mathcal{N}'_0 + \mathcal{N}''$, where $\mathcal{N}'_0$ corresponds to the sub-sum of terms with $t_1 c_2^2 - t_2 c_1 c_2 + t_3 c_1^2 = 0$, and $\mathcal{N}''$ corresponds to the terms with
\begin{equation}
\label{eq:quadraticformnonvanishingcondition}
t_1 c_2^2 - t_2 c_1 c_2 + t_3 c_1^2 \neq 0.
\end{equation}
Note that $\mathcal{N}_0'$ is empty unless $\Delta:=4\det(T) = -\square$.  The solutions $(c_1, c_2) \in \mz^2$ to $t_1 c_2^2 - t_2 c_1 c_2 + t_3 c_1^2=0$ lie along the union of two rational lines through the origin.  
Specifically, if $t_1 \neq 0$, the lines are $c_2 = c_1 \frac{t_2 \pm \sqrt{-\Delta}}{2t_1}$.  If $t_1 = 0$, then since $c_1 \neq 0$ by assumption, then the solutions lie on the line $t_3 c_1 = c_2$, which only has solutions if $t_3 \neq 0$, since $c_2 \neq 0$ by assumption.
Consider a line of the form $c_1 = a c_2$ for some $a \in \mq$, $a \neq 0$.  
Let $g = \gcd(c_1, c_2)$ and write $c_i = g u_i$ where $u_1$ and $u_2$ are uniquely determined from $a$ (up to $\pm$ sign), and $\gcd(u_1, u_2) = 1$.  The contribution to $\mathcal{N}_0'$ from this line is
\begin{equation}
\label{eq:N0'boundOneLine}
\ll W^{1/2}
\sum_{g \ll W}
\sum_{\substack{c_3, c_4 \ll X \\ 1 \leq |u_1 c_4 - u_2 c_3| \ll W/g}} 1
= W^{1/2}
\sum_{g \ll W}
\sum_{1 \leq |m| \ll W/g} 
\sum_{\substack{c_3, c_4 \ll X \\ u_1 c_4 - u_2 c_3 = m}} 1.
\end{equation}
By elementary number theory, if $c_3^{(0)}$ and $c_4^{(0)}$ satisfy $u_1 c_4^{(0)} - u_2 c_3^{(0)} = m$, then all solutions to $u_1 c_4 - u_2 c_3 = m$ take the form $c_3 = c_3^{(0)} + k u_1$, $c_4 = c_4^{(0)} + k u_2$, with $k \in \mz$.
Hence 
$| \{  |c_3|, |c_4| \ll X : u_1 c_4 - u_2 c_3 = m \}| \ll X$, with an implied constant uniform in $T$.  Therefore,
\eqref{eq:N0'boundOneLine} is
\begin{equation*}
    \ll W^{1/2}
\sum_{g \ll W}
\sum_{1 \leq |m| \ll W/g} X \ll W^{3/2} X \log X \ll W X^{2+\varepsilon}.
\end{equation*}

Now we turn to $\mathcal{N}''$.  
%To simplify the notation, we will not explicitly display that the $c_i$ are nonzero and satisfy \eqref{eq:quadraticformnonvanishingcondition}, and rather leave these conditions implicit.
By Corollary \ref{coro:tformula}, $t(C,T) \equiv t_1 (c_2 r + c_4 s)^2 - t_2 q (c_2 r + c_4 s) + t_3 q^2 \pmod{|\det(C)|}$,
where $q = (c_1, c_3)$, and where $c_1 r + c_3 s = q$.  We decompose the sum in $\mathcal{N}''$ according to the value of $q$, and write $c_1 = q c_1'$ and $c_3 = q c_3'$, and formally let $C' = (\begin{smallmatrix} c_1' & c_2 \\ c_3' & c_4 \end{smallmatrix})$, so that $q \det(C') = \det(C)$.
This gives
\begin{equation*}
\mathcal{N}''(X,W, T)
\leq
\sum_{q \ll X} 
\sum_{\substack{(c_1', c_3') = 1 \\ 1 \leq |c_1'|, |c_3'| \ll X/q}}
\sum_{\substack{(c_2, c_4, q) = 1 \\  1 \leq |c_2|, |c_4| \ll X  \\ 0 \neq |\det(C')| \ll W/q \\ \eqref{eq:quadraticformnonvanishingcondition} \text{ holds}}} \gcd(t(C,T), q \det(C'))^{1/2}.
\end{equation*}
For each $C'$, we factor $\det(C') = \gamma_q \gamma''$ where $\gamma_q \mid q^{\infty}$ and $(\gamma'', q) = 1$.  Then
\begin{align*}
(t(C,T), q |\det(C')|)^{1/2} 
%&\leq (t(C,T), q \gamma_q)^{1/2} (t(C,T), \gamma'')^{1/2}
%\\
\leq q^{1/2} \gamma_q^{1/2} (t(C,T), \gamma'')^{1/2}.
\end{align*}
%Recall that $t(C,T) \equiv t_1 p^2 \pmod{q}$, where $\gcd(p,q) = 1$, so that $(t(C,T), q) = (t_1,q)$.  
Hence
\begin{align*}
\mathcal{N}''(X,W, T)
\leq
\sum_{q \ll X} 
q^{1/2}
\sum_{\substack{(c_1', c_3') = 1 \\ 1 \leq |c_1'|, |c_3'| \ll X/q}}
\sum_{\substack{(c_2, c_4, q) = 1 \\  1 \leq |c_2|, |c_4| \ll X \\ 0 \neq \det(C') \ll W/q \\ \eqref{eq:quadraticformnonvanishingcondition} \text{ holds}}}
\gamma_q^{1/2} 
\gcd(t(C,T), \gamma'')^{1/2}
\\
\leq 
\mathop{\sum_{(d,q) = 1} \sum_{v|q^\infty}}_{qdv\ll W}(qdv)^{1/2}
\sum_{\substack{\alpha \shortmod{d} \\ t_1 \alpha^2 - t_2 q \alpha + t_3 q^2 \equiv 0 \shortmod{d}}}
\sum_{\substack{(c_1', c_3') = 1 \\ 1 \leq |c_1'|, |c_3'| \ll X/q}}
\sum_{\substack{(c_2, c_4, q) = 1 \\  1 \leq |c_2|, |c_4| \ll X \\ 0 \neq \det(C') \ll W/q \\  \det(C') \equiv 0 \shortmod{dv} \\ c_2 r + c_4 s \equiv \alpha \shortmod{d} \\ \eqref{eq:quadraticformnonvanishingcondition} \text{ holds}}}
1.
\end{align*}
Given $c_1'$ and $c_3'$, the pair of congruences defining $c_2$ and $c_4$ can be expressed as
\begin{equation*}
\begin{pmatrix} c_1' & -c_3' \\ s & r \end{pmatrix} \begin{pmatrix} c_4 \\ c_2 \end{pmatrix}
\equiv \begin{pmatrix} 0 \\ \alpha \end{pmatrix} \pmod{d}
\Longleftrightarrow
 \begin{pmatrix} c_4 \\ c_2 \end{pmatrix}
\equiv \begin{pmatrix} c_3' \alpha \\ c_1' \alpha \end{pmatrix} \pmod{d},
\end{equation*}
since $(\begin{smallmatrix} c_1' & -c_3' \\ s & r \end{smallmatrix}) \in \mathrm{SL}_2(\mz)$.
% Hence $\mathcal{N}''(X,W,T)$ is bounded by
% \begin{equation}
% \mathop{\sum_{(d,q) = 1} \sum_{v|q^\infty}}_{qdv\ll W}(qdv)^{1/2}
% \sum_{\substack{\alpha \shortmod{d} \\ t_1 \alpha^2 - t_2 q \alpha + t_3 q^2 \equiv 0 \shortmod{d}}}
% \sum_{\substack{(c_1', c_3') = 1 \\ 1 \leq |c_1'|, |c_3'| \ll X/q}}
% \sum_{\substack{(c_2, c_4, q) = 1 \\  1 \leq |c_2|, |c_4| \ll X, 
% 0 \neq \det(C') \ll W/q \\ \det(C') \equiv 0 \shortmod{v} \\
% c_2 \equiv c_1' \alpha \shortmod{d} \\ c_4 \equiv c_3' \alpha \shortmod{d} \\ \eqref{eq:quadraticformnonvanishingcondition} \text{ holds}}} 
% 1.
% \end{equation}
% At this stage, we would like to separate the conditions between $(c_1',c_2)$ and $(c_3',c_4)$. Now for any given $0\leq \alpha<d$, since $c_1'c_4\neq c_2c_3'$, either $c_2\neq c_1'\alpha$ or $c_4\neq c_3'\alpha$. By symmetry, we have \begin{align}
%     \mathcal{N}'(X,T)\ll \sum_{q \ll X} 
% (t_1, q)^{1/2}
% \sum_{(d,q) = 1} \sum_{r|q^\infty}(dr)^{1/2}
% \sum_{\substack{\alpha \shortmod{d} \\ t_1 \alpha^2 - t_2 q \alpha + t_3 q^2 \equiv 0 \shortmod{d}}}
% \sum_{\substack{(c_1', c_3') = 1 \\ c_1',c_3' \ll X/q}}
% \sum_{\substack{(c_2, c_4, q) = 1 \\  c_2, c_4 \ll X, 
% r|c_1'c_4 -c_2c_3'\neq0 \\
% c_2 \equiv c_1' \alpha \shortmod{d} \\ c_4 \equiv c_3' \alpha \shortmod{d}\\ c_4\neq c_3'\alpha}} 
% 1.
% \end{align}
Next we write $d = d_0 d_1d_2$ where $(c_1',d)=d_1$, $d_0|d_1^\infty$, and $(d_1, d_2) = 1$. 
The various congruences in the sums then split into systems of congruences modulo $d_0 d_1$ and $d_2$, which we analyze separately.  

First consider the congruences modulo $d_0 d_1$.
Note that $d_1| c_2$ from the congruence $c_2 \equiv c_1' \alpha \pmod{d}$.  In addition, the condition $(c_1', c_3') = 1$ implies that $(c_3', d_0 d_1) = 1$ so that $\alpha \equiv c_4 \overline{c_3'} \pmod{d_0 d_1}$, which uniquely determines $\alpha \pmod{d_0 d_1}$ given $c_3'$ and $c_4$.  %Substituting this into $c_2 \equiv c_1' \alpha \pmod{d_0 d_1}$, this implies that 
Of course,
$\det(C') \equiv 0 \pmod{d_0 d_1}$.  By positivity, we ignore any extra information one might extract from the quadratic congruence in $\alpha$ modulo $d_0 d_1$.

Next consider the congruences modulo $d_2$.  
Here $c_2 \equiv c_1' \alpha \pmod{d_2}$ is uniquely solvable 
in terms of $\alpha$
via 
$\alpha \equiv c_2 \overline{c_1'} \pmod{d_2}$.
%$\alpha \equiv (c_2/d_1) \overline{(c_1'/d_1)} \pmod{d_2}$ 
%and $d_1 | (c_1', c_2)$. 
(Substituting this into the congruence on $c_3',c_4$ simply records  $\det(C') \equiv 0 \pmod{d_2}$). Finally, the quadratic congruence on $\alpha$ modulo $d_2$ then becomes $t_1 c_2^2 - t_2 c_2 (c_1' q) + t_3 (c_1' q)^2 \equiv 0 \pmod{d_2}$. 

Inserting this information into $\mathcal{N}''$, we obtain
\begin{equation}
\label{eq:NboundMiddleofProof}
\mathcal{N}''(X,W,T) \ll 
\sum_{\substack{(d,q) = 1, v|q^\infty \\ qdv\ll W \\ d_0 d_1 d_2 = d \\ d_0 | d_1^{\infty}, (d_1, d_2) = 1}}(qdv)^{1/2} 
\mathcal{M}(X, W, T, d, q),
\end{equation}
where
\begin{equation*}
    \mathcal{M}(X, W, T, d, q)=
\sum_{\substack{(c_1', c_3') = 1 \\ (c_1',d)=d_1\\ 1 \leq |c_1'|,  |c_3'| \ll X/q}}
\sum_{\substack{(c_2, c_4, q) = 1 \\ 1\leq |c_2|, |c_4| \ll X,  0 \neq \det(C') \ll W/q\\ \det(C') \equiv 0 \shortmod{v d} \\
0 \neq t_1 c_2^2 - t_2 c_2 (c_1' q) + t_3 (c_1' q)^2 \equiv 0 \shortmod{d_2} \\ d_1|  c_2 }} 
1.
\end{equation*}

Our final major task is to show
\begin{equation}
\label{eq:Mbound}
\mathcal{M}(X, W, T, d, q)
\ll \frac{W}{qvd} \frac{X^2}{q d_1 d_2} (\prod_{p^j || d_2} p^{\lfloor j/2 \rfloor}) (\|T \| Xdq )^{\varepsilon}.
\end{equation}
It is easy to check that \eqref{eq:Mbound}, when applied into \eqref{eq:NboundMiddleofProof}, gives a bound consistent with \eqref{eq:Nbound}, which will then complete the proof of Theorem \ref{thm.Bounding(t,C)}.

Now we turn to proving \eqref{eq:Mbound}.
Our first step is to
 take $c_1'$ and $c_2$ to the outside, and sum over $c_3'$ and $c_4$ first.  
 Along the way, we mildly simplify the expression by
dropping the condition $(c_2, c_4, q) = 1$ and relaxing the condition $(c_1', d) = d_1$ to say $d_1 | c_1'$.  This gives
\begin{equation*}
\mathcal{M}(X, W, T, d, q)
\leq 
\sum_{\substack{1\leq |c_1'| \ll X/q, 1 \leq |c_2| \ll X \\
d_1 \mid (c_1', c_2) \\ 0 \neq t_1 c_2^2 - t_2 c_2 (c_1' q) + t_3 (c_1' q)^2 \equiv 0 \shortmod{d_2} }
}
\mathcal{M}_1,
\end{equation*}
where
\begin{equation}
\label{c_3c_4sum}
    \mathcal{M}_1 = 
     \mathcal{M}_1(X,W,T,d,q,c_1', c_2) = 
    \sum_{\substack{1 \leq |c_3'| \ll X/q,  1\leq |c_4| \ll X \\ c_1' c_4 - c_3' c_2 \equiv 0 \shortmod{vd} \\ 0 < |c_1' c_4 - c_3' c_2| \ll W/q}} 1.
\end{equation}
We claim that
\begin{equation}
\label{eq:M1claimedbound}
    \mathcal{M}_1 \ll \frac{W}{qv d} \Big(1 + X\min\Big(\frac{(c_1', c_2)}{c_2}, \frac{(c_1', c_2)}{c_1' q} \Big)\Big).
\end{equation}
Supposing that $c_1' c_4 - c_3' c_2 = \det(C') = n$ is fixed, then $c_1'y-c_2x=n$ has a solution $(x,y) \in \mz^2$ only if $(c_1',c_2)|n$. Now suppose that
$c_1' y_0 - c_2 x_0 = n$, for some $y_0 \ll X$ and $x_0 \ll X/q$.  The set of all solutions to $c_1' y - c_2 x = n$ are of the form $x = x_0 + \frac{c_1'}{(c_1', c_2)}k$ and $y = y_0 + \frac{c_2}{(c_1',c_2)}k$, with $k \in \mz$.  Hence the number of pairs $(c_3', c_4)$ with $\det(C') = n$ is
\begin{equation*}
    \ll 1 + X\min\Big(\frac{(c_1', c_2)}{c_2}, \frac{(c_1',c_2)}{c_1' q}\Big).
\end{equation*}
Thus 
\begin{align*}
 \mathcal{M}_1 \ll   \sum_{\substack{0<|n|\ll W/q\\ 
    n  \equiv 0 \shortmod{vd}
    \\
    (c_1',c_2)|n}}
     \Big(1 + X\min\Big(\frac{(c_1', c_2)}{c_2}, \frac{(c_1', c_2)}{c_1' q} \Big)\Big),
    % \ll \frac{W}{qv d}  \Big(1 + X\min\Big(\frac{(c_1', c_2)}{c_2}, \frac{(c_1', c_2)}{c_1' q} \Big)\Big),
\end{align*}
which simplifies to \eqref{eq:M1claimedbound}
by dropping the condition $n\equiv 0 \pmod{(c_1', c_2)}$.  
Therefore,
\begin{equation}
\label{eq:c1c2c3c4sum}
\mathcal{M} \ll
\sum_{\substack{1 \leq |c_1'| \ll X/q, 1\leq |c_2| \ll X \\ d_1 | (c_1', c_2) \\ 0 \neq t_1 c_2^2 - t_2 c_2 c_1' q + t_3 (c_1' q)^2 \equiv 0 \shortmod{d_2}}}
\frac{WX}{qv d}   \min\Big(\frac{(c_1', c_2)}{c_2}, \frac{(c_1', c_2)}{c_1' q}\Big).
\end{equation}
We proceed to show \eqref{eq:Mbound}.
% For the term with ``$1$", this can be seen as follows.  Factoring out $d_1$ from $c_1'$ and $c_2$, and using $(d_1, d_2) = 1$, we obtain
% \begin{equation}
% \sum_{\substack{1 \leq |c_1'| \ll X/q, 1\leq |c_2| \ll X \\ d_1 | (c_1', c_2) \\ 0 \neq t_1 c_2^2 - t_2 c_2 c_1' q + t_3 (c_1' q)^2 \equiv 0 \shortmod{d_2}}} 1
% \leq 
% \sum_{\substack{1 \leq |a| \ll \frac{X}{qd_1}, 1\leq |b| \ll \frac{X}{d_1} 
% \\
% 0 \neq t_1 b^2 - t_2q a b  + t_3q^2 a^2 \equiv 0 \shortmod{d_2}}} 1
% \end{equation}
Writing $c_1' = d_1 a$ and $c_2 = d_1 b$, 
 and then factoring out $g = \gcd(a,b)$,
we need to estimate
 \begin{multline}
 \label{eq:MboundMiddlePart}
\sum_{\substack{1 \leq |a| \ll \frac{X}{qd_1}, 1\leq |b| \ll \frac{X}{d_1} 
\\
0 \neq t_1 b^2 - t_2 q a b + t_3 q^2 a ^2 \equiv 0 \shortmod{d_2}}} 
\min\Big(\frac{(a,b)}{aq}, \frac{(a,b)}{b}\Big)
\\
\leq
\sum_{g \ll \frac{X}{d_1}} 
\sum_{\substack{1 \leq |a| \ll \frac{X}{qgd_1}, 1\leq |b| \ll \frac{X}{gd_1} 
\\
0 \neq t_1 b^2 - t_2 q a b + t_3q^2 a ^2 \equiv 0 \shortmod{\frac{d_2}{(d_2, g^2)}}}} 
\min\Big(\frac{1}{aq}, \frac{1}{b}\Big).
 \end{multline}
To handle the range $aq \leq b$, we break up the sum over $b$ into dyadic segments, $b \asymp B$.  
By Lemma \ref{lem.quadraticformDivisor}, and since $a$ and $b$ are nonzero, the function $(a, b) \mapsto t_1 b^2 -  t_2 q ba + t_3 q^2 a^2$ is $O((\|T\| X)^{\varepsilon})$-to-one.  In turn, the total number of integers in the codomain of this function that are divisible by $\frac{d_2}{(d_2, g^2)}$ is $O(\frac{B^2}{q \frac{d_2}{(d_2, g^2)}})$.  
Hence \eqref{eq:MboundMiddlePart} is
\begin{equation*}
\ll \sum_{g \ll \frac{X}{d_1}} 
\sum_{\substack{B \text{ dyadic} \\ 1 \ll B \ll \frac{X}{gd_1}}} \frac{1}{B}
\frac{B^2}{q \frac{d_2}{(d_2, g^2)}}
\ll \sum_{g \ll \frac{X}{d_1}} 
\frac{X}{g q d_1 \frac{d_2}{(d_2, g^2)}}
\ll \frac{X}{q d_1 d_2} \sum_{g \ll \frac{X}{d_1}} \frac{(d_2, g^2)}{g}.
\end{equation*}
The same bound arises from the range $b \leq aq$.
Using Rankin's trick, one can show
\begin{equation*}
    \sum_{n \leq X} \frac{(d,n^2)}{n} \ll (Xd)^{\varepsilon} \prod_{p^j || d} p^{\lfloor j/2 \rfloor}.
\end{equation*}
Inserting this bound into
\eqref{eq:c1c2c3c4sum}
then completes the proof of \eqref{eq:Mbound}.
\end{proof}

\section{Eisenstein series}
This section has three main goals.  In Section \ref{section:automorphicEisenstein}, we define the automorphic Eisenstein series for $\Sp_4(\mr)$ and state some of its fundamental properties, including connections between the minimal parabolic and Siegel parabolic Eisenstein series.  In Section \ref{section:incompleteEisensteinSeries}, we consider incomplete Eisenstein series associated to the minimal parabolic, and construct a family of such functions which are identically $1$, generalizing Nelson's construction for $\GL_2$.  Finally, in Section \ref{section:RankinSelbergBound}, we give a simple bound on the average size of Fourier coefficients of cusp forms, which has its basic origin in the analytic properties of the Eisenstein series.

\subsection{Automorphic Eisenstein series}
\label{section:automorphicEisenstein}
For $g \in \Sp_4(\mr)$, suppose the $M_0$-component of $g$ is $\mathrm{diag}(y_1, y_2, y_1^{-1}, y_2^{-1})$, in which case we
let $I_0(g, \nu_1, \nu_2) = y_1^{\nu_1 + 2} y_2^{2 \nu_2 - \nu_1 + 1}$.  Similarly, define $I_{\alpha}(g, u) = (y_1 y_2)^{u+\frac32}$, and note that $I_{\alpha}(g,u) = \det(\overline{m_{\alpha}}(g))^{u+\frac32}$.

The minimal parabolic Eisenstein series is defined by
\begin{equation*}
    E_0(g, \nu_1, \nu_2) = \sum_{\gamma \in (P_0 \cap \Gamma) \backslash \Gamma} I_0(\gamma g, \nu_1, \nu_2),
\end{equation*}
where the series converges absolutely for
\begin{equation}
\label{eq:minimalEisensteinSeriesAbsoluteConvergence}
\mathrm{Re}(2\nu_2 -  \nu_1) > 1,
\qquad
\mathrm{Re}(\nu_1 - \nu_2) > 1/2.
\end{equation}
%Remark.  Note that \eqref{eq:minimalEisensteinSeriesAbsoluteConvergence} implies that $\mathrm{Re}(\nu_2) > 3/2$ and $\mathrm{Re}(\nu_1) > 2$.

Let $f$ be an automorphic function on $\SL_2(\mz)$.
The Eisenstein series for the Siegel parabolic subgroup $P_{\alpha}$ induced from $f$ is defined by
\begin{equation*}
    E_{\alpha}(g,\nu,f) = 
    \sum_{\gamma \in (P_{\alpha} \cap \Gamma) \backslash \Gamma}
    f(\overline{m}_{\alpha}(\gamma g)) I_{\alpha}(\gamma g, \nu),
\end{equation*}
We mainly need the case where $f= E_s$ is the 
classical Eisenstein series on $\mathrm{SL}_2(\mz) \backslash \mh$, namely
\begin{equation*}
    E_s(z) = \sum_{\tau \in \Gamma_{\infty}^{(1)} \backslash \Gamma^{(1)}} \mathrm{Im}(\tau z)^s,
\end{equation*}
where $\Gamma^{(1)} = \mathrm{SL}_2(\mz)$ and $\Gamma_{\infty}^{(1)} = \{ \pm (\begin{smallmatrix} 1 & n \\ 0 & 1 \end{smallmatrix}): n \in \mz \}$.  In this case, the series $E_{\alpha}(g, \nu, E_s)$ converges absolutely for
\begin{equation}
\label{eq:SiegelEisensteinSeriesAbsoluteConvergence}
\mathrm{Re}(s) > 1, \qquad \mathrm{Re}(\nu-s) > 1/2.
\end{equation}
\begin{mylemma}
\label{lemma:EalphaVsE0}
    Suppose \eqref{eq:SiegelEisensteinSeriesAbsoluteConvergence} holds.  Then
    \begin{equation*}
        E_{\alpha}(g, \nu, E_s) = E_0(g, s+\nu-1/2, \nu).
    \end{equation*}
    The identity extends meromorphically to $(s,\nu) \in \mc^2$.
\end{mylemma}
Remark.  This is stated without proof in \cite[Prop 7.1]{Hangman}, and is not difficult to verify.

We can also go the other direction, to some extent.  Let $h: \mathrm{GL}_2(\mr) \rightarrow \mc$ be a smooth function defined in Iwasawa coordinates by $h((\begin{smallmatrix} 1 & u \\ & 1 \end{smallmatrix}) (\begin{smallmatrix} y_1 & \\ & y_2 \end{smallmatrix}) k ) = h((\begin{smallmatrix} 1 & u \\ & 1 \end{smallmatrix}))$, with the property that
$\sum_{n \in \mz} h((\begin{smallmatrix} 1 & u +n \\ & 1 \end{smallmatrix})) = 1$, and such that the sum is locally finite (meaning, for each $u$ there are finitely many $n \in \mz$ such that $h((\begin{smallmatrix} 1 & u +n \\ & 1 \end{smallmatrix})) \neq 0$).
With such $h$ we have the following.
\begin{myprop}
\label{prop:minimalEisensteinSeriesInTermsOfJLSPoincareSeries}
Suppose \eqref{eq:minimalEisensteinSeriesAbsoluteConvergence} holds.  
Let $\nu = \nu_2$ and $s = \nu_1 - \nu_2 + 1/2$, and 
let $\phi: \GL_2^{+}(\mr)/\mathrm{SO}_2(\mr) \rightarrow \mc$ be defined by
\begin{equation}
\label{eq:phidefminimalEisensteinSeriesInTermsOfJLSPoincareSeries}
\phi(g) = h(g) \mathrm{Im}(g)^s 
\det(g)^{\nu + \frac32},
%I_{\alpha}(g, \nu),
\end{equation}
where $\mathrm{Im}((\begin{smallmatrix} 1 & u \\ & 1 \end{smallmatrix}) (\begin{smallmatrix} y_1 & \\ & y_2 \end{smallmatrix}) k)^s = (y_1/y_2)^s$.
Then
\begin{equation*}
E_0(g, \nu_1, \nu_2) = P_0(g, \phi),
\end{equation*}
where $P_0$ was defined in \eqref{eq:P0defVariant} (or \eqref{eq:EincompleteDef}).
\end{myprop}
\begin{proof}
Recall from Lemma \ref{lemma:EalphaVsE0} that  $E_0(g, \nu_1, \nu_2) = E_{\alpha}(g, \nu, E_s)$.  Then
\begin{equation*}
E_0(g, \nu_1, \nu_2) = 
\sum_{\gamma \in (P_{\alpha} \cap \Gamma) \backslash \Gamma}
    E_s(\overline{m}_{\alpha}(\gamma g)) I_{\alpha}(\gamma g, \nu)
    % = 
    % \sum_{\gamma \in (P_{\alpha} \cap \Gamma) \backslash \Gamma}
    % E_s(\overline{m}_{\alpha}(\gamma g)) \det(\overline{m}_{\alpha}(\gamma g))^{\nu+\frac32}
    .
\end{equation*}
View $h$ as a function on $\mh$ via $h(x+iy) = h((\begin{smallmatrix} 1 & x\\ & 1 \end{smallmatrix}))$.
Using the definition of $E_s(z)$, and since $\sum_{\tau \in \Gamma_{\infty}^{(1)}} h( \tau w) = 1$ for all $w \in \mh$, we have
\begin{equation*}
E_s(z) = \sum_{\gamma \in \Gamma_{\infty}^{(1)} \backslash \Gamma^{(1)}} \mathrm{Im}(\gamma z)^s
= \sum_{\gamma \in \Gamma_{\infty}^{(1)} \backslash \Gamma^{(1)}}
\mathrm{Im}(\gamma z)^s
\sum_{\tau \in \Gamma_{\infty}^{(1)}} h(\tau \gamma z).
\end{equation*}
Hence $E_s(z) = \sum_{\tau \in \mathrm{SL}_2(\mz)} \mathrm{Im}(\tau z)^s h(\tau z)$.  Therefore,
\begin{equation*}
E_0(g, \nu_1, \nu_2) = 
\sum_{\gamma \in (P_{\alpha} \cap \Gamma) \backslash \Gamma}
\sum_{\tau \in \mathrm{SL}_2(\mz)} \mathrm{Im}(\tau \overline{m}_{\alpha}(\gamma g))^s h(\tau \overline{m}_{\alpha}(\gamma g))
   I_{\alpha}(\gamma g, \nu).
\end{equation*}
Recalling from Lemma \ref{lemma:malphaForGL2vsSp4commutingFormula} that $\tau \overline{m}_{\alpha}(\gamma g) = \overline{m}_{\alpha}(m(\tau) \gamma g)$, 
and noting that $I_{\alpha}(m(\tau) g, \nu) = I_{\alpha}(g, \nu)$ for all $\tau \in \SL_2(\mz)$, we obtain
\begin{equation*}
E_0(g, \nu_1, \nu_2) = 
\sum_{\gamma \in (P_{\alpha} \cap \Gamma) \backslash \Gamma}
\sum_{\tau \in \mathrm{SL}_2(\mz)} \mathrm{Im}(\overline{m}_{\alpha}(m(\tau) \gamma g))^s h(\overline{m}_{\alpha}(m(\tau) \gamma g))
   I_{\alpha}( m(\tau) \gamma g, \nu).
\end{equation*}
Next we argue that $m(\tau) \gamma$ runs over a complete set of coset representatives for $(N_{\alpha} \cap \Gamma) \backslash \Gamma$.  This is easy because every element of $P_{\alpha} \cap \Gamma$ can be expressed as $(\begin{smallmatrix} \tau & * \\ & \tau^{-t} \end{smallmatrix})$  or as $(\begin{smallmatrix} \tau & * \\ & \tau^{-t} \end{smallmatrix}) \cdot \mathrm{diag}(-1,1,-1,1)$ with $\tau \in \mathrm{SL}_2(\mz)$.  Note that
$(\begin{smallmatrix} \tau & * \\ & \tau^{-t} \end{smallmatrix}) = (\begin{smallmatrix} 1 & X \\ & 1 \end{smallmatrix})  m(\tau)$ for some $X \in \M_2^{\mathrm{sym}}(\mz)$.
Thus
\begin{equation}
\label{eq:E0formula}
E_0(g, \nu_1, \nu_2) = 
\sum_{\gamma \in (N_{\alpha} \cap \Gamma) \backslash \Gamma}
\mathrm{Im}(\overline{m}_{\alpha}(\gamma g))^s h(\overline{m}_{\alpha}(\gamma g))
   I_{\alpha}( \gamma g, \nu). 
\end{equation}
To complete the proof, recall that $I_{\alpha}(g, \nu) = \det(\overline{m}_{\alpha}(g))^{\nu + \frac32}$, and compare \eqref{eq:E0formula} with \eqref{eq:P0defVariant}.
\end{proof}

Next we summarize some fundamental properties of the minimal parabolic Eisenstein series.
We define the completed Eisenstein series $E_0^*(g,\nu_1, \nu_2)$ to be
\begin{equation*}
\zeta^*(\nu_1 + 1) \zeta^*(2 \nu_2 + 1) \zeta^*(2 \nu_2 - \nu_1 + 1) \zeta^*(2 \nu_1 - 2 \nu_2 + 1) E_0(g, \nu_1, \nu_2).
\end{equation*}
Following the notation of \cite[Section 2]{Hangman},
the Weyl group $W$ is 
$$W = \{1, s_{\alpha}, s_{\beta}, s_{\alpha} s_{\beta}, s_{\beta} s_{\alpha}, s_{\alpha} s_{\beta} s_{\alpha}, s_{\beta} s_{\alpha} s_{\beta}, w_0 = s_{\alpha} s_{\beta} s_{\alpha} s_{\beta} \},$$ 
which is the dihedral group of order $8$.  
For our purposes, 
it is convenient to record the action of the Weyl group elements on the tuple $(\nu_1, \nu_2)$:
\begin{table}[h!]
\centering
\begin{tabular}{c||c|c|c|c|c|c|c|c}
$w$  & $1$ &  $s_{\alpha}$ & $s_{\beta}$ & $s_{\alpha} s_{\beta}$      \\
\hline
$w(\nu_1, \nu_2)$ & $(\nu_1, \nu_2)$ &  $(2 \nu_2 - \nu_1, \nu_2)$ & $(\nu_1, \nu_1 - \nu_2)$ & $(2 \nu_2 - \nu_1, -\nu_1 + \nu_2)$ \\
\hline \hline
$w$  & $s_{\beta} s_{\alpha}$ &  $s_{\alpha} s_{\beta} s_{\alpha}$ & $s_{\beta} s_{\alpha} s_{\beta}$ &  $w_0$  \\
\hline
$w(\nu_1, \nu_2)$ &  $(\nu_1 - 2 \nu_2, \nu_1 - \nu_2)$  & $(-\nu_1, \nu_2 - \nu_1)$ & $(\nu_1 - 2 \nu_2, -\nu_2)$ & $(-\nu_1, - \nu_2)$
\end{tabular}
\end{table}

The main properties of $E_0^*$ are stated as follows.  
\begin{mytheo}
\label{thm:EisensteinSeriesMinimalProperties}
The completed Eisenstein series $E_0^*(g, \nu_1, \nu_2)$, initially defined on \eqref{eq:minimalEisensteinSeriesAbsoluteConvergence}, extends to a meromorphic function on $\mc^2$.  It satisfies the functional equation
\begin{equation}
\label{eq:EisensteinFE}
E_0^*(g, w(\nu_1, \nu_2)) = E_0^*(g, \nu_1, \nu_2)
\end{equation}
for all $w$ in the Weyl group.  It has poles along the $12$ lines:
\begin{align}
\label{eq:Lines}
&L_1: \nu_1 - 2 \nu_2 = 0,& \quad &L_4: \nu_1 = 0,& \quad &L_7: \nu_2 = 0,& \quad &L_{10}: - 2 \nu_1 + 2 \nu_2 = 0, \\
&L_2: \nu_1 - 2 \nu_2 = -1,& \quad &L_5: \nu_1 = -1,& \quad &L_8: \nu_2 = -1/2,& \quad &L_{11}: - 2 \nu_1 + 2 \nu_2 = -1, \nonumber \\
&L_3: \nu_1 - 2 \nu_2 = 1,& \quad &L_6: \nu_1 = 1,& \quad &L_9: \nu_2 = 1/2,& \quad &L_{12}: - 2 \nu_1 + 2 \nu_2 = 1 \nonumber.
\end{align}
The residue of $E_0^*(g, \nu_1, \nu_2)$ at $(\nu_1, \nu_2) = (2,3/2)$ (on the intersection of $L_2$ and $L_{11}$) is a nonzero constant function of $g$.  The only other constant residue is at $(-2, -3/2)$.
\end{mytheo}
A good reference for the general properties of Eisenstein series, including the meromorphic continuation and functional equation, is \cite{Shahidi}.  For the coordinates of the $12$ polar lines, we use the well-known fact that the poles of the Eisenstein series are located among the poles of the constant term.  The constant term of $E_0^*$ is given in convenient and explicit form in \cite[Theorem 1.1]{Hangman}.

\subsection{Incomplete Eisenstein series}
\label{section:incompleteEisensteinSeries}
Consider a function of the form
\begin{equation}
\label{eq:E0incompleteDefinition}
E_0^*(g) = 
\frac{1}{(2 \pi i)^2} \int_{(\sigma_1)} \int_{(\sigma_2)} 
E_0^*(g, \nu_1, \nu_2) H(\nu_1, \nu_2) d\nu_1 d \nu_2,
\end{equation}
where $H$ is a holomorphic function bounded in vertical strips, and where the integrals are over $\mathrm{Re}(\nu_i) = \sigma_i$, in the region \eqref{eq:minimalEisensteinSeriesAbsoluteConvergence}.
Let
\begin{equation*}
H^*(\nu_1, \nu_2) = H(\nu_1, \nu_2) \zeta^*(\nu_1 + 1) \zeta^*(2 \nu_2 + 1) \zeta^*(2 \nu_2 - \nu_1 + 1) \zeta^*(2 \nu_1 - 2 \nu_2 + 1),
\end{equation*}
so that
\begin{align}
\label{eq:E0incompleteVariant2}
E_0^*(g) = 
\frac{1}{(2 \pi i)^2} \int_{(\sigma_1)} \int_{(\sigma_2)} 
E_0(g, \nu_1, \nu_2) H^*(\nu_1, \nu_2) d\nu_1 d \nu_2.
%\\
%= 
%\sum_{\gamma \in (P_0 \cap \Gamma) \backslash \Gamma} 
%\frac{1}{(2 \pi i)^2} \int_{(\sigma_1)} \int_{(\sigma_2)} 
%I_0(\gamma g, \nu_1, \nu_2) H(\nu_1, \nu_2) d\nu_1 d \nu_2.
\end{align}

\begin{myprop}
\label{prop:unfoldingWithIncompleteEisensteinSeriesIdenticallyOne}
Let $f \in L^{\infty}(\Gamma \backslash \mh_2)$.  With the identification $Z \leftrightarrow g$ via \eqref{eq:SigelidentificationUpperHalfPlane}, we have
\begin{multline*}
\int_{\Gamma \backslash \mh_2} f(Z) E_0^*(g) d \mu 
=
\int_{(\sigma_1)} \int_{(\sigma_2)}  
H^*(\nu_1, \nu_2)
\\
\times \int_{\substack{x_1, x_2, x_3 \in \mz \backslash \mr \\ 
r_1, r_2 \in \mr_{>0}, u \in \mr}}
\Big(
f(X+ i Y) 
h((\begin{smallmatrix} 1 & u \\ & 1 \end{smallmatrix}))
(r_1/r_2)^{\frac{\nu_1 - \nu_2 + \frac12}{2}}(r_1 r_2)^{\frac{\nu_2}{2} + \frac34}
 \frac{ dX  du dr_1 dr_2}{r_1^3 r_2^2} 
\Big)
 \frac{d\nu_1 d \nu_2}{(2 \pi i)^2},
\end{multline*}
where $\sigma_1$ and $\sigma_2$ are such that \eqref{eq:minimalEisensteinSeriesAbsoluteConvergence} holds.
\end{myprop}
\begin{proof}
By Proposition \ref{prop:minimalEisensteinSeriesInTermsOfJLSPoincareSeries}, we have
\begin{equation*}
\int_{\Gamma \backslash \mh_2} f(Z) E_0^*(g) d \mu
=
\frac{1}{(2 \pi i)^2} \int_{(\sigma_1)} \int_{(\sigma_2)} 
\int_{\Gamma \backslash \mh_2} f(Z)
P_0(g, \phi) d\mu
 d\nu_1 d \nu_2,
\end{equation*}
where $\phi(g) = h(g) \mathrm{Im}(g)^s I_{\alpha}(g, \nu)$.  Next we apply
\eqref{eq:PoincareSeriesUnfolding} with $\phi(R) = h((\begin{smallmatrix} 1 & u \\ & 1 \end{smallmatrix})) (r_1/r_2)^{s/2} (r_1 r_2)^{\frac{\nu}{2} + \frac{3}{4}}$, 
giving the claimed formula.
\end{proof}

\begin{mytheo}
\label{thm:Einc=1}
There exists a choice of acceptable function $H$ so that $E_0^*(g) =1$ for all $g \in \mathrm{Sp_4}(\mr)$.  Specifically, any $H$ satisfying the following properties has $E_0^*(g) = 1$:
\begin{enumerate}
    \item $H(\nu_1, \nu_2)$ is holomorphic in $\mc^2$, and bounded in any Cartesian product of vertical strips,
    \item $H(\nu_1, \nu_2) = P(\nu_1, \nu_2) H_0(\nu_1, \nu_2)$, where $P(x,y)$ is the polynomial defined by 
    \begin{equation*}
P(\nu_1, \nu_2) = \prod_{w \in W} \ell_1(w(\nu_1, \nu_2)) \cdot \ell_{10}(w(\nu_1, \nu_2)), 
\end{equation*}
where $\ell_1(x,y) = x-2y$ and $\ell_{10}(x,y) = -2x+2y$,
and where $H_0$ is holomorphic on $\mc^2$,
\item $H(w(\nu_1, \nu_2)) = \det(w) H(\nu_1, \nu_2)$ for all $w \in W$,
\item $H^*(2,3/2) \cdot \mathrm{Res}_{(\nu_1, \nu_2) = (2,3/2)} E_0(g, \nu_1, \nu_2) = 1$.
\end{enumerate}
\end{mytheo}
Remarks.  We have made no attempt to provide the most general class of test functions in this theorem.  Rather, we imposed conditions to make the proof as simple as possible.
The choice of notation $\ell_1$ and $\ell_{10}$ comes from the fact that $\ell_i$ vanishes along the line $L_i$ (in \eqref{eq:Lines}) for $i =1, 10$.
\begin{proof}
We begin by showing that the collection of functions satisfying (1)--(4) is nonempty.  For instance, let $H_{00}(\nu_1, \nu_2) = \exp(\nu_1^2 + \nu_2^2)$, and define
\begin{equation*}
H_{0}(\nu_1, \nu_2) = \sum_{w \in W} \det(w) H_{00}(w(\nu_1, \nu_2)),
\end{equation*}
which satisfies $H_0(w(\nu_1, \nu_2)) = \det(w) H_0(\nu_1, \nu_2)$.  By construction, $P$ is invariant under $W$, so $H$ satisfies (2) and (3).  It is easy to see that $H_{00}$ is entire with rapid decay on any vertical strips $\mathrm{Re}(\nu_i) = \sigma_i$, for $-T \leq \sigma_1, \sigma_2 \leq T$.  It is also not difficult to check that $H_{00}(w(\nu_1, \nu_2))$ has the same decay properties for any $w \in W$.  This follows because the change of variables induced from $w$ has determinant $\pm 1$ and therefore does not change the discriminant of the positive definite quadratic form $x^2 + y^2$.  
Hence $H$ satisfies (1).  Finally, for (4), a brute force computer calculation gives $H_0(2,3/2) = 851.215\dots \neq 0$.  
Moreover, $P(2,3/2) \neq 0$ (this will be easily explained in the proof below, in view of Figure \ref{fig:polarlinesAll}).
Hence $H(2,3/2) \neq 0$ and $H^*(2,3/2) \neq 0$, so rescaling $H$ will achieve (4).

Now we proceed to show that if $H$ satisfies (1)--(4), then $E_0^*(g)$ is identically $1$.
We begin from the definition \eqref{eq:E0incompleteDefinition}, taking say $\sigma_1 = 3$ and $\sigma_2 = 2.1$, which is in the region of absolute convergence.  For ease of reference, see Figure \ref{fig:polarlinesAll} for a graph of the $12$ polar lines from Theorem \ref{thm:EisensteinSeriesMinimalProperties}.
\begin{figure}[ht]
  \centering
  \begin{tikzpicture}[scale=1.3]
    % Clipping region: x in [-2.2, 2.2], y in [-2, 2]
    \clip (-2.2,-2) rectangle (2.2,2);

    % Axes
    \draw[->] (-2.2,0) -- (2.2,0);
    \draw[->] (0,-2) -- (0,2);

    % Axis labels
    \node at (2.1,-0.15) {$\nu_1$};
    \node at (-0.15,1.9) {$\nu_2$};

    % x - 2y = c lines => y = 0.5x - c/2
    \foreach \c/\col in {-1/blue, 0/blue, 1/blue} {
      \draw[thick, \col] (-2.2,{0.5*(-2.2)-\c/2}) -- (2.2,{0.5*(2.2)-\c/2});
    }

    % -2x + 2y = c => y = x + c/2
    \foreach \c/\col in {-1/red, 0/red, 1/red} {
      \draw[thick, \col] (-2.2,{-2.2 + \c/2}) -- (2.2,{2.2 + \c/2});
    }

    % Horizontal lines: y = -0.5, 0, 0.5
    \foreach \y in {-0.5,0,0.5} {
      \draw[green!70!black, dashed] (-2.2,\y) -- (2.2,\y);
    }

    % Vertical lines: x = -1, 0, 1
    \foreach \x in {-1,0,1} {
      \draw[orange, dashed] (\x,-2) -- (\x,2);
    }
  \end{tikzpicture}
  \caption{Polar lines of $E_0^*(g, \nu_1, \nu_2)$}
\label{fig:polarlinesAll}
\end{figure}
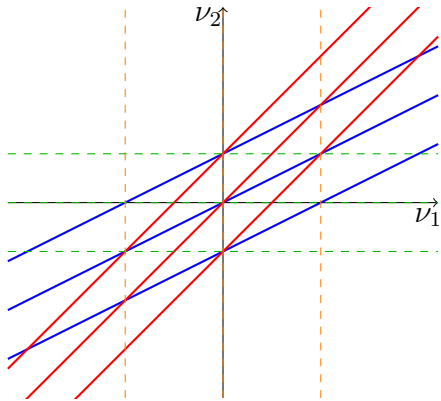
It is convenient to choose $H$ to vanish along some of these lines, which motivates the presence of $P$.
Here $P$ is Weyl-group invariant and vanishes to order $4$ along the polar lines through the origin (namely, $L_1$, $L_4$, $L_7$, and $L_{10}$). The fact that $P(2,3/2) \neq 0$ can be seen from the figure, since the only two lines passing through $(2,3/2)$ are $L_2$ and $L_{11}$, which do not pass through the origin (the Weyl group fixes the origin).  With $H(\nu_1, \nu_2) = P(\nu_1, \nu_2) \cdot H_0(\nu_1, \nu_2)$, then $E_0^*(g, \nu_1, \nu_2) \cdot P(\nu_1, \nu_2)$ does not have any polar lines which pass through the origin.  See Figure \ref{fig:polarlinesAlltrimmed} for a simplified figure displaying only the relevant lines, and also only in the first quadrant:
\begin{figure}[ht]
  \centering
  \begin{tikzpicture}[scale=1.3]
% Clip to first quadrant: x in [0,3], y in [0,2.5]
    \clip (0,0) rectangle (3,2.25);

    % Axes arrows only in first quadrant
    \draw[->] (0,0) -- (3,0) node[right] {$\nu_1$};
    \draw[->] (0,0) -- (0,2.5) node[above left] {$\nu_2$};

    % x - 2y = c lines for c = -1, 1 (exclude c=0)
    \foreach \c/\col in {-1/blue,  1/blue} {
      \draw[thick, \col] (-2.2,{0.5*(-2.2)-\c/2}) -- (3,{0.5*(3)-\c/2});
    }

    % -2x + 2y = c => y = x + c/2
    \foreach \c/\col in {-1/red, 1/red} {
      \draw[thick, \col] (-2.2,{-2.2 + \c/2}) -- (3,{3 + \c/2});
    }

    % Horizontal lines: y = -0.5, 0, 0.5
    \foreach \y in {-0.5,0.5} {
      \draw[green!70!black, dashed] (-2.2,\y) -- (3,\y);
    }

    % Vertical lines: x = -1, 0, 1
    \foreach \x in {-1,1} {
      \draw[orange, dashed] (\x,-3) -- (\x,3);
    }
  \end{tikzpicture}
  \caption{Polar lines of $P(\nu_1, \nu_2) \cdot E_0^*(g, \nu_1, \nu_2)$}
\label{fig:polarlinesAlltrimmed}
\end{figure}
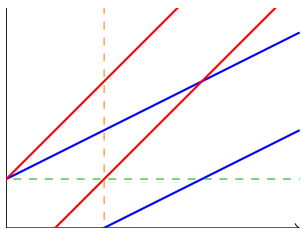

Now we shift contours in \eqref{eq:E0incompleteDefinition} as follows.  We will always move one contour at a time so that we can view the double integral as iterated single integrals, and thereby apply single variable complex analysis.
We begin by fixing $\sigma_2 = 2.1$ and shift $\sigma_1$ to $\sigma_1 = 1.7$, crossing a pole along the line $L_{11}$ only (precisely, at $\nu_1 = \nu_2 + 1/2$, with $\sigma_1 = 2.6$.).  Next we fix $\sigma_1$ at $1.7$ and shift $\sigma_2$ from $2.1$ to $\sigma_2 = 1.3$, passing a pole along the line $L_2$ only, at $\nu_2 = \frac12 +  \frac{\nu_1}{2}$ with $\sigma_2 = 1.35$.
Next we can adjust the lines of integration without passing any poles, taking $\sigma_1 = 1.2$ and $\sigma_2 = 0.8$.
We follow this by moving $\sigma_1$ to $0.8$, crossing a pole along $L_6$ only.  Next we shift $\sigma_2$ to $0.4$, crossing a pole along $L_9$ only.  Finally, we move the contour to $\sigma_1 = \sigma_2 = 0$ without crossing any more poles.
In total, we have
\begin{equation*}
E_0^*(g) = I_{0,0} + R_{11} + R_{2} + R_{6} + R_{9},
\end{equation*}
where
\begin{equation*}
I_{0,0} = \frac{1}{(2 \pi i)^2} \int_{(0)} \int_{(0)} 
E_0^*(g, \nu_1, \nu_2) H(\nu_1, \nu_2) d\nu_1 d \nu_2,
\end{equation*}
and the terms $R_i$ are the residue terms along the lines $L_i$.
Applying a change of variables by $w\in W$ with $\det(w) = -1$, using 
the functional equation of Eisenstein series \eqref{eq:EisensteinFE}, and
using property (3) shows $I_{0,0} = - I_{0,0}$ so $I_{0,0} = 0$.

Next we analyze the residual terms.  We will show $R_{11} = 1$ and $R_2 = R_6 = R_9 = 0$.
We have
\begin{equation*}
R_{11} = \frac{1}{2 \pi i} \int_{(2.1)} 
\mathrm{Res}_{\nu_1 = 1/2 + \nu_2}(
E_0^*(g, \nu_1, \nu_2) H(\nu_1, \nu_2))  d \nu_2,
\end{equation*}
with similar definitions for the other residual terms to be considered in turn.  To further analyze $R_{11}$, we shift contours to the line $\sigma_2 = 0$, crossing poles at $\sigma_2 = 3/2$ and $\sigma_2 = 1/2$.  The Eisenstein series has a constant residue on the former pole.  The latter pole occurs at $(1,1/2)$ which is along the line $L_1$.  The Eisenstein series has a triple pole at this point, while $H$ vanishes to order at least $4$, so the latter singularity is removable.  Hence $R_{11} = c + R_{11}'$, where
$c = H(2,3/2) \cdot \mathrm{Res}_{(\nu_1, \nu_2) = (2,3/2)} E_0^*(g, \nu_1, \nu_2)=1$ (by (4)) and 
\begin{equation*}
R_{11}'= \frac{1}{2 \pi i} \int_{(0)} 
\mathrm{Res}_{\nu_1 = 1/2 + \nu_2}(
E_0^*(g, \nu_1, \nu_2) H(\nu_1, \nu_2))  d \nu_2.
\end{equation*}
We will derive similar formulas for the other residual terms, and finally analyze each such term.  For the term $R_{2}$, we shift contours to $\sigma_1 = 0$; along the way we encounter poles at $(1,1)$ (of order $2$) and $(0,1/2)$ (of order $3$), however $P$ vanishes to order at least $4$ at both of these points, so they do not contribute any residues.  Hence
\begin{equation*}
R_2 = \frac{1}{2 \pi i} \int_{(0)} 
\mathrm{Res}_{\nu_2 = \frac12 + \frac{\nu_1}{2}}(
E_0^*(g, \nu_1, \nu_2) H(\nu_1, \nu_2))  d \nu_1.
\end{equation*}
For $R_6$ we shift contours to $\sigma_2 = 1/2$; 
the only pole along this path is a removable singularity at $(1,1/2)$, so 
\begin{equation*}
R_{6}= \frac{1}{2 \pi i} \int_{(1/2)} 
\mathrm{Res}_{\nu_1 = 1}(
E_0^*(g, \nu_1, \nu_2) H(\nu_1, \nu_2))  d \nu_2.
\end{equation*}
Finally, for $R_{9}$ we directly shift contours to $\sigma_1 = 1/2$ without crossing any poles, giving
\begin{equation*}
R_{9}= \frac{1}{2 \pi i} \int_{(1/2)} 
\mathrm{Res}_{\nu_2 = 1/2}(
E_0^*(g, \nu_1, \nu_2) H(\nu_1, \nu_2))  d \nu_1.
\end{equation*}

To complete the proof, we argue that $R_{11}' = R_2 = R_6 = R_9 = 0$.  Let us examine $R_{11}'$ in detail.
The Weyl element $w_{11} = s_{\beta} s_{\alpha} s_{\beta}$ maps the line $L_{11}$ to itself.  Letting $F(\nu_1, \nu_2) = E_0^*(g, \nu_1, \nu_2) H(\nu_1, \nu_2)$ (suppressing $g$ in the notation), we have
\begin{equation*}
F(\nu_1 - 2 \nu_2, - \nu_2) = F(w_{11}(\nu_1, \nu_2)) =
- F(\nu_1, \nu_2), 
\end{equation*}
since $E_0^*$ is invariant under $w_{11}$, while $H$ changes sign, since $\det(w_{11}) = -1$.  Hence
\begin{equation*}
\mathrm{Res}_{\nu_1 = 1/2 + \nu_2}(
F(\nu_1, \nu_2))
= \mathrm{Res}_{u=0} F(1/2 + \nu_2 + u, \nu_2)
= - \mathrm{Res}_{u=0} F(1/2 - \nu_2 + u, - \nu_2).
\end{equation*}
Hence $R_{11}$ is the integral of an odd function, so it therefore vanishes.

The case of $R_2$ is very similar to that of $R_{11}$, using that $s_{\alpha} s_{\beta} s_{\alpha}$ fixes $L_2$.  Precisely, 
for $R_2$, we use $F(\nu_1, \nu_2) = -F(-\nu_1, \nu_2 - \nu_1)$ so 
$$\mathrm{Res}_{\nu_2 = \tfrac12 + \tfrac{\nu_1}{2}} F(\nu_1, \nu_2) = 
\mathrm{Res}_{u=0} F(\nu_1, \tfrac12 + \tfrac{\nu_1}{2} + u) 
= -\mathrm{Res}_{u=0} F(-\nu_1, \tfrac12 - \tfrac{\nu_1}{2} + u).$$
Hence $R_2$ is an integral of an odd function, so it vanishes.

For $R_6$ and $R_9$, we will use that
$s_{\beta}$ fixes $L_6$, and $s_{\alpha}$ fixes $L_9$. 
For $R_6$, we have that $F(\nu_1, \nu_2) = - F(\nu_1, \nu_1 - \nu_2)$, 
so
$$
\mathrm{Res}_{\nu_1 = 1} F(\nu_1, \nu_2)
= \mathrm{Res}_{u=0} F(1+u, \nu_2) 
 = - \mathrm{Res}_{u=0} F(1+u, 1+u - \nu_2).
$$
We next argue that
\begin{equation*}
\int_{(1/2)} \mathrm{Res}_{u=0} F(1+u, 1+u - \nu_2) d\nu_2
= \int_{(1/2)} \mathrm{Res}_{u=0} F(1+u, 1 - \nu_2) d\nu_2.
\end{equation*}
Here we can view the left hand side as a double integral of the form
\begin{equation*}
\frac{1}{2 \pi i} \int_{(1/2)} \oint_{|w| = \varepsilon} F(1+w, 1+w - \nu_2) \frac{dw}{w} d\nu_2.
\end{equation*}
The point is that 
the only poles in a neighborhood of
$\sigma_1 = 1$ and $\sigma_2 = 1/2$ is the line $\nu_1 = 1$.  Hence we can change variables $\nu_2 \rightarrow \nu_2 + w$ and shift the $\nu_2$-contour back to $1/2$, without encountering any poles.  Therefore, $R_6$ is the integral of an odd function, and hence it vanishes.

The case of $R_9$ is very similar to that of $R_6$.
We use
$F(\nu_1, \nu_2) = - F(2 \nu_2 - \nu_1, \nu_2)$, so
$$
\mathrm{Res}_{\nu_2 = 1/2} F(\nu_1, \nu_2)
= \mathrm{Res}_{u=0} F(\nu_1, 1/2+u) 
 = - \mathrm{Res}_{u=0} F(1+2u - \nu_1, 1/2+u),
$$
and similarly to the case of $R_6$, inside the integral we can replace $\mathrm{Res}_{u=0} F(1+2u-\nu_1, 1/2+u)$ by $\mathrm{Res}_{u=0} F(1-\nu_1, 1/2+u)$.  Hence $R_9 = 0$.
\end{proof}

\subsection{A basic Rankin-Selberg estimate}
\label{section:RankinSelbergBound}
Define an equivalence relation 
on $\Lambda^{+}$ by $M_1 \sim M_2$ if
there exists $\gamma \in \mathrm{GL}_2(\mz)$ so that $\gamma M_1 \gamma^t = M_2$.
\begin{myprop}\label{prop.RSBound}
Let $F$ be a Siegel cusp form of weight $k$ for $\Sp_4(\mz)$.
We have
\begin{equation*}
\sum_{\substack{M \in \Lambda^{+}/\{\sim\} \\ \det(M) \leq X }} \frac{|a_F(M)|^2}{\det(M)^{k-\frac32}} \ll X^{\frac32+\varepsilon}.
\end{equation*}
\end{myprop}
%Remark.  For comparison, note that
%\begin{equation}
%\sum_{\substack{T \in \Lambda^{+} \\ \mathrm{Tr}(T) \leq X }} 1 \ll X^{3+\varepsilon}.
%\end{equation}
\begin{proof}
Maass \cite{Maass} (for an English reference, see \cite[Prop. 6]{Sturm}) showed that
\begin{equation*}
\sum_{M \in \Lambda^{+}/\{\sim\}} \frac{|a_F(M)|^2}{\det(M)^s}
\end{equation*}
converges absolutely for $\mathrm{Re}(s) > k$.  Hence using Rankin's trick,
\begin{equation*}
\sum_{\substack{M \in \Lambda^{+}/\{ \sim \} \\ \det(M) \leq X }} \frac{|a_F(M)|^2}{\det(M)^{k-\frac32}}
\leq 
 \sum_{\substack{M \in \Lambda^{+}/\{ \sim \} \\ \det(M) \leq X }} \frac{|a_F(M)|^2}{\det(M)^{k-\frac32}}
 \frac{X^{3/2+\varepsilon}}{\det(M)^{3/2+\varepsilon}} \ll X^{3/2+\varepsilon}. \qedhere
\end{equation*}
\end{proof}
% Remark. Recall that the standard reduction theory of binary quadratic forms means that every equivalence class of $M = (\begin{smallmatrix} a_1 & a_2/2 \\ a_2/2 & a_3 \end{smallmatrix})$ contains a representative with $0 \leq a_2 \leq a_1 \leq a_3$.  Hence an alternative way to write this bound is
% \begin{equation}
% \sum_{\substack{M = (\begin{smallmatrix} a_1 & a_2/2 \\ a_2/2 & a_3 \end{smallmatrix}) \\ 0 \leq a_2 \leq a_1 \leq a_3 \\ 0 < \det(M) \leq X}}
% \frac{|a_F(M)|^2}{\det(M)^{k-\frac32}} \ll X^{3/2+\varepsilon}.
% \end{equation}
% Note that the number of matrices $M$ being summed over above is $X^{3/2 + o(1)}$.

\begin{mycoro}
\label{coro:RankinSelbergVariant}
We have
\begin{equation*}
\sum_{\substack{M   \in \Lambda^{+} \\ \|M \|_{\infty} \leq W}} 
\frac{|a_F(M)|^2}{\det(M)^{k-\frac32}} \ll W^{3+\varepsilon}.
\end{equation*}
\end{mycoro}
\begin{proof}
We have
\begin{equation}
\label{eq:CorollaryRankinSelbergBoundFirstLine}
\sum_{\substack{M   \in \Lambda^{+} \\ \|M \|_{\infty} \leq W}} 
\frac{|a_F(M)|^2}{\det(M)^{k-\frac32}}
= 
\sum_{\substack{M   \in \Lambda^{+}/\{ \sim \} \\ \|M \|_{\infty} \leq W}} 
\frac{|a_F(M)|^2}{\det(M)^{k-\frac32}}
\sum_{\substack{\gamma \in \mathrm{GL}_2(\mz) \\ \| \gamma M \gamma^t \|_{\infty} \leq W}} 1.
\end{equation}
Say $\gamma = (\begin{smallmatrix} a & b \\ c & d \end{smallmatrix})$ and $M = (\begin{smallmatrix} m_1 & m_2/2 \\ m_2/2 & m_3 \end{smallmatrix})$.  The upper-left and lower-right entries of $\gamma M \gamma^t$ are
$m_1 a^2 + m_2 ab + m_3 b^2$ and $m_1 c^2 + m_2 cd + m_3 d^2$, respectively.  By completing the square,
\begin{equation*}
    m_1 a^2 + m_2 ab + m_3 b^2
    = m_1 \Big(a + \frac{m_2 b}{2m_1}\Big)^2 + \frac{\det(M)}{m_1} b^2,
\end{equation*}
so that the condition $\|\gamma M \gamma^t \|_{\infty} \leq W$ implies that $|b| \leq \frac{\sqrt{m_1 W}}{\sqrt{\det(M)}} \leq \frac{W}{\sqrt{\det(M)}}$.  The same bound holds for each of $a$, $c$, and $d$.  It is well-known that the number of $\gamma \in \mathrm{GL}_2(\mz)$ with $\|\gamma \|_{\infty} \leq X$ is $\ll X^{2+\varepsilon}$.  Hence the inner sum over $\gamma$ appearing in \eqref{eq:CorollaryRankinSelbergBoundFirstLine} is $\ll \frac{W^{2+\varepsilon}}{\det(M)}$.  Therefore,
\begin{equation*}
    \sum_{\substack{M   \in \Lambda^{+} \\ \|M \|_{\infty} \leq W}} 
\frac{|a_F(M)|^2}{\det(M)^{k-\frac32}}
\ll 
\sum_{\substack{M   \in \Lambda^{+}/\{\sim \} \\ \det(M) \leq W^2}} 
\frac{|a_F(M)|^2}{\det(M)^{k-\frac32}} \frac{W^{2+\varepsilon}}{\det(M)}
\ll 
\sum_{\substack{M   \in \Lambda^{+}/\{ \sim \}\\ \det(M) \leq W^2}} 
\frac{|a_F(M)|^2}{\det(M)^{k-\frac32}} \frac{W^{3+\varepsilon}}{\det(M)^{3/2+\varepsilon}},
\end{equation*}
which is bounded by $O(W^{3+\varepsilon})$ as in the proof of Proposition \ref{prop.RSBound}.
\end{proof}

\section{Shifted convolution sums }
In this section we elaborate on the calculation of the triple product integral $\langle F P_Q(\cdot, \phi), G \rangle$.  In Section \ref{section:basics}, we simplify this integral as much as possible, and derive some of the basic properties of the resulting integral transforms, in order to interpret the triple product as a shifted convolution sum of length $\approx N$.  In Section \ref{section:alternateUnfolding}, we give the alternative expression for the triple product by inserting $1$ as an incomplete Eisenstein series, and unfolding with it instead.  
\subsection{Basics}
\label{section:basics}
We begin this section with a basic unfolding result.
\begin{myprop}[Unfolding]
\label{prop:unfoldingBasicPoincare}
Suppose $F, G \in S_k(\Gamma)$.
Let $\phi=\phi_N$ be as in \eqref{eq:phiNdef}, 
and let $P_{Q}(Z, \phi)$ be as in \eqref{eq:PoincareSeriesDefinition}. 
%Define $\phi_1$ by $\phi(R) = \exp(2 \pi \mathrm{Tr}(Q Y)) \phi_1(R)$, where $Y = R R^t$.  
Then
\begin{equation*}
\langle F P_{Q}(\cdot, \phi), G \rangle =
\sum_{\substack{M_1, M_2 \in \Lambda^+ \\ M_1 + Q = M_2}}
a_F(M_1) \overline{a_G(M_2)} \cdot I(M_1 +M_2, \phi),
\end{equation*}
where $I(M, \phi)$ is the Laplace transform of $\phi$ (times a power of determinant) on the space $\mathrm{M}_2(\mr)^{\mathrm{sym}, >0}$. Explicitly,
%{\color{red}$I(M,\phi)$ is the Laplace transform of $\phi(Y)\det(Y)^{k-3}$ on $\Lambda^+$. Expressing in terms of the coordinates \eqref{eq:YvsAcoordinates}, this is given by} 
\begin{equation}
\label{eq:PoincareSeriesIntegralTransform}
I(M ,\phi) = 
\int_{Y\in \mathrm{M}_2(\mr)^{\mathrm{sym}, >0}}
\exp(-2\pi \mathrm{Tr}(MY))
\phi(Y)\det(Y)^{k-3}dY.
\end{equation}
\end{myprop}
\begin{proof}
Let $a_F(M)$ and $a_G(M)$ be the Fourier coefficients of $F$ and $G$ as in \eqref{eq:SiegelFourierExpansion}.
We apply \eqref{eq:PoincareSeriesUnfoldingWithYcoord} with $f(Z) = F(Z) \overline{G}(Z) (\det Y)^{k}$ (which is $\Gamma$-invariant), and
perform the $X$-integral, immediately giving the result.
% \begin{multline}
% \langle F P_{Q}(\cdot, \phi), G \rangle 
% =
% \sum_{\substack{M_1, M_2 \in \Lambda^+ \\ M_1 + Q = M_2}}
% a_F(M_1) \overline{a_G(M_2)}
% \\
% \times
% \int_{r_1, r_2 \in \mr_{>0}, u \in \mr}
% \exp(-2\pi \mathrm{Tr}((M_1+M_2)Y)
% \phi((\begin{smallmatrix} 1 & u \\ & 1 \end{smallmatrix}) (\begin{smallmatrix} \sqrt{r_1} & \\ & \sqrt{r_2} \end{smallmatrix}))
% (r_1 r_2)^k
% \frac{du dr_1 dr_2}{r_1^3 r_2^2}.
% \end{multline}
\end{proof}

Remark: we can also express $I(M,\phi)$ in terms of the coordinates \eqref{eq:YvsAcoordinates} by 
\begin{equation*}
I(M ,\phi) = 
\int_{r_1, r_2 \in \mr_{>0}, u \in \mr}
\exp(-2\pi \mathrm{Tr}(MY))
\phi((\begin{smallmatrix} 1 & u \\ & 1 \end{smallmatrix}) (\begin{smallmatrix} \sqrt{r_1} & \\ & \sqrt{r_2} \end{smallmatrix}))
(r_1 r_2)^k
\frac{du dr_1 dr_2}{r_1^3 r_2^2}.
\end{equation*}

\begin{mylemma}[Simplification]
\label{lemma:PoincareSeriesIntegralSimplified}
Let $\phi=\phi_N$ be as in \eqref{eq:phiNdef}.
Suppose that $M = (\begin{smallmatrix} a_1 & a_2/2 \\ a_2/2 & a_3 \end{smallmatrix}) > 0$.
Then $I(M, \phi)$ defined by \eqref{eq:PoincareSeriesIntegralTransform} satisfies
$I(M, \phi) = \det(M)^{-k + \frac32} I_1(M, \phi)$, where
\begin{multline*}
I_1(M, \phi) = 
\int_{r_1, r_2 \in \mr_{>0}, u \in \mr}
\exp(-2 \pi (  r_1 + r_2 + u^2 r_2))
\\
\varphi(\frac{u \sqrt{\det(M)}}{a_1} - \frac{a_2}{2a_1}, \frac{Nr_1}{a_1}, \frac{N a_1 r_2}{\det(M)})
r_1^{k-2} r_2^{k-1}
\frac{du dr_1 dr_2}{r_1 r_2}.
\end{multline*}
\end{mylemma}
\begin{proof}
Recall that $\mathrm{Tr}(MY) = a_1 (r_1 + u^2 r_2) + a_2 u r_2 + a_3 r_2$, so that
\begin{equation}
\label{eq:ItphiDefinition}
I(M, \phi) = 
\int_{r_1, r_2 \in \mr_{>0}, u \in \mr}
\exp(-2 \pi ( a_1 r_1 + (a_1 u^2 + a_2 u + a_3) r_2))
\varphi(u, Nr_1, Nr_2)
(r_1 r_2)^k
\frac{du dr_1 dr_2}{r_1^3 r_2^2}.
\end{equation}
Now change variables $r_1 \rightarrow r_1/a_1$ and $u \rightarrow \frac{\sqrt{\det(M)}}{a_1} u - \frac{a_2}{2a_1}$, giving
\begin{multline*}
I(M, \phi) = 
\int_{r_1, r_2 \in \mr_{>0}, u \in \mr}
\exp(-2 \pi  (r_1 + (\frac{\det(M)}{a_1} u^2 + \frac{\det(M)}{a_1}) r_2))
\\
\times
 \varphi\Big(\frac{u \sqrt{\det(M)}}{a_1} - \frac{a_2}{2a_1}, \frac{N r_1}{a_1}, N r_2 \Big)
\frac{\det(M)^{1/2}}{a_1} \frac{r_1^{k-2} r_2^{k-1}}{a_1^{k-2}}
\frac{du dr_1 dr_2}{r_1 r_2}.
\end{multline*}
Finally, changing variables $r_2 \rightarrow r_2 \frac{a_1}{\det(M)}$ completes the proof.
\end{proof}

\begin{myprop}[Estimates for the weight function]
\label{prop:PoincareSeriesIntegralEstimates}
Let notation be as in Lemma \ref{lemma:PoincareSeriesIntegralSimplified}.  Suppose that $\varphi(u, t_1, t_2)$ is a fixed Schwartz-class function, supported on $t_1 \asymp t_2 \asymp 1$.
Let $W(a_1, a_2, a_3) = I_1(M, \phi)$, viewed as a function of $a_1$, $a_2$, and $a_3$, restricted by $\det(M) = a_1 a_3 - a_2^2/4 > 0$.  Then
\begin{equation}
\label{eq:Wbound}
W^{(i_1,i_2,i_3)}(a_1, a_2, a_3) \ll 
\Big(1 + \frac{N^2}{\det(M)}\Big)^{i_1+i_2+i_3} \frac{1}{N^{i_1+i_2 + i_3}}
\Big(1 + \frac{a_1}{N}\Big)^{-A} \Big(1 + \frac{a_3}{N}\Big)^{-A}. 
\end{equation}
\end{myprop}
\begin{proof}
We first consider the bound with $i_1=i_2=i_3=0$.  
By the support of $\varphi$ in terms of $r_1$, we have $r_1 \asymp \frac{a_1}{N}$, and hence $\int_{r_1 \asymp \frac{a_1}{N}} r_1^{k-3} \exp(-2 \pi r_1) dr_1 \ll (1 + \frac{a_1}{N})^{-A}$.  Similarly, for the integral over $r_2$, we have 
\begin{equation}
\label{eq:t2integralbound}
\int_{r_2 \asymp \frac{\det(M)}{N a_1}} r_2^{k-2} \exp(- 2\pi r_2) dr_2 \ll 
\frac{\sqrt{\det(M)}}{\sqrt{a_1 N}}
\Big(1 + \frac{\det(M)}{a_1 N}\Big)^{-A}.
\end{equation}
Next we examine the $u$-integral.  Suppose that the support of $\varphi$ implies $2 \pi r_2 \geq c \frac{\det(M)}{a_1 N}$.  Then $\exp(- 2\pi u^2 r_2) \leq \exp(-c \frac{u^2 \det(M)}{a_1 N} )$, and then
\begin{multline*}
    \intR \exp(-2 \pi r_2 u^2) |\varphi\Big(\frac{u \sqrt{\det(M)} -a_2/2}{a_1}, \cdot \Big)| du
   \\
   \ll \frac{\sqrt{a_1 N}}{\sqrt{\det(M)}} \intR \exp(- cu^2)  
    |\varphi\Big(\frac{u \sqrt{a_1 N} -a_2/2}{a_1}, \cdot \Big)| du
    \ll \frac{\sqrt{a_1 N}}{\sqrt{\det(M)}} \Big(1 + \frac{a_2^2}{a_1 N}\Big)^{-A}.
\end{multline*}
Combining these bounds, we obtain
\begin{equation}
\label{eq:WboundInProof}
    W(a_1, a_2, a_3) \ll 
    \Big(1 + \frac{a_1}{N}\Big)^{-A}
    \Big(1 + \frac{\det(M)}{a_1 N}\Big)^{-A} \Big(1 + \frac{a_2^2}{a_1 N} \Big)^{-A}.
\end{equation}
Note that $(1 + \frac{\det(M)}{a_1 N})^{-1} (1 + \frac{a_2^2}{a_1 N})^{-1}  \leq (1 + \frac{a_3}{N})^{-1}$, so \eqref{eq:WboundInProof} simplifies to give \eqref{eq:Wbound}.

Next we consider derivatives.  For this, it seems easier to return to the expression from \eqref{eq:ItphiDefinition}, prior to the changes of variables.  To aid in the calculations, note that $\frac{\partial}{\partial a_1} \det(M) = a_3$, 
$\frac{\partial}{\partial a_2} \det(M) = -a_2/2$, and $\frac{\partial}{\partial a_3} \det(M) = a_1$.  Hence
\begin{multline*}
\frac{\partial}{\partial a_1} I_1(M, \phi_1) = 
\frac{\partial}{\partial a_1} 
\Big[
\det(M)^{k-\frac32} 
\int_{r_1, r_2 \in \mr_{>0}, u \in \mr}
\exp(-2 \pi ( a_1 r_1 + (a_1 u^2 + a_2 u + a_3) r_2))
\\
\times 
\varphi(u, Nr_1, Nr_2)
(r_1 r_2)^k
\frac{du dr_1 dr_2}{r_1^3 r_2^2}
\Big]
.
\end{multline*}
By the chain rule, this derivative takes the form
\begin{multline*}
\frac{a_3}{\det(M)} (k-\tfrac32) \times I_1(M, \phi)
+ \det(M)^{k-\frac32} \int 
\int_{r_1, r_2 \in \mr_{>0}, u \in \mr}
(- 2 \pi) (r_1 + u^2 r_2)
\\
\times 
\exp(-2 \pi ( a_1 r_1 + (a_1 u^2 + a_2 u + a_3) r_2))
\varphi(u, Nr_1, Nr_2)
(r_1 r_2)^k
\frac{du dr_1 dr_2}{r_1^3 r_2^2}.
\end{multline*}
For the first term, we can re-apply the previous bound on $I_1$, and use that (effectively) $\frac{a_3}{\det(M)} \ll \frac{N}{\det(M)}$.  For the second term, we have that $r_1 \asymp N^{-1}$, $r_2 \asymp N^{-1}$, and (effectively) $u \ll 1$, so the bound is multiplied by $N^{-1}$.  Repeated derivatives follow the same general pattern.
Similar discussions work equally well for the partial derivatives with respect to $a_2$ and $a_3$.
\end{proof}

\subsection{Alternate unfolding}
\label{section:alternateUnfolding}
Let notation be as in Proposition \ref{prop:unfoldingBasicPoincare}.  Let $H$ be as in Theorem \ref{thm:Einc=1}, 
such that $E_0^*(g) = 1$.
Suppose that $P_Q$ has the Fourier expansion
\begin{equation}
\label{eq:PoincareSeriesFourierExpansion}
P_Q(Z, \phi)
= \sum_{T \in \Lambda} a_Q(T, Y, \phi) e(\mathrm{Tr}(TX)).
\end{equation}
\begin{myprop}
\label{prop:ShiftedSumUnfoldedViaEisenstein}
 Let notation be as above.  Using the coordinates \eqref{eq:YvsAcoordinates}, assuming the $\nu_i$ lines of integration are in the region of absolute convergence \eqref{eq:minimalEisensteinSeriesAbsoluteConvergence}, we have
\begin{multline}
\label{eq:FGPQinnerproduct}
\langle F P_Q(\cdot, \phi), G \rangle 
=
\int_{(\sigma_1)} \int_{(\sigma_2)}  
H^*(\nu_1, \nu_2)
\sum_{T + M_1= M_2} a_F(M_1) \overline{a_G}(M_2) 
\int_{
r_1, r_2 \in \mr_{>0}, u \in \mr}
a_{Q}(T, Y, \phi)
\\
\times
\exp(- 2 \pi \mathrm{Tr}((M_1+M_2)Y))
h((\begin{smallmatrix} 1 & u \\ & 1 \end{smallmatrix}))
(r_1/r_2)^{\frac{\nu_1 - \nu_2 + \frac12}{2}}(r_1 r_2)^{\frac{\nu_2}{2} + \frac34}
(r_1 r_2)^k 
 \frac{ du dr_1 dr_2}{r_1^3 r_2^2} 
 \frac{d\nu_1 d \nu_2}{(2 \pi i)^2}.
\end{multline}
\end{myprop}
\begin{proof}
We begin by writing $\langle F P_Q(\cdot, \phi), G \rangle = \langle F P_Q(\cdot, \phi) E_0^*, G \rangle$, and then use $E_0^*$ to unfold the integral, as in Proposition \ref{prop:unfoldingWithIncompleteEisensteinSeriesIdenticallyOne}.  
We let $f(Z) = (\det Y)^k F(Z) \overline{G}(Z) P_Q(g, \phi)$.  
Applying the Fourier expansion of $P_Q$, as well as the standard Fourier expansions of $F$ and $G$, we obtain the claimed formula.
\end{proof}

\section{The Fourier expansion of the Poincare series}
\label{section:FourierExpansionPoincare}
In this section, we develop various properties of the Fourier coefficients appearing in \eqref{eq:PoincareSeriesFourierExpansion}.  
\subsection{Initial steps}
To begin, we have by definition
\begin{equation*}
a_Q(T, Y, \phi) = \int_{(\mr/\mz)^3} P_Q(Z, \phi) e(- \mathrm{Tr}(TX)) dX.
\end{equation*}
Applying the definition \eqref{eq:PoincareSeriesDefinition}, we have
\begin{equation}
a_Q(T, Y, \phi) = 
\sum_{\gamma \in \Gamma_{\infty} \backslash \Gamma} 
\int_{(\mr/\mz)^3} 
e(\mathrm{Tr}(Q \mathrm{Re}(\gamma Z))) \phi(\iota^{-1}(\mathrm{Im}(\gamma (Z)))
 e(- \mathrm{Tr}(TX)) dX.
\end{equation}
We now follow the method of Kitaoka \cite{Kitaoka}.
His Lemma 1 (in \S 2) states that
\begin{equation*}
\Gamma_{\infty} \gamma \Gamma_{\infty} = \bigcup_{S \in \Lambda'/\theta(\gamma)} \Gamma_{\infty} \gamma (\begin{smallmatrix} 1 & S \\ & 1 \end{smallmatrix}), \qquad \text{(disjoint union)}
\end{equation*}
where $\theta(\gamma) = \{ S \in \Lambda' : \gamma (\begin{smallmatrix} 1 & S \\ & 1 \end{smallmatrix}) \gamma^{-1} \in \Gamma_{\infty} \}$, and where $\Lambda' = \mathrm{M}_2^{\mathrm{sym}}(\mz)$.  Suppose that $\mathfrak{h}$ runs over a complete system of representatives for $\Gamma_{\infty} \backslash \Gamma /\Gamma_{\infty}$.  Kitaoka's Lemma 2 (\S 2) then shows that $\gamma = (\begin{smallmatrix} A & B \\ C & D \end{smallmatrix})$ is parameterized by $C$ and $D \pmod{C \Lambda'}$.  
Note $\gamma \cdot (\begin{smallmatrix} 1 & S \\ 0 & 1 \end{smallmatrix}) (Z)  = \gamma ( Z + S)$.
Hence $a_Q(T,Y, \phi)$ equals
\begin{equation}
\begin{split}
  \sum_{\gamma \in \mathfrak{h}}
\sum_{S \in \Lambda'/\theta(\gamma)} 
\int_{(\mr/\mz)^3}
 e(\mathrm{Tr}(Q \mathrm{Re}(\gamma (Z +S) )) \phi(\iota^{-1}(\mathrm{Im}(\gamma  (Z +S)))  
e(- \mathrm{Tr}(TX)) dX 
\\
\label{eq:aQformulaWithS}
=  \sum_{C}
\sum_{D \shortmod{C \Lambda'}} 
\sum_{S \in \Lambda'/\theta(\gamma)} 
\int_{(\mr/\mz)^3 + S}
 e(\mathrm{Tr}(Q \mathrm{Re}(\gamma (Z) )) \phi(\iota^{-1}(\mathrm{Im}(\gamma (Z)))  
e(- \mathrm{Tr}(TX)) dX,
\end{split}
\end{equation}

where we changed variables $X \rightarrow X- S$, and used that $\mathrm{Tr}(TS) \in \mz$.

Now the sum over $C$ breaks up into those $C$ of rank $0$, $1$, and $2$.  We correspondingly write
$a_Q(T, Y, \phi) = a_Q^{(0)}(T, Y, \phi) + a_Q^{(1)}(T, Y, \phi) + a_Q^{(2)}(T, Y, \phi)$.
% Inserting this decomposition into \eqref{eq:FGPQinnerproduct}, we then obtain \blue{pick a better name than $S_i$}
% \begin{equation}
% \langle F P_G(\cdot, \phi), G \rangle = S_0 + S_1 + S_2,
% \end{equation}
% where
% \begin{multline}
% S_i 
% =
% \int_{(\sigma_1)} \int_{(\sigma_2)}  
% H^*(\nu_1, \nu_2)
% \sum_{T + S= R} a_F(S) \overline{a_G}(R) 
% \int_{
% r_1, r_2 \in \mr_{>0}, u \in \mr}
% a_{Q}^{(i)}(T, Y, \phi)
% \\
% \times
% \exp(- 2 \pi \mathrm{Tr}((S+R)Y))
% h((\begin{smallmatrix} 1 & u \\ & 1 \end{smallmatrix}))
% (r_1/r_2)^{\frac{\nu_1 - \nu_2 + \frac12}{2}}(r_1 r_2)^{\frac{\nu_2}{2} + \frac34}
% (r_1 r_2)^k 
%  \frac{ du dr_1 dr_2}{r_1^3 r_2^2} 
%  \frac{d\nu_1 d \nu_2}{(2 \pi i)^2}.
% \end{multline}

\subsection{Rank \texorpdfstring{$2$}{2}}
We compute along the lines of Kitaoka to obtain the following simplified formula for $a_Q^{(2)}$.
\begin{myprop}
\label{prop:rank2SumOfKloostermanSums}
We have that
\begin{equation*}
a_Q^{(2)}(T, Y, \phi) = \sum_{\det(C) \neq 0}  K(Q, T;C) \mathcal{I}(Q, T, Y, C, \phi),
\end{equation*}
where $K$ is the Kloosterman sum defined in Section \ref{section:KloostermanSumsSymplectic}, and where
\begin{equation}
\label{eq:PoincareSeriesIntegralProp51}
\mathcal{I}(Q, T, Y, C, \phi)
=
\int_{\mr^3}
 e(\mathrm{Tr}[Q C^{-t} \mathrm{Re}(- Z^{-1}) C^{-1} -TX]) \phi(C^{-t} \iota^{-1}(\mathrm{Im}(-Z^{-1})))  
dX.
\end{equation}
\end{myprop}

\begin{proof}
For $C$ of rank $2$, Kitaoka (Lemma 5, \S 2) shows that $\theta(\gamma) = \{0 \}$, so $S$ runs over $\Lambda'$.  Note $\sum_S \int_{(\mr/\mz)^3 + S} = \int_{\mr^3}$.  Hence
\begin{equation*}
a_Q^{(2)}(T, Y, \phi)
= \sum_{\det(C) \neq 0} \sum_{D \shortmod{C \Lambda}} 
\int_{\mr^3}  e(\mathrm{Tr}(Q \mathrm{Re}(\gamma  ( Z ) )) \phi(\iota^{-1}(\mathrm{Im}(\gamma (Z)))  
e(- \mathrm{Tr}(TX)) dX.
\end{equation*}
Next change variables $Z \rightarrow Z - C^{-1} D$, which is allowable since $C$ has rank $2$.  Note
\begin{equation*}
\gamma (Z - C^{-1} D) = (AZ + B - A C^{-1} D)(CZ)^{-1} 
= A C^{-1} + (B - A C^{-1} D) Z^{-1} C^{-1}.
\end{equation*}
Hence
\begin{equation*}
a_Q^{(2)}(T,Y) 
= \sum_{\det(C) \neq 0} K(Q,T;C) \mathcal{I}(Q, T, C, \phi),
\end{equation*}
%\begin{multline*}
%a_Q^{(2)}(T,Y) = \sum_{\det(C) \neq 0}
%\sum_{D \shortmod{C \Lambda}} 
%e(\mathrm{Tr}(Q A C^{-1} + T C^{-1} D))
%\mathcal{I}(Q, T, C, \phi),
%\int_{\mr^3}
% e(\mathrm{Tr}(Q \mathrm{Re}((B - AC^{-1} D) Z^{-1} C^{-1})) 
%\\
%\times  
% \phi(\iota^{-1}(\mathrm{Im}((B-A C^{-1} D) Z^{-1} C^{-1}))  
%e(- \mathrm{Tr}(TX)) dX,
%\end{multline*}
where
$\mathcal{I}(Q, T, Y, C, \phi)$ is given by
\begin{equation*}
\int_{\mr^3}
 e(\mathrm{Tr}(Q \mathrm{Re}((B - AC^{-1} D) Z^{-1} C^{-1})) \phi(\iota^{-1}(\mathrm{Im}((B-A C^{-1} D) Z^{-1} C^{-1}))  
e(- \mathrm{Tr}(TX)) dX.
\end{equation*}

We claim that $B - A C^{-1} D = - C^{-t}$.
Since $(\begin{smallmatrix} A & B \\ C & D \end{smallmatrix}) \in \mathrm{Sp_4}(\mr)$, we have that $A^t D - C^t B = I$, and since $C$ is invertible, $C^{-t} A^t D - B = C^{-t}$.  We also have that $A^t C = C^t A$, so 
$C^{-t} A^t = AC^{-1}$.  This gives the claim.  Hence
\begin{equation*}
\mathcal{I} = 
\int_{\mr^3} e(\mathrm{Tr}[Q \mathrm{Re}(C^{-t} (- Z^{-1}) C^{-1} - TX]) \phi(\iota^{-1}(\mathrm{Im}(C^{-t} (-Z^{-1}) C^{-1}))  
 dX.
\end{equation*}

To complete the proof, note that
$$
\mathrm{Re}(C^{-t} (- Z^{-1}) C^{-1})
= C^{-t} \mathrm{Re}( - Z^{-1}) C^{-1}
$$
and
$$
\phi(\iota^{-1}(\mathrm{Im}(C^{-t} (-Z^{-1}) C^{-1}))
= \phi(\iota^{-1}( C^{-t}(\mathrm{Im}(-Z^{-1}) C^{-1})
= \phi(C^{-t} \iota^{-1}((\mathrm{Im}(-Z^{-1}))). \qedhere
$$
\end{proof}

\begin{mylemma}
\label{eq:IboundRank2}
Let $\mathcal{I} = \mathcal{I}(Q,T,Y, C,\phi)$ be defined as in \eqref{eq:PoincareSeriesIntegralProp51}, with $\phi$ as in \eqref{eq:phiNdef}.  Then
\begin{equation*}
\mathcal{I} \ll (r_1 r_2)^{3/4} \frac{N^{3/2}}{|\det(C)|^{3/2}}.
\end{equation*}
Moreover, with $Y = R R^t$, $\mathcal{I}(Q,T,Y,C, \phi)$ is supported on matrices $C$ such that
\begin{equation}
\label{eq:CRRCtruncation}
\| C R R^t C^t \|_{\infty} \ll N.
\end{equation}
\end{mylemma}
\begin{proof}
By applying the triangle inequality to $\mathcal{I}$, we have
\begin{equation*}
|\mathcal{I}| \leq 
\int_{\mr^3}
|\phi(C^{-t} \iota^{-1}(\mathrm{Im}(-Z^{-1})))|
dX.
\end{equation*}
As a first step, recall we write $Z = X+iY$ with $Y = R R^t$ as in \eqref{eq:YvsAcoordinates}.  We change variables by writing $X = R X' R^t$, which gives $dX = (r_1 r_2)^{3/2} dX'$ and $Z \rightarrow R (X' + iI) R^t$.  We then rename $X'$ back to $X$.  This gives
\begin{equation}
\label{eq:IboundInProof}
|\mathcal{I}| \leq (r_1 r_2)^{3/2} 
\int_{\mr^3}
|\phi(C^{-t} R^{-t} \iota^{-1}(\mathrm{Im}(-(X + iI)^{-1})))|
dX.
\end{equation}
Note that $\mathrm{Im}(-(X+iI)^{-1}) = (X^2 + I)^{-1}$.  Hence by the support of $\phi$, the integral may be restricted to $X$ such that
\begin{equation*}
C^{-t} R^{-t} (I+ X^2)^{-1} R^{-1} C^{-1} = \begin{pmatrix} 1 & U \\ & 1 \end{pmatrix} \begin{pmatrix} T_1 & \\ & T_2 \end{pmatrix} \begin{pmatrix} 1 & \\ U & 1 \end{pmatrix},
\end{equation*}
with $U \ll 1$ and $T_1 \asymp T_2 \asymp N^{-1}$.  Taking the inverse of both sides, we obtain that
\begin{equation*}
C R (I+ X^2) R^t C^t = \begin{pmatrix} 1 & 0 \\ -U & 1 \end{pmatrix} \begin{pmatrix} 1/T_1 & \\ & 1/T_2 \end{pmatrix} \begin{pmatrix} 1 &-U \\  & 1 \end{pmatrix}
= \begin{pmatrix} 1/T_1 & -U/T_1 \\ -U/T_1 & T_2^{-1} + T_1^{-1} U^2 \end{pmatrix},
\end{equation*}
which has diagonal entries of size $\asymp N$ and off-diagonal entries of size $\ll N$.  Hence
\begin{equation*}
\| C R (I+X^2) R^t C^t \|_{\mathrm{op}} \ll N,
\end{equation*}
where $\| M \|_{\mathrm{op}}$ denotes the operator norm (the maximal eigenvalue of $M$).  For symmetric positive definite matrices, the condition $M_1 \leq M_2$ implies that $\| M_1 \|_{\mathrm{op}} \leq \| M_2 \|_{\mathrm{op}}$.  
% \blue{Quick proof: According to Wikipedia, the Loewner order is a partial ordering on the convex cone of positive semi-definite matrices, which is precisely what we are using when saying $M_1 \leq M_2$.  Then Wikipedia just says that the Frobenius norm and the spectral norm are examples of monotone norms (those satisfying $M_1 \leq M_2 \Longrightarrow \|M_1 \| \leq \| M_2 \|$.  Maybe an easy way to see this directly is that we can write $M_1 = P_1^t P_1$ and then $x^t M_1 x = (P_1 x)^t (P_1 x) = |P_1 x|^2$.  Then $|P_1 x|^2 \leq |P_2 x|^2$ for all $x$.  On the other hand, according to Wikipedia, $\|M_1 \| = \max_{|x| = |y| = 1} |x^t M_1 y| = \max_{|x| = |y| = 1} |(P_1 x)^t \cdot (P_1 y)| 
% \leq \max_{|x| = |y| = 1} |P_1 x| \cdot |P_1 y|
% $ (by Cauchy, and with equality iff $P_1 x$ and $P_1 y$ are parallel.)  This equals $\max_x |P_1 x|^2$.
% Hence $\|M_1 \| \leq \max_{|x| = 1} |P_1 x|^2 \leq \max_{|x| = 1} |P_2 x|^2$.  When $v$ is a unit eigenvector of $P_2$ of maximal eigenvalue, then $\max_{|x| = 1} |P_2 x|^2 = |P_2 v|^2 = \|M_2 \|$.  This is the proof for the operator norm.
% }
Hence $\mathcal{I}$ vanishes unless $\| C R R^t C^t \|_{\mathrm{op}} \ll N$.  Since any two norms on a finite dimensional vector space are comparable, we obtain the truncation \eqref{eq:CRRCtruncation} for the sup norm.

Imposing the condition \eqref{eq:CRRCtruncation}, we then bound the integral \eqref{eq:IboundInProof} by
\begin{equation*}
|\mathcal{I}| 
\leq
(r_1 r_2)^{3/2}
\int_{\| C R X^2 R^t C^t \|_{\mathrm{op}} \ll N} 1 dX.
\end{equation*}
Now write the singular value decomposition of $CR$ by $CR = U D V$ with $U, V$ orthogonal and $D$ diagonal, say $D = (\begin{smallmatrix} \lambda_1 & \\ & \lambda_2 \end{smallmatrix})$.  Note that $\| U M U^t \|_{\mathrm{op}} = \|M  \|_{\mathrm{op}}$ for any orthogonal $U$, so 
\begin{equation*}
|\mathcal{I}| 
\leq
(r_1 r_2)^{3/2}
\int_{\| D V X^2 V^t D \|_{\mathrm{op}} \ll N} 1 dX.
\end{equation*}
Next we change variables $X \rightarrow V^t X V$, which implies $X^2 \rightarrow V^t X^2 V$.  The Lebesgue measure $dX$ is invariant under such orthogonal transformations.  Hence
\begin{equation*}
|\mathcal{I}| 
\leq
(r_1 r_2)^{3/2}
\int_{\| D  X^2  D \|_{\mathrm{op}} \ll N} 1 dX
\ll (r_1 r_2)^{3/2} \int_{\substack{\lambda_1^2 (x_1^2 + x_2^2) \ll N \\ \lambda_2^2 (x_2^2 + x_3^2) \ll N}} 1 dX
\ll 
(r_1 r_2)^{3/2}
\frac{N^{3/2}}{|\lambda_1 \lambda_2|^{3/2}},
\end{equation*}
where we used that $|x_2| \ll N^{1/2} \min(|\lambda_1|^{-1}, |\lambda_2|^{-1}) \ll N^{1/2} (|\lambda_1 \lambda_2|)^{-1/2}$.
Note that $|\lambda_1 \lambda_2| = |\det(CR)|$ and $(r_1 r_2)^{1/2} = \det(R)$, so the bound simplifies as claimed.
\end{proof}

\begin{mylemma}
\label{lemma:CsupportCondition}
With $C = (\begin{smallmatrix} c_1 & c_2 \\ c_3 & c_4 \end{smallmatrix})$,
the condition that $\| C R R^t C^t \|_{\infty} \ll N$ implies
\begin{gather*}
c_1^2 r_1 \ll N, 
\qquad (c_1 u + c_2)^2 r_2 \ll N, \\
c_3^2 r_1 \ll N, 
\qquad (c_3 u + c_4)^2 r_2 \ll N.
\end{gather*}
\end{mylemma}
\begin{proof}
Using \eqref{eq:YvsAcoordinates}, note $C R R^t C^t$ equals
\begin{equation}
\label{eq:CRRCformula}
\begin{pmatrix} 
c_1^2r_1+(c_1u+c_2)^2r_2 & c_1c_3(r_1+u^2r_2)+(c_1c_4+c_2c_3)ur_2+c_2c_4r_2\\ 
*
%c_1c_3(r_1+u^2r_2)+(c_1c_4+c_2c_3)ur_2+c_2c_4r_2
& c_3^2r_1+(c_3u+c_4)^2r_2
\end{pmatrix}.
\end{equation}
Hence the upper-left and lower-right entries give the stated constraints.
% Also note that
% \begin{equation}
% c_1^2 (r_1 + u^2 r_2) + 2c_1 c_2 u r_2 + c_2^2 r_2
% = c_1^2 r_1 + (c_1 u + c_2)^2 r_2. 
% \end{equation}
%From the upper-left entry, we can restrict $c_1$ and $c_2$ as claimed.  The lower-right entry gives the restrictions on $c_3$ and $c_4$, by symmetry.
\end{proof}

\begin{mytheo}
\label{thm.Rank2Bound}
Let $\phi$ be as in \eqref{eq:phiNdef}. Suppose $Y$ is given by 
\eqref{eq:YvsAcoordinates}
with $r_1,r_2 > 0$ 
and $u \ll1$. Then
\begin{equation}
\label{eq:aQ2bound}
|a_Q^{(2)}(T,Y, \phi)| \ll (r_1 r_2)^{3/4} (r_1^{-1} + r_2^{-1})^{1+\varepsilon} N^{5/2+\varepsilon} (1+|\det(T)|)^{\varepsilon}.
\end{equation}
\end{mytheo}
\begin{proof}
By Proposition \ref{prop:rank2SumOfKloostermanSums} and Lemma \ref{eq:IboundRank2} we have
\begin{equation*}
|a_Q^{(2)}(T,Y, \phi)|
\ll (r_1 r_2)^{3/4} N^{3/2}
\sum_{\substack{\det(C) \neq 0 \\ \| C R R^t C^t\|_{\infty} \ll N}} |\det(C)|^{-3/2} |K(Q,T;C)|.
\end{equation*}
We assume $T \neq 0$, and discuss the changes needed to handle $T=0$ at the end of the proof.
Write $C = (\begin{smallmatrix} c_1 & c_2 \\ c_3 & c_4 \end{smallmatrix})$.  Let $\alpha = \gcd(c_1, c_2, c_3, c_4)$, and replace $c_i$ by $c_i/\alpha$.
By \eqref{eq:KitaokaSimplified}, we have the bound
\begin{equation*}
|a_Q^{(2)}(T,Y, \phi)|
\ll (r_1 r_2)^{3/4} N^{3/2}
\sum_{\alpha, \beta \geq 1} \alpha^{\varepsilon}
\sum_{\substack{ \| C R R^t C^t\| \ll N/\alpha^2 \\ UCV = (\begin{smallmatrix} 1 & \\ & \beta \end{smallmatrix})}} \frac{(\beta, t)^{1/2}}{\beta},
\end{equation*}
where recall $t$ is the lower-right entry of $V^t T V$.

From Lemma \ref{lemma:CsupportCondition}, we deduce from $r_1,r_2 > 0$ and $u \ll 1 $ that 
$\|C\|_{\infty} \ll (N/\alpha^2)^{1/2} (r_1^{-1/2} + r_2^{-1/2})$.
% \begin{align}
%     \|CRR^tC^t\|\ll N/\alpha^2 \quad \Longrightarrow \quad \|C\|_{\infty} \ll N^{1/2}/\alpha^2.
% \end{align}
As a result, we have 
\begin{align*}
    |a_Q^{(2)}(T,Y, \phi)|
\ll (r_1 r_2)^{3/4} N^{3/2}
\sum_{\alpha,\beta\geq1}\frac{\alpha^{\varepsilon}}{\beta}
\sum_{\substack{(c_1, c_2, c_3, c_4) = 1 \\ |c_i| \ll  \frac{N^{1/2}}{\alpha} (r_1^{-1/2} + r_2^{-1/2}) \\ \det(C) = \beta }} (\beta, t)^{1/2}.
\end{align*}
Recalling the definition of $\mathcal{N}(X,W,T)$ from \eqref{eq:Nbound}, we have
\begin{align*} 
|a_Q^{(2)}(T,Y, \phi)|
\ll (r_1 r_2)^{3/4} N^{3/2}
\sum_{\alpha \geq 1}
\sum_{\substack{1 \ll Z \ll \frac{N}{\alpha^2} (\frac{1}{r_1} + \frac{1}{r_2})  \\ \text{dyadic}}} \frac{\alpha^{\varepsilon}}{Z}
    \mathcal{N}\Big(\frac{\sqrt{N}}{\alpha} \big(\frac{1}{\sqrt{r_1}} + \frac{1}{\sqrt{r_2}}\big), Z, T\Big).
\end{align*}
Theorem \ref{thm.Bounding(t,C)} implies
\begin{equation*}
    |a_Q^{(2)}(T, Y, \phi)| \ll (r_1 r_2)^{3/4} N^{3/2+\varepsilon} 
    \sum_{\alpha \geq 1} 
   \sum_{\substack{1 \ll Z \ll \frac{N}{\alpha^2} (\frac{1}{r_1} + \frac{1}{r_2})  \\ \text{dyadic}}}
    \frac{\alpha^{\varepsilon}}{Z} \frac{N}{\alpha^2} \Big(\frac{1}{r_1} + \frac{1}{r_2}\Big)^{1+\varepsilon} Z |\det(T)|^{\varepsilon}, 
\end{equation*}
which simplifies to give \eqref{eq:aQ2bound}.

In case $T=0$, we 
use the symmetry \eqref{eq:KloostermanSymmetry} to switch the roles of $Q$ and $T$ (recall $Q \neq 0$ by assumption).  The rest of the proof carries through with only minor changes.
\end{proof}

\subsection{Rank \texorpdfstring{$1$}{1}}
We begin with the initial computations of Kitaoka.
In this section, we will use the notation $\sumdag_{U,V}$ to mean that 
 $U$ and $V$ run over any set of coset representatives of the form
 \begin{align*}
    U=(\begin{smallmatrix}
        * & * \\ u_3 & u_4
    \end{smallmatrix})/\{\pm1\} \in \Gamma_{\infty}^{+} \backslash \mathrm{GL}_2(\mathbb{Z})/\{ \pm 1\} \quad \text{ and } \quad V=(\begin{smallmatrix}
        v_1 & * \\ v_3 & *
    \end{smallmatrix})\in \mathrm{GL}_2(\mathbb{Z})/ \Gamma_{\infty}^{+},
\end{align*}
where $\Gamma_{\infty}^{+} = \{ (\begin{smallmatrix} 1 & n \\ 0 & 1 \end{smallmatrix}) : n \in \mz \}$.
\begin{myprop}
\label{prop.Rank1Poincare}
    Write $UQU^t=(\begin{smallmatrix}
    f_1 & f_2/2 \\ f_2/2 & f_3
\end{smallmatrix})$ and $V^{-1}TV^{-t}=(\begin{smallmatrix}
    g_1 & g_2/2 \\ g_2/2 & g_3
\end{smallmatrix})$. Moreover,  write $(\begin{smallmatrix}
    z_1' & z_2'\\ z_2' & iy_3'
\end{smallmatrix})=(\begin{smallmatrix}
    x_1 & x_2 \\ x_2 & 0
\end{smallmatrix})+iV^tYV$. Then we have that 
\begin{align*}
        a_Q^{(1)}(T,Y,\phi)=\sum_{\pm}\sum_{c\geq1}\sumdag_{U,V}\delta(f_3=g_3)
        K_1(Q, T, c, U, V) \mathcal{I}_1(Q,T,Y,c,U,V,\phi),
    \end{align*}
    where \begin{align*}
    K_1(Q,T,c,U,V)=\sumstar_{d_1\shortmod{c}}\sum_{d_2\shortmod{c}}e_c(d_1g_1+d_2g_2+\overline{d_1}(f_1\mp d_2f_2+d_2^2f_3))
\end{align*}
and $\mathcal{I}_1= \mathcal{I}_1(Q,T, Y, c,U,V, \phi)$ with 
\begin{multline}
    \label{I1Def0}
    \mathcal{I}_1 = 
    \int_{\mr^2}  e(- g_1 x_1 - g_2 x_2 )
     e\Big(\mathrm{Tr}
     U Q U^t \mathrm{Re}\begin{pmatrix}  - \frac{1}{c^2 z_1'} & \frac{\pm z_2'}{cz_1'} \\ \frac{\pm z_2'}{cz_1'} & \frac{-(z_2')^2}{z_1'} \end{pmatrix}  \Big)
     \\
     \times
     \phi \circ \iota^{-1} \Big(U^t \mathrm{Im} \begin{pmatrix}  - \frac{1}{c^2 z_1'} & \frac{\pm z_2}{cz_1'} \\ \frac{\pm z_2'}{cz_1'} & \frac{-(z_2')^2}{z_1'} + i y_3'\end{pmatrix}  U \Big) dx_1 dx_2.
\end{multline}
% \begin{align}
% \label{I1Def0}
%     \mathcal{I}_1=&\ \int_{\substack{x_1,x_2\in\mr}} e\bigg( - \mathrm{Re}\left(\frac{f_1}{c^2z_1'}\mp\frac{f_2z_2'}{cz_1'}\right)-g_1x_1-g_2x_2\bigg)\nonumber\\
%     &\  \phi\left(\iota^{-1}\left(\mathrm{Im} \ U^t\left(\begin{smallmatrix}
%         -c^{-2}z_1'^{-1} & c^{-1}z_2'/z_1' \\ c^{-1}z_2'/z_1' & -z_2'^2/z_1'+z_3'
%     \end{smallmatrix}\right) U\right) \right)dx_1dx_2.
% \end{align}
% if $f_3=0$, \begin{align}\label{I1Def}
%     \mathcal{I}_1=&\ \int_{\substack{x_1,x_2\in\mr}} e\bigg( -\mathrm{Re}\left(\frac{f_3}{z_1'}\left(z_2'+\frac{g_2z_1'\mp f_2/c}{2f_3}\right)^2+\frac{\det(Q)}{c^2f_3z_1'}\right)-\frac{\det(T)}{f_3}x_1\bigg)\nonumber\\
%     &\  \phi\left(\iota^{-1}\left(\mathrm{Im} \ U^t\left(\begin{smallmatrix}
%         -c^{-2}z_1'^{-1} & c^{-1}z_2'/z_1' \\ c^{-1}z_2'/z_1' & -z_2'^2/z_1'+z_3'
%     \end{smallmatrix}\right) U\right) \right)dx_1dx_2
% \end{align}
% if $f_3\neq0$.
\end{myprop}
\begin{proof}
    We continue with \eqref{eq:aQformulaWithS}.
    For $C$ of rank $1$, Kitaoka's Lemma 4 (in \S 2) gives 
    \begin{align*}
    a_Q^{(1)}(T,Y,\phi)=&\ \sum_{d_3 = \pm 1}\sum_{c\geq1}\sumstar_{d_1\shortmod{c}}\sum_{d_2\shortmod{c}}\sumdag_{U,V}\nonumber\\
    \sum_{S\in\Lambda'/\theta(M)}
    &\ \int_{(\mz \backslash \mr)^3+S }e(\mathrm{Tr}(Q\ \mathrm{Re}(MZ)-TX)) \phi(\iota^{-1}\ \mathrm{Im} (MZ) )dX,
\end{align*}
where $M=(\begin{smallmatrix}
    * & *\\C&D
\end{smallmatrix})$ is given by $C=U^{-1}(\begin{smallmatrix}
    c &0\\0&0
\end{smallmatrix})V^t$ and $D=U^{-1}(\begin{smallmatrix}
    d_1&d_2\\0& \pm 1
\end{smallmatrix})V^{-1}$, $\theta(M)=\{S\in\Lambda|V^tSV=(\begin{smallmatrix}
    0&0\\0&*
\end{smallmatrix})\}$, and the sum over $U,V$ is as described before the proposition.

For a given $d_2$ and $U$, let
$W=(\begin{smallmatrix}
    1 & \mp d_2\\ 0 & \pm1 
\end{smallmatrix})U$. Then $C=W^{-1}(\begin{smallmatrix}
    c&0\\0&0
\end{smallmatrix})V^t$ and $D=W^{-1}(\begin{smallmatrix}
    d_1 & 0 \\ 0 & 1
\end{smallmatrix})V^{-1}$. 
Since
$M\in \Gamma_\infty\backslash\Gamma/\Gamma_\infty$ is parametrized by $C$ and $D\ \ \shortmod{C\Lambda'}$, we can take 
\begin{align}
\label{eq:M0def}
    M=\left(\begin{matrix}
        W^t & \\ & W^{-1}
    \end{matrix}\right)M_0 \left(\begin{matrix}
        V^t & \\ & V^{-1}
    \end{matrix}\right), \quad \text{ where } \quad M_0=\left(\begin{matrix}
        a & 0 & b & 0 \\ 0 & 1 & 0 & 0 \\ c & 0 & d_1 & 0 \\ 0 & 0 & 0 & 1
    \end{matrix}\right)
\end{align}
with $a$ and $b$ satisfying $ad_1-bc=1$. 
As a result, we have 
$$
M(Z) =\left(
\begin{pmatrix}
        W^t & \\ & W^{-1}
    \end{pmatrix} \cdot
    M_0\right) (V^tZV  )
    =W^t \cdot (M_0 ( V^tZV ))
    \cdot
    W.
    $$
Hence  
\begin{multline*}
    a_Q^{(1)}(T,Y,\phi)=\sum_{\pm}\sum_{c\geq1}\sumstar_{d_1\shortmod{c}}\sum_{d_2\shortmod{c}}\sumdag_{U,V}
    \sum_{S\in\Lambda' / \theta(M)}
    \int_{(\mz \backslash \mr)^3 + S} 
    \\
     e(\mathrm{Tr}(QW^t \mathrm{Re} (M_0 (V^tZV)) W-TX))
     \phi(\iota^{-1} (W^t \mathrm{Im} (M_0(V^tZV)) W)dX.
\end{multline*}

Write $Z'=V^tZV=X'+iV^tYV$, and change variables $X' = V^t X V$, where $dX' = dX$, so
\begin{multline*}
    a_Q^{(1)}(T,Y,\phi)=\sum_{\pm}\sum_{c\geq1}\sumstar_{d_1\shortmod{c}}\sum_{d_2\shortmod{c}}\sumdag_{U,V}
    \sum_{S\in\Lambda' / \theta(M)}
    \int_{(\mz \backslash \mr)^3 + V^t S V} 
    \\
     e(\mathrm{Tr}(QW^t \mathrm{Re} (M_0 (Z')) W- T V^{-t} X' V^{-1}))
     \phi(\iota^{-1} (W^t \mathrm{Im} (M_0(Z')) W)dX'.
\end{multline*}
Since $\theta(M)=\{S\in\Lambda'|V^tSV=(\begin{smallmatrix}
    0&0\\0&*
\end{smallmatrix})\}$, 
then $\sum_{S \in \Lambda'/\theta(M)} \int_{(\mz \backslash \mr)^3 + V^t SV}$ 
becomes an integral over
$X' = (\begin{smallmatrix} x_1' & x_2' \\ x_2' & x_3' \end{smallmatrix})$ with $x_1', x_2' \in \mr$ and $x_3' \in \mr/\mz$.  Using the cyclic property of the trace, we then deduce
\begin{multline*}
    a_Q^{(1)}(T,Y,\phi)=\sum_{\pm}\sum_{c\geq1}\sumstar_{d_1\shortmod{c}}\sum_{d_2\shortmod{c}}\sumdag_{U,V}
    \int_{\substack{x_1', x_2' \in \mr \\ x_3' \in \mr/\mz}}  
    \\
     e(\mathrm{Tr}(WQW^t \mathrm{Re} (M_0 (Z'))  - V^{-1}T V^{-t} X' ))
     \phi(\iota^{-1} (W^t \mathrm{Im} (M_0(Z')) W)dX'.
\end{multline*}

Writing $Z'=(\begin{smallmatrix}
    z_1' &z_2'\\ z_2'&z_3'
\end{smallmatrix})$, a direct computation 
with \eqref{eq:M0def}
shows that 
\begin{align*}
    M_0(Z')=&\ \left(\begin{matrix}
        (az_1'+b)(cz_1'+d_1)^{-1} & z_2'(cz_1'+d_1)^{-1} \\ z_2'(cz_1'+d_1)^{-1} & -cz_2'^2(cz_1'+d_1)^{-1}+z_3'
    \end{matrix}\right)\nonumber\\
    =&\ 
    \left(\begin{matrix}
        a/c-c^{-1}(cz_1'+d_1)^{-1} & z_2'(cz_1'+d_1)^{-1} \\ z_2'(cz_1'+d_1)^{-1} & -cz_2'^2(cz_1'+d_1)^{-1}+z_3'
    \end{matrix}\right).
\end{align*}
Recalling $UQU^t=(\begin{smallmatrix}
    f_1 & f_2/2 \\ f_2/2 & f_3
\end{smallmatrix})$,
% and $V^{-1}TV^{-t}=(\begin{smallmatrix}
%     g_1 & g_2/2 \\ g_2/2 & g_3
% \end{smallmatrix})$, 
we have 
$$WQ W^t = (\begin{smallmatrix} f_1 \mp d_2 f_2 + d_2^2 f_3 & \pm f_2/2 - d_2 f_3 \\ \pm f_2/2 - d_2 f_3 & f_3 \end{smallmatrix}).$$
% \begin{align}\label{WSWt}
%     WQW^t=\left(\begin{matrix}
%         f_1\mp d_2f_2+d_2^2f_3 & \pm f_2/2-d_2f_3 \\ \pm f_2/2-d_2f_3 & f_3
%     \end{matrix}\right).
% \end{align}
Note that $\mathrm{Tr}(V^{-1} T V^{-t} X') = g_1 x_1' + g_2 x_2' + g_3 x_3'$, and also observe that $\mathrm{Im}(M_0(Z'))$ is independent of $x_3'$.  We can then simplify the formula for $a_Q^{(1)}(T, Y, \phi)$ by integrating over $x_3'$, which detects $f_3 = g_3$.  In addition, we change variables $x_1' \rightarrow x_1' - \frac{d_1}{c}$, and rename $x_1'$ and $x_2'$ back to $x_1$ and $x_2$.  In all, we obtain
\begin{multline}
\label{eq:aQ1formulaMiddle}
    a_Q^{(1)}(T,Y,\phi)=\sum_{\pm}\sum_{c\geq1}\sumstar_{d_1\shortmod{c}}\sum_{d_2\shortmod{c}}\sumdag_{U,V} \delta_{f_3 = g_3}
    \int_{\mr^2}  e(- g_1 (x_1 - \frac{d_1}{c}) - g_2 x_2 )
    \\
     e\Big(\mathrm{Tr}
     Q W^t \mathrm{Re}\begin{pmatrix} \frac{a}{c} - \frac{1}{c^2 z_1'} & \frac{z_2'}{cz_1'} \\ \frac{z_2'}{cz_1'} & \frac{-(z_2')^2}{z_1'} \end{pmatrix} W \Big)
     \phi \circ \iota^{-1} \Big(W^t \mathrm{Im} \begin{pmatrix}  - \frac{1}{c^2 z_1'} & \frac{z_2'}{cz_1'} \\ \frac{z_2'}{cz_1'} & \frac{-(z_2'
     )^2}{z_1'} + i y_3'\end{pmatrix}  W \Big) dx_1 dx_2.
\end{multline}

Note that $\mathrm{Tr}(Q W^t (\begin{smallmatrix} a/c & 0 \\ 0 & 0 \end{smallmatrix})W) 
= \mathrm{Tr}(W Q W^t (\begin{smallmatrix} a/c & 0 \\ 0 & 0 \end{smallmatrix})) = \frac{a}{c}(f_1 \mp d_2 f_2 + d_2^2 f_3)$.
Next we substitute back
$W = (\begin{smallmatrix} 1 & \mp d_2 \\ 0 & \pm 1 \end{smallmatrix}) U$.  
To aid in simplifying, note
\begin{equation*}
    W^t 
    \begin{pmatrix}  - \frac{1}{c^2 z_1'} & \frac{z_2}{cz_1'} \\ \frac{z_2'}{cz_1'} & \frac{-(z_2')^2}{z_1'} + i y_3'\end{pmatrix}  W
    % \begin{pmatrix} 1 & 0 \\ \mp d_1 & \pm 1 \end{pmatrix}
    % \begin{pmatrix}  - \frac{1}{c^2 z_1} & \frac{z_2}{cz_1} \\ \frac{z_2}{cz_1} & \frac{-z_2^2}{z_1} + i y_3\end{pmatrix}  
    % \begin{pmatrix} 1 & \mp d_2 \\ 0 & \pm 1 \end{pmatrix}
    % U
  =
  U^t \begin{pmatrix}  - \frac{1}{c^2 z_1'} & \pm \frac{z_2'+\frac{d_2}{c}}{cz_1'} \\ \pm \frac{z_2' + \frac{d_2}{c}}{cz_1'} & \frac{-(z_2' + \frac{d_2}{c})^2}{z_1'} + i y_3'\end{pmatrix} U.
\end{equation*}
Hence changing variables $x_2 \rightarrow x_2 - \frac{d_2}{c}$, 
and using the cyclic property of trace again,
we obtain
\begin{multline*}
    a_Q^{(1)}(T,Y,\phi)=\sum_{\pm}\sum_{c\geq1}\sumstar_{d_1\shortmod{c}}\sum_{d_2\shortmod{c}}\sumdag_{U,V} \delta_{f_3 = g_3}
    \int_{\mr^2}  e(- g_1 (x_1 - \frac{d_1}{c}) - g_2 (x_2 - \frac{d_2}{c} )
    \\
     e_c(a (f_1 \mp d_2 f_2 + d_2^2 f_3))
     e\Big(\mathrm{Tr}
     U Q U^t \mathrm{Re}\begin{pmatrix}  - \frac{1}{c^2 z_1} & \frac{\pm z_2}{cz_1} \\ \frac{\pm z_2}{cz_1} & \frac{-z_2^2}{z_1} \end{pmatrix}  \Big)
     \phi \circ \iota^{-1} \Big(U^t \mathrm{Im} \begin{pmatrix}  - \frac{1}{c^2 z_1} & \frac{\pm z_2}{cz_1} \\ \frac{\pm z_2}{cz_1} & \frac{-z_2^2}{z_1} + i y_3\end{pmatrix}  U \Big) dx_1 dx_2.
\end{multline*}
 Finally, using that $ad_1 \equiv 1 \pmod{c}$, 
we obtain the statement of Proposition \ref{prop.Rank1Poincare}.
\end{proof}

\begin{mylemma}\label{lem.Rank1IntegralBound}
    Let $\mathcal{I}_1=\mathcal{I}_1(Q,T,c,U,V)$ be as defined in \eqref{I1Def0}, with $\phi$ as in \eqref{eq:phiNdef}. Then 
    \begin{align}
    \label{eq:I1bound}
        \mathcal{I}_1\ll \frac{1}{c^2 (\det Y)^{1/2}}  \frac{1}{|u_3| + |u_4|},
    \end{align}
    and is supported on $c$ and matrices $U,V, Y$ such that 
    \begin{align}
    \label{eq:Rank1SummationConditions}
        c\ll (\det Y)^{-1/2}, \quad u_3^2 + u_4^2 \ll (\det Y)^{-1/2}, \quad \text{ and } \quad |v_1v_3| \ll N \frac{\sqrt{y_1 y_3}}{\sqrt{\det Y}}.
    \end{align}
\end{mylemma}
\begin{proof}
    By the triangle inequality,
    \begin{align*}
    |\mathcal{I}_1| \leq \mathcal{J}:=\int_{\mr^2} \phi\left(\iota^{-1}\left(\mathrm{Im} \ U^t\left(\begin{smallmatrix}
        -c^{-2}z_1'^{-1} & c^{-1}z_2'/z_1' \\ c^{-1}z_2'/z_1' & -z_2'^2/z_1'+z_3'
    \end{smallmatrix}\right) U\right) \right)dx_1dx_2.
\end{align*}
By a direct calculation,
$\mathrm{Im}(\begin{smallmatrix} - c^{-2} z_1' & (c z_1')^{-1} z_2' \\ (c z_1')^{-1} z_2' & -(z_1')^{-1} (z_2')^2 + z_3' \end{smallmatrix})$ equals
\begin{align*}
    % \mathrm{Im}\left(\begin{matrix}
    %     -\frac{1}{c^2z_1'}& \frac{z_2'}{cz_1'}\\ \frac{z_2'}{cz_1'} & -\frac{z_2'^2}{z_1'}+z_3'
    % \end{matrix}\right)=
    (x_1^2+y_1'^2)^{-1}
    \left(\begin{matrix}
        \frac{y_1'}{c^2} & \frac{x_1y_2'-x_2y_1'}{c}\\ \frac{x_1y_2'-x_2y_1'}{c} & x_2^2y_1'-2x_1x_2y_2'-y_1'y_2'^2+y_3'(x_1^2+y_1'^2)
    \end{matrix}\right).
\end{align*}
Note that $y_1'>0$ as $Y'=V^tYV$ is positive definite. Applying a change of variable $x_1\mapsto x_1y_1'$, 
and using $\det(Y') = y_1' y_3' - (y_2')^2 = \det(Y)$, 
we arrive at 
\begin{align*}
    \mathcal{J}
    % =&\ y_1'\int_{\mathbb{R}}\int_{\mathbb{R}}\phi\circ \iota^{-1}\left(\frac{1}{(x_1^2+1)y_1'}\ U^t\left(\begin{smallmatrix}
    %     c^{-2} & (x_1y_2'-x_2)/c \\ (x_1y_2'-x_2)/c & (x_2-x_1y_2')^2+(y_1'y_3'-y_2'^2)(x_1^2+1)
    % \end{smallmatrix}\right) U\right) dx_1dx_2\nonumber\\
    =&\ y_1'\int_{\mathbb{R}}\int_{\mathbb{R}}\phi\circ \iota^{-1}\left(\frac{1}{(x_1^2+1)y_1'}\ U^t\left(\begin{smallmatrix}
        c^{-2} & (x_1y_2'-x_2)/c \\ (x_1y_2'-x_2)/c & (x_2-x_1y_2')^2+\det(Y)(x_1^2+1)
    \end{smallmatrix}\right) U\right) dx_1dx_2.
\end{align*}
Changing variable $x_2\mapsto -x_2/c+x_1y_2'$, we obtain 
\begin{align*}
    \mathcal{J}=\frac{y_1'}{c}\int_{\mathbb{R}}\int_{\mathbb{R}}\phi\circ \iota^{-1}\left(\frac{1}{c^2(x_1^2+1)y_1'}\ U^t\left(\begin{smallmatrix}
        1 & x_2 \\ x_2 & x_2^2+c^2\det(Y)(x_1^2+1)
    \end{smallmatrix}\right) U\right) dx_1dx_2.
\end{align*}
Note that with $X_2 := (\begin{smallmatrix} 1 & x_2 \\ 0 & 0 \end{smallmatrix})$, we have
\begin{equation*}
    \begin{pmatrix} 1 & x_2 \\ x_2 & x_2^2 + c^2 \det(Y) (x_1^2 + 1) \end{pmatrix}
    = X_2^t X_2 
    %\begin{pmatrix} 1 & 0 \\ x_2 & 0 \end{pmatrix} \begin{pmatrix} 1 & x_2 \\ 0 & 0 \end{pmatrix} 
    + \begin{pmatrix} 0 & 0 \\ 0 & c^2 \det(Y) (x_1^2 + 1) \end{pmatrix},
\end{equation*}
which is the sum of two positive semidefinite matrices (by the way, this calculation shows why $\mathcal{J}$ only depends on the bottom row of $U$, which is a nice sanity check).  

Recall that $\phi \circ \iota^{-1}(X)$ is supported on matrices $X$ with $\|X\|_{\infty} \ll N^{-1}$, and also $\det(X) \asymp N^{-2}$.  Computing the determinant, we obtain the restriction
$    \frac{  \det(Y)  }{c^2 (x_1^2 + 1) (y_1')^2} \asymp N^{-2}$,
and in particular, 
\begin{equation}
\label{eq:candx1bounds}
    c \ll \frac{N \sqrt{\det(Y)}}{y_1'},
    \qquad \text{and}
    \qquad |x_1| \ll \frac{N \sqrt{\det(Y)}}{c y_1'}.
\end{equation}
Since
\begin{equation*}
\frac{1}{c^2 (x_1^2 + 1) y_1'} U^t \begin{pmatrix} 0 & 0 \\ 0 & c^2 \det(Y) (x_1^2 + 1) \end{pmatrix} U
= \begin{pmatrix}
u_3^2 \frac{\det(Y)}{y_1'} & * \\ * & u_4^2 \frac{\det(Y)}{y_1'}
\end{pmatrix},
\end{equation*}
we obtain the restrictions
\begin{equation}
\label{eq:u3u4bounds}
    u_3^2 + u_4^2 \ll \frac{y_1'}{N \det(Y)}.
\end{equation}
Note that since $1 \leq u_3^2 + u_4^2$, we need $y_1' \gg N \det(Y)$ for the support to be nonempty.  We can then update the upper bound on $c$ in \eqref{eq:candx1bounds} to state $c \ll (\det Y)^{-1/2}$, consistent with \eqref{eq:Rank1SummationConditions}.  Also, since $c \geq 1$, we obtain the restriction $y_1' \ll N (\det Y)^{1/2}$.  Hence $y_1'$ may be localized by
\begin{equation}
\label{eq:y1'restrictions}
    N \det Y \ll y_1' \ll N (\det Y)^{1/2}.
\end{equation}
Applying this upper bound on $y_1'$ back into \eqref{eq:u3u4bounds} gives a truncation on $u_3$ and $u_4$ consistent with \eqref{eq:Rank1SummationConditions}.

By considering the upper-left and lower-right entries of $(U X_2)^t (UX_2)$, we obtain that
\begin{equation*}
    (u_1 + x_2 u_3)^2 +  
(u_2 + x_2 u_4)^2 
    \ll \frac{c^2 (x_1^2 + 1) y_1'}{N} \asymp \frac{N \det(Y)}{y_1'}.
\end{equation*}
If $u_3 \neq 0$, then this implies $|x_2 + \frac{u_1}{u_3}| \ll \frac{\sqrt{N \det(Y)}}{|u_3| \sqrt{y_1'}}$, and a similar bound holds if $u_4 \neq 0$.  Thus $x_2$ may be restricted to an interval of length
\begin{equation}
\label{eq:x2bound}
    \ll \frac{\sqrt{N \det(Y)}}{\sqrt{y_1'}} \frac{1}{|u_3| + |u_4|}.
\end{equation}
Combining the restrictions on $x_1$ and $x_2$ from \eqref{eq:candx1bounds} and \eqref{eq:x2bound}, 
and using \eqref{eq:y1'restrictions},
we deduce \eqref{eq:I1bound}.

Finally, we need to obtain the restriction on $v_1$ and $v_2$ from \eqref{eq:Rank1SummationConditions}.  Recall that $Y' = V^t Y V$ and $V = (\begin{smallmatrix} v_1 & * \\ v_3 & * \end{smallmatrix}) \in \GL_2(\mz)$, so by direct calculation,
\begin{align*}
    y_1'=&\ v_1^2y_1+2v_1v_3y_2+v_3^2y_3.
\end{align*}
Note
\begin{align*}
y_1' = (|v_1| \sqrt{y_1} - |v_3| \sqrt{y_3})^2 + 2 |v_1 v_3| (\sqrt{y_1 y_3} - |y_2|)
\geq 2|v_1 v_2| \frac{\det(Y)}{\sqrt{y_1 y_3} + |y_2|}.
\end{align*}
Using the upper bound on $y_1'$ from \eqref{eq:y1'restrictions}, we then obtain $|v_1 v_3| \ll  N \frac{\sqrt{y_1 y_3}}{\sqrt{\det Y}}$, which is the final claim in the lemma.
% Therefore, $|v_1 v_3| \leq \frac{y_1' \sqrt{y_1 y_3}}{\det(Y)} \ll N \frac{\sqrt{y_1 y_3}}{\sqrt{\det Y}}$.
% If $v_1v_3y_2\geq0$, then $y_1'\gg v_1^2+v_3^2\gg |v_1v_3|$ as $y_1,y_3>0$ and $y_1,y_3\gg 1$. On the other hand if $v_1v_3y_2<0$, write $\epsilon=y_2/|y_2|$. Then $\epsilon v_1v_3<0$ and
% \begin{align}
%     y_1'=&\ (v_1\sqrt{y_1}+\epsilon v_3\sqrt{y_3})^2-2\epsilon v_1v_3(\sqrt{y_1y_3}-\epsilon y_2)\nonumber\\
%     =&\ (v_1\sqrt{y_1}+\epsilon v_3\sqrt{y_3})^2+2|v_1v_3y_2|\left(\sqrt{1+\frac{\det(Y)}{|y_2|}}+\epsilon\right)\gg |v_1v_3|\det(Y).
% \end{align}
% With $\det(Y)\asymp 1$ and $y_1'\asymp N$, we conclude the bound \begin{align}\label{Rank1JBound}
%     \mathcal{J}\ll&\ N\delta(c, \det(Y)\asymp 1, u_3,u_4\ll 1, v_1v_3\ll N).
% \end{align}
\end{proof}

\begin{mytheo}\label{thm.Rank1Bound}
Let $\phi$ be as in \eqref{eq:phiNdef}. Then
\begin{equation}
\label{eq:aQ1boundFinal}
|a_Q^{(1)}(T,Y, \phi)| \ll \frac{1}{(\det Y)^{3/2}} \Big(1 + \frac{N \sqrt{y_1 y_3}}{\sqrt{\det(Y)}} \Big)^{1+\varepsilon}.
\end{equation}
\end{mytheo}
\begin{proof}
    By Proposition \ref{prop.Rank1Poincare} and Lemma \ref{lem.Rank1IntegralBound}, trivially bounding the character sum $K_1$ gives 
    \begin{align*}
        |a_Q^{(1)}(T,Y, \phi)|
        \ll \sum_{c, u_3^3, u_4^2 \ll (\det Y)^{-1/2}}
        \sum_{|v_3 v_4| \ll N \frac{\sqrt{y_1 y_3}}{\sqrt{\det Y}}} 
        c^2 \frac{1}{c^2 (\det Y)^{1/2}} \frac{1}{|u_3| + |u_4|}.
    \end{align*}
Note that if $v_3 = 0$ then $v_4 = \pm 1$ and vice-versa.
 This bound simplifies to give \eqref{eq:aQ1boundFinal}.   
\end{proof}

\subsection{Rank \texorpdfstring{$0$}{0}}
Finally, we have the contribution from $C=0$.
\begin{myprop}
\label{prop.Rank0Poincare}
    We have that 
    \begin{align}
    \label{eq:aQ0formula}
        a_Q^{(0)}(T,Y,\phi)=\sum_{\substack{U\in\mathrm{GL}_2(\mz)\\ U^t QU=T}}\phi(\iota^{-1} (UYU^t) ).
    \end{align}
    Suppose $Y$ is given by \eqref{eq:YvsAcoordinates} with $r_1, r_2 > 0$ and $u \ll 1$, and $\phi$ is as in \eqref{eq:phiNdef}.  Then 
    \begin{align}
    \label{eq:aQ0bound}
        a_Q^{(0)}(T,Y,\phi)\ll N^{-1}(r_1^{-1}+r_2^{-1})^{1+\varepsilon}.
    \end{align}
\end{myprop}
\begin{proof}
    For $C=0$, Kitaoka's Lemma 3 (in \S 2) gives (switching $U$ with $U^t$ compared to Kitaoka)
    \begin{align*}
     a_Q^{(0)}(T,Y,\phi)
     %=&\ \sum_{U\in\mathrm{GL}_2(\mz)}\int_{0\leq x_1,x_2,x_3\leq1}e(\mathrm{Tr}(Q\ \mathrm{Re} (\begin{smallmatrix}
    %     U & \\ & U^{-t}
    % \end{smallmatrix})( Z)-TX))\phi(\iota^{-1} \mathrm{Im} (\begin{smallmatrix}
    %     U & \\ & U^{-t}
    % \end{smallmatrix})(Z) )dX\nonumber\\
    =&\ \sum_{U\in\mathrm{GL}_2(\mz)}\int_{0\leq x_1,x_2,x_3\leq1}e(\mathrm{Tr}(QUXU^t-TX))\phi(\iota^{-1} (UYU^t) )dX. \nonumber
    % \\
    % =&\ \sum_{U\in\mathrm{GL}_2(\mz)}\phi(\iota^{-1} (UYU^t) )\int_{0\leq x_1,x_2,x_3\leq1}e(\mathrm{Tr}(U^tQU-T)X)dX.
\end{align*}
Using $\mathrm{Tr}(QUXU^t) =\mathrm{Tr}(U^t Q U X)$
and integrating over $x_1,x_2,x_3$, we obtain \eqref{eq:aQ0formula}.
% \begin{align}
%     a_Q^{(0)}(T,Y,\phi)=\sum_{\substack{U\in\mathrm{GL}_2(\mz)\\ UQU^t=T}}\phi(\iota^{-1} U^tYU )=\sum_{\substack{U\in\mathrm{GL}_2(\mz)\\ UQU^t=T}}\phi(U^t\iota^{-1} (Y) ).
% \end{align}

It remains to show the upper bound \eqref{eq:aQ0bound}. 
Recall that $\phi \circ \iota^{-1}(M)$ is supported on matrices $M$ such that $\| M  \|_{\infty} \ll N^{-1}$.
Write $U=(\begin{smallmatrix}
    u_1 & u_2 \\ u_3 & u_4
\end{smallmatrix})$, 
and
$Y=RR^t$ as in \eqref{eq:YvsAcoordinates}.
By a simple modification of Lemma \ref{lemma:CsupportCondition}, the sum in \eqref{eq:aQ0formula} may be restricted to $U \in \GL_2(\mz)$ such that
\begin{gather*}
u_1^2 r_1 \ll \frac{1}{N}, 
\qquad (u_1 u + u_2)^2 r_2 \ll \frac{1}{N}, \\
u_3^2 r_1 \ll \frac{1}{N}, 
\qquad (u_3 u + u_4)^2 r_2 \ll \frac{1}{N}.
\end{gather*}
In particular, $|u_1|, |u_3| \ll N^{-1/2} r_1^{-1/2}$ and since $u \ll 1$ we then deduce that $|u_2|, |u_4| \ll N^{-1/2} (r_1^{-1/2} + r_2^{-1/2})$.  Hence the sum over $U$ may be restricted to matrices $U \in \GL_2(\mz)$ such that $\|U \|_{\infty} \ll N^{-1/2} (r_1^{-1/2} + r_2^{-1/2})$.  The total number of such $U$'s is $\ll N^{-1} (r_1^{-1} + r_2^{-1})^{1+\varepsilon}$.
% Then the support of $\phi$ implies that \begin{align}
%     \left(\begin{smallmatrix}
%         \sqrt{r_1}u_1 & \sqrt{r_2}(u_1u+u_3)\\ \sqrt{r_1}u_2 & \sqrt{r_2}(u_2u+u_4)
%     \end{smallmatrix}\right)=\left(\begin{smallmatrix}
%         O(1/N) & O(1/N) \\ O(1/N) & O(1/N)
%     \end{smallmatrix}\right).
% \end{align}
% Together with $u\ll1$, we have $|u_j|\ll N^{-1}(r_1^{-1}+r_2^{-1})$ for all $j=1,2,3,4$. Noticing that no such $U\in \mathrm{GL}_2(\mathbb{Z})$ exist unless $N^{-1}(r_1^{-1}+r_2^{-1})\gg1$, we obtain the desired bound.
\end{proof}

\section{Proof of Theorem \ref{thm:mainthm}}
We return to \eqref{eq:FGPQinnerproduct}, set $\sigma_1 = 350$ and $\sigma_2 = 200$ (which is in the region \eqref{eq:minimalEisensteinSeriesAbsoluteConvergence}), and apply the triangle inequality.  Integrating trivially over the $\nu_1$ and $\nu_2$ variables gives
\begin{multline}
\label{eq:FGPQinnerproductAbsoluteValue}
|\langle F P_Q(\cdot, \phi), G \rangle|
\ll
\sum_{T + M_1= M_2} |a_F(M_1) \overline{a_G}(M_2) |
\int_{\substack{
r_1, r_2 \in \mr_{>0} \\ u \ll 1}}
\\
|a_{Q}(T, Y, \phi)|
\exp(- 2 \pi \mathrm{Tr}((M_1+M_2)Y))
r_1^{174 + k}
r_2^{23.5 + k}
du \frac{dr_1 dr_2}{r_1 r_2}.
\end{multline}
Note that if $M_1 = (\begin{smallmatrix} m_1 & m_2/2 \\ m_2/2 & m_3 \end{smallmatrix})$ and $M_2 = (\begin{smallmatrix} n_1 & n_2/2 \\ n_2/2 & n_3 \end{smallmatrix})$,  then $\mathrm{Tr}(M_1 Y) = m_1 (r_1 + r_2 u^2) + m_2 r_2 u + m_3 r_2$, and similarly for $\mathrm{Tr}(M_2 Y)$.

{\bf Rank $2$.}
Consider the contribution to \eqref{eq:FGPQinnerproductAbsoluteValue} from $|a_Q^{(2)}(T,Y, \phi)|$, where the bound is given by Theorem \ref{thm.Rank2Bound}.  With this substitution, this contribution is bounded by
\begin{equation*}
N^{5/2+\varepsilon}
    \sum_{M_1, M_2} |a_F(M_1) \overline{a_G}(M_2) |
    (1+|\det(M_2-M_1)|)^{\varepsilon} 
    \cdot I(M_1, M_2),
\end{equation*}
where
\begin{multline*}
    I(M_1, M_2) = 
    \int_{\substack{
r_1, r_2 \in \mr_{>0} \\ u \ll 1}}
\exp(-r_2 (m_1 u^2 + m_2 u + m_3) + (n_1u^2 + n_2 u + n_3))
\\
\exp(- (m_1 + n_1) r_1)
 (\frac{1}{r_1} + \frac{1}{r_2})^{1+\varepsilon}
r_1^{174.75 + k}
r_2^{25.25 + k}
du \frac{dr_1 dr_2}{r_1 r_2}.
\end{multline*}
We claim that
\begin{equation}
\label{eq:IboundClaim}
    I(M_1, M_2) \ll \frac{1}{\|M_1 \|_{\infty}^4 \|M_2\|_{\infty}^4} (\det M_1 \cdot \det M_2)^{-\frac{k}{2}-5}.
\end{equation}
Using \eqref{eq:IboundClaim},  $|\det(M_2 - M_1)|^{\varepsilon} \ll \|M_1\|_{\infty}^{\varepsilon} \cdot \|M_2 \|_{\infty}^{\varepsilon}$, and applying Corollary \ref{coro:RankinSelbergVariant}, will then show that the contribution of $a_Q^{(2)}$ to \eqref{eq:FGPQinnerproductAbsoluteValue} is $\ll N^{5/2+\varepsilon}$, consistent with Theorem \ref{thm:mainthm}.

Now we show the claim \eqref{eq:IboundClaim}.  It is easy to see that
\begin{equation}
\label{eq:Somer1Integral}
    \int_{0}^{\infty} \exp(-(m_1 + n_1) r_1) 
    r_1^{174.75 + k} 
    (\frac{1}{r_1} + \frac{1}{r_2})^{1+\varepsilon}
    \frac{dr_1}{r_1}
    \ll (1 + r_2^{-1-\varepsilon}) (m_1 + n_1)^{-k-173}.
\end{equation}
Using $(x+y)^{-1} \leq x^{-1/2} y^{-1/2}$, valid for $x,y> 0$, then \eqref{eq:Somer1Integral} is $\ll (1+r_2^{-1-\varepsilon}) (m_1 n_1)^{-\frac{k}{2} - 80}$.  Similarly estimating the $r_2$-integral, we obtain the bound
\begin{equation}
\label{eq:IboundMiddleSteps}
    I \ll  
    \sum_{j \in \{0,1\}}
    \int_{u \ll 1} \frac{(m_1 n_1)^{-\frac{k}{2}-80} du}{(m_1 u^2 + m_2 u + m_3)^{\frac{k}{2} + 12+\frac{1+4j(1 + \varepsilon)}{8}} (n_1 u^2 + n_2 u + n_3)^{\frac{k}{2} + 12+\frac{1+4j(1+\varepsilon)}{8}}}.
\end{equation}
Next we claim that
\begin{equation}
\label{eq:quadraticformbound}
    (m_1 u^2 + m_2 u + m_3)^{-1} 
    \ll
    \begin{cases}
        \frac{m_1}{\det(M_1)}, \qquad &\text{always} \\
        m_3^{-1}, \qquad &\text{if } m_1 \leq \delta m_3, \text{ $\delta > 0$ small}.
    \end{cases}
\end{equation}
By completing the square, it is easy to see that
$m_1 u^2 + m_2 u + m_3 \geq \frac{\det(M_1)}{m_1}$, for all $u \in \mr$, giving the first bound.  Next suppose 
$m_1 \leq \delta m_3$ with $\delta > 0$ small compared to the implied constant in $u \ll 1$.  We have 
$$m_1 u^2 + m_2 u + m_3 = m_3\Big(1 + \frac{u m_2}{2m_3}\Big)^2 + \frac{u^2}{m_3} \det(M_1) \geq m_3 ( 1 + \frac{u m_2}{2m_3})^2.$$
Note that $m_2^2 \leq 4 m_1 m_3 \leq 4 \delta m_3^2$, so $\frac{|m_2|}{m_3} \leq 2 \sqrt{\delta}$, which implies that $1+\frac{u m_2}{2m_3} \gg 1$ for $\delta > 0$ sufficiently small.  Hence $m_1 u^2 + m_2 u + m_3 \gg m_3 \asymp \|M_1 \|_{\infty}$.

Now we apply \eqref{eq:quadraticformbound} to \eqref{eq:IboundMiddleSteps} as follows.  If $m_1 \geq \delta m_3$ for some fixed but arbitrarily small $\delta > 0$, then $\|M_1 \|_{\infty} \asymp m_1$, since $\det(M_1) \geq 0$ implies
$m_2^2 \leq 4m_1 m_3 \leq 4 \delta^{-1} m_1^2$.  We then use the first bound in \eqref{eq:quadraticformbound}, giving
\begin{equation*}
    \frac{m_1^{-\frac{k}{2}-80}}{ (m_1 u^2 + m_2 u + m_3)^{\frac{k}{2} + 12+\frac{1+4j(1+\varepsilon)}{8}}} 
    \ll (\det M_1)^{-\frac{k}{2}-12} m_1^{-50} \ll (\det M_1)^{-\frac{k}{2}-12} \| M_1 \|_{\infty}^{-50}.
\end{equation*}
If $m_1 \leq \delta m_3$ then $\|M_1 \|_{\infty} \asymp m_3$. We borrow $m_1 u^2 + m_2 u + m_3)^{-4}$ and use the second bound in \eqref{eq:quadraticformbound} for this part, and use the first bound in \eqref{eq:quadraticformbound} for the remaining factors, giving
\begin{equation*}
     \frac{m_1^{-\frac{k}{2}-80}}{(m_1 u^2 + m_2 u + m_3)^{\frac{k}{2} + 12+\frac{1+4j(1+\varepsilon)}{8}}}
    \ll \frac{1}{m_3^4} \frac{1}{m_{1}^{60}} (\det M_1)^{-\frac{k}{2}-8} 
    \ll (\det{M_1})^{-\frac{k}{2}-8} \|M_1 \|_{\infty}^{-4} .
\end{equation*}
Combining the two cases, and applying the same bounds for $M_2$ as just derived for $M_1$, we then obtain \eqref{eq:IboundClaim}.

{\bf Rank $1$.}  
Consider the contribution to \eqref{eq:FGPQinnerproductAbsoluteValue} from $|a_Q^{(1)}(T,Y, \phi)|$, where the bound is given by Theorem \ref{thm.Rank1Bound}.
Here the bound has a similar shape to the rank $2$ case, except the power of $N$ is only $N^{1+\varepsilon}$ instead of $N^{5/2+\varepsilon}$, and the exponents in terms of the $r_i$ variables are slightly adjusted.
It is not hard to check that the previous method shows that the contribution of the rank $1$ terms to \eqref{eq:FGPQinnerproductAbsoluteValue} is $\ll N$.

{\bf Rank $0$.}
This case is even easier than the rank $1$ terms, and using Proposition \ref{prop.Rank0Poincare}, the contribution to $\eqref{eq:FGPQinnerproductAbsoluteValue}$ is $\ll N^{-1}$.

\bibliographystyle{alpha}
\bibliography{refs}	

\end{document}